\pdfoutput=1 

\RequirePackage{fix-cm}
\documentclass[smallextended]{svjour3}
\usepackage[utf8]{inputenc}
\usepackage{booktabs}
\usepackage{empheq}
\usepackage{amssymb,amsmath,graphics,graphicx,subfig}
\usepackage[title]{appendix}
\usepackage{mathrsfs}
\usepackage{tikz}
\usepackage{multirow}
\definecolor{lightblue}{rgb}{0.22,0.45,0.70}
\definecolor{lightgreen}{rgb}{0.22,0.50,0.25}
\usepackage[right = 2cm, left=2cm, top = 2.5cm, bottom =2.5cm]{geometry}
\usepackage[colorlinks=true,breaklinks=true,linkcolor=lightblue,citecolor=lightblue]{hyperref}

\definecolor{darkred}{rgb}{0.82,0.15,0.20}
\definecolor{darkblue}{rgb}{0.82,0.15,0.12}

\renewenvironment{proof}{\noindent{\it Proof.}}{\hfill$\square$}

\numberwithin{equation}{section}
\numberwithin{figure}{section}
\numberwithin{table}{section}
\numberwithin{lemma}{section}
\numberwithin{corollary}{section}
\numberwithin{theorem}{section}
\numberwithin{remark}{section}

\newcommand{\vertiii}[1]{{\left\vert\kern-0.25ex\left\vert\kern-0.25ex\left\vert #1 
		\right\vert\kern-0.25ex\right\vert\kern-0.25ex\right\vert}}

\allowdisplaybreaks

\begin{document}
	\titlerunning{Stabilized VEM for coupled Stokes-Temperature equation}   
	\title{An equal-order virtual element framework for the coupled Stokes-Temperature  equation with nonlinear viscosity}    
	\authorrunning{Mishra, Natarajan}
	\author{Sudheer Mishra \and Natarajan E }
	\institute{
		Natarajan E \at
		Department of Mathematics, Indian Institute of Space Science and Technology, Thiruvananthapuram, 695547, India \\
		\email{thanndavam@iist.ac.in}.\\
		Sudheer Mishra \at Department of Mathematics, Indian Institute of Space Science and Technology, Thiruvananthapuram, 695547, India\\
		\email{sudheermishra.20@res.iist.ac.in}.
		\and
	}
	\date{}
	\maketitle
	
\begin{abstract}
	In this work, we present and analyze a novel stabilized virtual element formulation for the coupled Stokes-Temperature equation on polygonal meshes, employing equal-order element pairs where viscosity depends on temperature. The main objective of the proposed virtual elements is to develop a stabilized virtual element problem that avoids higher-order derivative terms and bilinear forms involving velocity, pressure and temperature, thereby avoiding the coupling between virtual element pairs. Moreover, it also reduces the violation of divergence-free constraints and offers reasonable control over the gradient of temperature. We derive the stability of the continuous solution using the Banach fixed-point theorem under sufficiently small data. The stabilized coupled virtual element problem is formulated using the local projection-based stabilization methods. We demonstrate the existence and uniqueness of the stabilized  discrete solution using the Brouwer fixed-point theorem and the contraction theorem under the assumption of sufficient small data by showing the well-posedness of the stabilized decoupled virtual element problems. Furthermore, we derive the error estimates with optimal convergence rates in the energy norms. We present several numerical examples to confirm the theoretical findings. Additionally, the numerical behavior of the proposed stabilized method is shown to be robust with respect to linear and non-linear thermal conductivity.
	
\end{abstract}

	\keywords{Stokes-Temperature equation  \and Stabilized virtual elements\and Inf-sup condition \and Fixed-point theorem \and General polygons.}
	\subclass{ 65N12, 65N30, 76D07}

	\section{Introduction}
The coupling of a viscous incompressible flow with a temperature transport equation where the dynamic viscosity depends on temperature arises in many physical applications, such as the thermal structure of subduction zones, natural and thermal convection, climate predictions, heat transfer in nanoparticle fluids, aluminum production, nuclear power plants, chemical distillation process, oceanic flows, motion of bio-membranes and so on. Due to its significance, the coupled thermo-fluid dynamics problem has garnered considerable attention from physicists, geologists, mathematicians, and engineers over the past decades. Consequently, several numerical techniques have been developed to compute the distribution of the temperature field \cite{mfem44,mfem45,mfem46,mfem49,mfem48}.

The coupled Stokes-Temperature equation can be viewed as a subproblem derived from the generalized Boussinesq problem, which models the motion of incompressible fluid flow governed by the Stokes equations coupled with the transport heat equation, where the convective term is driven by the fluid velocity. Notable work related to the Boussinesq equation includes \cite{mfem51,mfem52,mfem53,mvem3},  among others. The mixed formulation of this coupling also appears in the sedimentation-consolidation process of particles, which has been extensively studied in \cite{mfem39,mfem56,mfem57}. In \cite{mfem55}, the authors investigated the finite element approximation of the coupled Stokes-Heat equations under nonlinear slip boundary conditions, employing stable finite element pairs. More recently, Araya et al. analyzed the same coupling for Stokes-Temperature equations using the equal-order finite element approximations in \cite{mfem43}.

Most of the works mentioned above, notably \cite{mfem39,mfem56,mfem57,mfem55}, rely on stable primal/mixed finite element approximations, where the finite element pairs are suitably chosen to satisfy the discrete inf-sup condition. Violation of this condition leads to the appearance of spurious oscillations or moving fronts in the solution domain, a common issue in unified (equal-order) finite/virtual element frameworks. Additionally, convection-dominated regimes also tend to produce such oscillations. To address these instabilities, several stabilization methods have been analyzed, including the Galerkin-least-square (GLS) \cite{mfem28}, streamline upwind/Petrov-Galerkin (SUPG) \cite{mfem29,mfem30}, continuous interior penalty (CIP) \cite{mfem32}, and local projection stabilization (LPS) \cite{mfem33,mfem31}.

The Virtual Element Method (VEM) was introduced in \cite{vem1,vem25}, representing a modern advancement in the discretization of partial differential equations. It extends the traditional Finite Element Method (FEM) by incorporating more general polygonal and polyhedral meshes, including non-convex meshes with hanging nodes and high distortion. Due to its flexibility, VEM has received much attention from engineering communities. Here we cite some novel works in continuum mechanics \cite{mvem17,mvem18,mvem19,mvem20}.
The VEM literature is extensive and comprehensive, addressing a wide range of problems, including elliptic, parabolic, and hyperbolic equations \cite{vem40,vem8,vem32,vem33}, stabilized VEM for convection-diffusion problems \cite{vem4,vem04,vem08}, and Cahn-Hillard \cite{vem11}. Finally, we address some remarkable contributions in fluid dynamics for the Stokes or Navier-Stokes problem \cite{mvem9,mvem11,mvem12,mvem3,mvem2,mvem16}.

In this work, we investigate a stabilized conforming Virtual Element Method (VEM) for the coupled Stokes-Temperature problem with nonlinear viscosity, utilizing equal-order virtual element pairs. The stabilized VEM is derived using local projection stabilization techniques as studied in \cite{mfem31,vem28m,mvem21}. In \cite{mvem3}, the authors proposed a VEM formulation for the Navier-Stokes equation coupled with the heat equation, employing stable virtual element pairs where the discrete velocity satisfies the divergence-free constraint. Recently, a least-squares based stabilized FEM has been discussed in \cite{mfem43}, focusing on diffusion-dominated temperature transport regimes. Their study stabilizes the entire system except for the continuity equation. However, this method incorporates higher-order derivative terms and multiple forms involving more than one finite element, which can increase the coupling between finite element pairs. Furthermore, we emphasize that the cross-terms $\alpha_1 \left( \nabla p_h,  \nabla \cdot \tau_h \right)_K + \alpha_1 \left( -  \nabla \cdot \sigma_h, \frac{1}{8\mu(\phi)}\nabla q_h \right)_K$ can become more problematic in the presence of high derivatives of pressure and stress, as pointed out by R. Codina in \cite{mfem58}. In contrast, we propose a more concise stabilization technique utilizing suitable projectors that provide separate stabilization terms without relying on higher-order derivative terms.
Due to the unified framework, the classical coupled virtual element problem may suffer from violations of the discrete inf-sup condition, violations of the divergence-free constraints, and the presence of convection-dominated transport regimes. These issues are addressed in the stabilized coupled virtual element problem. The theoretical analysis of the stabilized virtual element problem is primarily focused on diffusion-dominated transport regimes, similar to \cite{mvem3,mfem39,mfem43}. In addition, the numerical behavior of the proposed method demonstrates its effectiveness in addressing convection-dominated transport regimes. 

We establish the stability of the continuous problem under sufficiently small data by applying the fixed-point and contraction theorems. For the discrete stabilized problem, we first establish the well-posedness of the stabilized decoupled problems followed by deriving the well-posedness of the coupled problem under sufficiently small data. We show the optimal convergence rates in the energy norms. Our numerical findings validate the expected convergence rates in both diffusion-dominated and convection-dominated transport regimes. The proposed VEM effectively reduces violations of the divergence-free condition, as confirmed by simulations of realistic problems. Additionally, the numerical behavior of the stabilized VEM remains
robust to temperature-dependent thermal conductivity.

\subsection{Notations}
 We use the following notations throughout the work. Let $\Omega$ be a bounded open domain in $\mathbb{R}^2$ with Lipschitz boundary $\partial \Omega$.  We assume $\mathcal{D} \subseteq \Omega$ represents any measurable set. The Sobolev space $W^k_p(\mathcal{D})$ is equipped with the standard norm $\|\cdot\|_{k,p,\mathcal{D}}$ and its semi-norms $|\cdot|_{k,\mathcal{D}}$. For $k=0$, it reduces to the space $L^2(\mathcal{D})$, equipped with the usual $L^2$-inner product $(\cdot, \cdot)_{0,\mathcal{D}}$ and the $L^2$-norm $\|\cdot\|_{0,\mathcal{D}}$. {Often the subscript $\mathcal{D}$ will be omitted in the $L^2$-inner products when no confusion arises.}  We use bold fonts or symbols to denote vector-valued functions/tensors. The Poincar$\acute{\text{e}}$ constant depends on $\Omega$, is denoted by $C_P$. Additionally, we utilize the embedding constant $C_q>0$, which satisfies the compact embedding $W^1_2(\Omega) \hookrightarrow L^q(\Omega)$ for $q \geq 1$ such that
\begin{align*}
	\| \omega \|_{0,q,\Omega} \leq C_q \|\omega\|_{1,\Omega} \qquad \text{for all} \,\, \omega \in W^1_2(\Omega).
\end{align*}  
For $p=2$, the Sobolev space $W^k_p(\mathcal{D})$ reduces to $H^k(\mathcal{D})$.

\subsection{Structure of the paper}
The remaining structure of the present work is organized as follows. In Section \ref{sec:2}, we introduce the coupled Stokes-Temperature equations with temperature-dependent viscosity and establish the well-posedness of the primal formulation of the coupled problem. Section \ref{sec:3} delves into the fundamentals of virtual elements and derives the stabilized virtual element formulation. Section \ref{sec:4} demonstrates the well-posedness of the stabilized virtual element problem. Section \ref{sec:5} is dedicated to deriving optimal error estimates using the energy norms. The numerical results obtained with the proposed VEM are presented in Section \ref{sec:6}, and finally, Section \ref{sec:7} presents concluding remarks.

\section{The coupled Stokes-Temperature equation } \label{sec:2}
We consider the following coupled Stokes-Temperature equation on general polygonal domain $\Omega \subset \mathbb R^2$ as follows: 
\begin{empheq}[left=(S) \empheqlbrace]{align}
		\text{Find}\,(\sigma, \mathbf{u}, p, \phi)\,\, \text{such that}, \\
		\sigma &=  \mu(\phi) \varepsilon (\mathbf{u})  \quad \text{in} \quad \Omega, \label{stress-1} \\
		\quad - \nabla \cdot \sigma + \nabla p &= \alpha \phi \mathbf{f}  \quad \text{in} \quad \Omega, \label{stoke-1}\\
		\nabla \cdot \mathbf{u} &= 0 \quad  \text{in} \quad \Omega, \label{stoke-2}\\
		\mathbf{u} &= \widehat{\mathbf{f}}  \quad \text{on} \quad  \partial \Omega, \label{stoke-3}\\
		-\kappa \Delta \phi + \mathbf{u} \cdot \nabla \phi &= g  \quad \text{in} \quad \Omega, \label{heat-1}\\
		\phi &= \widehat{g} \quad  \text{on} \quad \partial \Omega, \label{heat-2} 
\end{empheq}
where $\sigma$ is the stress tensor and $\varepsilon (\mathbf{u})$ is the deformation rate tensor. The effective dynamic viscosity is defined by $\mu : \mathbb{R} \rightarrow \mathbb{R}^{+}$ and the thermal conductivity is given by $\kappa >0$. The velocity vector field is denoted by $\mathbf{u}$, the scalar pressure field by $p$ and the scalar temperature field by $\phi$. We assume the load terms $\mathbf{f} \in [L^{2}(\Omega)]^2$ and $g \in L^2(\Omega)$. {Furthermore, the positive constant $\alpha$ represents the thermal Rayleigh number.}

For the sake of simplicity, we combine Eqs. \eqref{stress-1} and \eqref{stoke-1}, and we call the obtained expression as the generalized Stokes equation, given as follows
\begin{align}
	- \nabla \cdot ( \mu(\phi) \varepsilon (\mathbf{u})) + \nabla p = \alpha \phi \mathbf{f}  \qquad \text{in}\,\,\,\, \Omega.
\end{align}
Furthermore, we assume the following:
\begin{align}
	\widehat{\mathbf{f}} = \mathbf 0, \qquad  \widehat{g}=0.	\label{homo9}
\end{align}
In Eqn. \eqref{homo9}, we consider the homogeneous Dirichlet boundary conditions for simplicity. However, non-homogeneous conditions can also be addressed. For example, in \cite{mvem3}, the authors introduce non-homogeneous Dirichlet condition for the temperature field. Mimicking this approach, the present work can be extended to the non-homogeneous case by appropriately modifying the spaces for the  velocity and temperature fields.

Additionally, let us assume that there exist positive constants $\mu_{min}$, $\mu_{max}$ and $L_\mu$, such that
\begin{align}
	\mu_{min} \leq \mu(\xi) \leq \mu_{max} \quad \text{for all}\,\, \xi \in \mathbb R,  \qquad  \quad  | \mu(\xi_1) - \mu(\xi_2)| \leq L_\mu | \xi_1 - \xi_2 |, \text{ for all }\,\, \xi_1, \xi_2 \in \mathbb R. \label{mu-reg}
\end{align}
We now introduce the following spaces for the coupled Stokes-Temperature equation: 
\begin{center}
	$\mathbf{V} = [H_0^1(\Omega)]^2$,  \qquad  $Q = L_0^2(\Omega) = \big\{ q \in L^2(\Omega): \int_{\Omega} q d \Omega = 0 \big\}$,   \qquad $\Sigma=H_0^1(\Omega)$, \qquad\qquad \qquad\qquad \qquad \qquad 
\end{center}
endowed with their natural norms. Additionally, we define the following settings:
\begin{itemize}
	\item $a_V(\cdot; \cdot,\cdot): \Sigma \times \mathbf{V} \times \mathbf{V}  \rightarrow \mathbb R$, \qquad $a_V(\phi;\mathbf{v},\mathbf{w}):= \int_{\Omega} \mu(\phi) \varepsilon (\mathbf{v}): \varepsilon (\mathbf{w})\, d\Omega$;
	\item $b(\cdot,\cdot):\mathbf{V} \times Q  \rightarrow \mathbb R$, \qquad \qquad \qquad\quad $b(\mathbf{v},q):= \int_{\Omega} (\nabla \cdot \mathbf{v}) q\, d\Omega$;  
	\item $a_T(\cdot,\cdot):\Sigma \times \Sigma \rightarrow \mathbb R$, \qquad \qquad \quad \,\, $a_T(\phi,\psi):= \int_{\Omega} \kappa \nabla \phi \cdot \nabla \psi \, d\Omega$;
	\item $c_T(\cdot; \cdot,\cdot): \mathbf{V} \times \Sigma \times \Sigma  \rightarrow \mathbb R$, \qquad \, $c_T(\mathbf{v}; \phi,\psi):= \int_{\Omega} (\mathbf{v} \cdot \nabla \phi) \psi \, d\Omega$;
	\item For given $\phi \in \Sigma$, define $F_\phi: \mathbf{V} \rightarrow \mathbb R$, \, $F_\phi(\mathbf{v}):= \int_{\Omega} (\alpha \mathbf{f}\phi)\cdot \mathbf{v} \, d\Omega$;
	\item $G: \Sigma \rightarrow \mathbb R$, \qquad \qquad \qquad \qquad \qquad \quad  $G(\psi):= \int_{\Omega} g \psi \, d\Omega$.
\end{itemize}
We further define the following forms:
\begin{align*}
	A_V[\phi; (\mathbf{v},q), (\mathbf{w}, r)]&:= a_V(\phi;\mathbf{v},\mathbf{w}) - b(\mathbf{w},q) + b(\mathbf{v},r) \qquad \, \text{for all} \,\, \phi \in \Sigma \,\, \text{and} \,\, (\mathbf{v},q), (\mathbf{w}, r) \in \mathbf{V} \times Q, \\
	A_T ( \mathbf{v}; \phi, \psi )&:= a_T(\phi,\psi) +  c^S_T(\mathbf{v}; \phi,\psi) \qquad  \qquad  \quad \text{ for all} \,\, \mathbf{v} \in \mathbf{V} \,\, \text{and}\,\, \phi, \psi \in \Sigma,
\end{align*}
where $c^S_T(\mathbf{v}; \phi,\psi)= \frac{1}{2} \big[c_T(\mathbf{v}; \phi,\psi) - c_T(\mathbf{v}; \psi, \phi)\big]$ represents the skew-symmetric part of $c_T$. 

Therefore, the weak formulation for the coupled Stokes-Temperature equation $(S)$ is readily as follows 
\begin{align}
	\begin{cases}
		\text{Find}\,\, (\mathbf{u}, p, \phi) \in \mathbf{V} \times Q \times \Sigma, \,\, \text{such that}\\
		A_V[ \phi; (\mathbf{u}, p), (\mathbf{v}, q)] + A_T(\mathbf{u}; \phi, \psi) = F_\phi(\mathbf{v}) + G(\psi) \qquad \text{for all}\,\, (\mathbf{v}, q, \psi) \in \mathbf{V} \times Q \times \Sigma.
	\end{cases}
	\label{variation-1}
\end{align}

The boundary $\partial \Omega$ is Lipschitz and parameters $\mu>0$ and $\kappa>0$ are continuous, and $g \in L^2(\Omega)$. Therefore, the existence of a solution for problem \eqref{variation-1} follows from \cite[Theorem 2.1]{mfem0}. We emphasize that \cite[Theorem 2.1]{mfem0} shows the existence of a weak solution for continuous problem \eqref{variation-1} when load term $g=0$, and the convective term vanishes in \cite{mfem0}. By making a slight modification to the arguments in \cite{mfem0}, we can extend this result to show the existence of weak solutions for \eqref{variation-1} when $g\neq 0$.

In the following section we introduce the fixed-point formulation equivalent to the problem \eqref{variation-1} to establish the uniqueness of its solution.

\subsection{The fixed-point problem} \label{cnt_fixed} 
We now aim to achieve the fixed-point problem of the variational problem \eqref{variation-1}. To do this, we will first decouple the problem \eqref{variation-1} into the primal formulation of the momentum equation (generalized Stokes problem) and the primal formulation of the temperature (advection-diffusion) equation. Then we formulate the fixed-point problem using these decoupled problems. In this sequel, we define the operator $\mathbf{S}: \Sigma \rightarrow \mathbf{V} \times Q$ such that
\begin{align}
	\widehat{\phi} \rightarrow \mathbf{S}(\widehat{\phi}) = (S^1(\widehat{\phi}), S^2(\widehat{\phi})) := (\widehat{\mathbf{u}}, \widehat{p}), \label{opt_S}
\end{align}
where $(\widehat{\mathbf{u}}, \widehat{p}) \in \mathbf{V} \times Q$ denotes the solution of the following problem:
\begin{align}
	\begin{cases}
		\text{For given $ \widehat{\phi} \in \Sigma$, seek}\,\, (\widehat{\mathbf{u}}, \widehat{p}) \in \mathbf{V} \times Q, \,\, \text{such that}\\
		A_V[ \widehat{\phi}; (\widehat{\mathbf{u}}, \widehat{p}), (\mathbf{v}, q)]  = F_{\widehat{\phi}}(\mathbf{v})  \qquad \text{for all}\,\, (\mathbf{v}, q) \in \mathbf{V} \times Q.
	\end{cases}
	\label{dvariation-1}
\end{align}
Additionally, we introduce $\mathbf{V}_{div}:= \{ \mathbf{v} \in \mathbf{V} \,\, \text{such that} \,\, \nabla \cdot \mathbf{v}=0\}$. Therefore, we define the operator $\mathbb{M} : \mathbf{V}_{div} \rightarrow \Sigma$ corresponding to the temperature equation by, 
\begin{align}
	\widehat{\mathbf{u}} \rightarrow \mathbb{M}(\widehat{\mathbf{u}}) = \widehat{{\phi}}, \label{opt_T}
\end{align}
where $\widehat{\phi} \in \Sigma$ is the solution of the following problem:
\begin{align}
	\begin{cases}
		\text{For given} \,\, \widehat{\mathbf{u}} \in \mathbf{V}_{div}, \,\, \text{seek} \,\, \widehat{\phi} \in \Sigma, \,\, \text{such that}\\
		A_T(\widehat{\mathbf{u}}; \widehat{\phi}, \psi) =  G(\psi) \qquad \text{for all}\,\,  \psi \in  \Sigma.
	\end{cases}
	\label{dvariation-2}
\end{align}
Finally, we define $\mathbb{T}: \Sigma \rightarrow \Sigma$ by $\mathbb{T}(\widehat{\phi}):= \mathbb{M}(S^1(\widehat{\phi}))$, for all $\widehat{\phi} \in \Sigma$. Then, the fixed-point problem can be defined as follows:
\begin{align}
	\begin{cases}
		\text{Find} \,\, {\phi} \in \Sigma, \,\, \text{such that}\\
		\mathbb{T}({\phi}) = \phi.
	\end{cases}
	\label{fixed}
\end{align}
Notably, the fixed-point problem \eqref{fixed} is equivalent to the problem \eqref{variation-1}. Therefore, we establish the stability of the problem \eqref{fixed} using the Banach fixed-point theorem following \cite{mfem39}.

\subsection{Well-posedness of the decoupled problems} 
 We now establish the well-posedness of the decoupled problems \eqref{dvariation-1} and \eqref{dvariation-2}. To do this, we define the norms over the spaces  $\mathbf{V} \times Q$ and $\Sigma$ as follows:
\begin{align}
	\vertiii{(\mathbf{v}, q)}^2 &:= \mu_{min} \|\nabla \mathbf{v}\|^2_{0,\Omega} + \|q\|^2_{0,\Omega} \qquad \,\, \text{for all} \,\, (\mathbf{v}, q) \in \mathbf{V} \times Q, \\
	\vertiii{\phi}_{\Sigma}&:= \sqrt{\kappa} \|\nabla \phi\|_{0,\Omega} \qquad \,\, \text{for all} \,\, \phi \in \Sigma.
\end{align} 

\begin{remark} \label{cinfsup}
	For given $\phi \in \Sigma$, combining the continuity and coercivity of $a_V$ with the continuity and inf-sup condition of the form $b$, it is evident that for any $(\mathbf{u}, p) \in \mathbf{V} \times Q$
	\begin{align}
		(\mathbf{v}, q) \neq (\mathbf{0},0), \quad \sup \limits_{(\mathbf{v}, q) \in \mathbf{V} \times Q} \dfrac{A_{V} [ \phi; (\mathbf{u}, p), (\mathbf{v}, q)]}{\vertiii{(\mathbf{v}, q)}} \gtrsim {\beta}_0 \vertiii{(\mathbf{u}, p)}, \label{cwellpos-0}
	\end{align}
	{where the constant $\beta_0>0$ depends on $\mu$ and is the inf-sup constant associated with the form $A_{V}[\cdot;\cdot,\cdot]$.}
\end{remark}

\begin{lemma} \label{d1_unique}
	For given $\widehat{{\phi}} \in \Sigma$, the decoupled problem \eqref{dvariation-1} has a unique solution $(\widehat{\mathbf{u}}, \widehat{p}) \in \mathbf{V} \times Q$ such that it holds 
	\begin{align}
		\vertiii{\mathbf{S}(\widehat{{\phi}})} := \vertiii{(\widehat{\mathbf{u}}, \widehat{{p}})} \lesssim \dfrac{ \alpha C_q^2 (1+C_P)^2}{\beta_0 \sqrt{ \kappa \mu_{min}}} \|\mathbf{f}\|_{0, \Omega} \vertiii{\widehat{\phi}}_{\Sigma}. \label{dc1}
	\end{align}
\end{lemma}
\begin{proof}
	Since the form $a_V(\cdot; \cdot, \cdot)$ is continuous and coercive, and $b(\cdot,\cdot)$ satisfies the inf-sup condition  \cite{mfem22}, it follows that problem \eqref{dvariation-1} has a unique solution $(\widehat{\mathbf{u}}, \widehat{p}) \in \mathbf{V} \times Q$. Employing Remark \ref{cinfsup} and the Sobolev embedding theorem, we infer
	\begin{align}
		\beta_{0} \vertiii{(\widehat{\mathbf{u}}, \widehat{{p}})}^2 & \leq  A_V[\phi; (\widehat{\mathbf{u}}, \widehat{{p}}), (\widehat{\mathbf{u}}, \widehat{{p}})] = (\alpha \widehat{\phi} \mathbf{f}, \widehat{\mathbf{u}} ) \nonumber \\
		& \lesssim  \alpha C_q^2 (1+C_P)^2\|\mathbf{f}\|_{0, \Omega} \|\nabla \widehat{{\phi}}\|_{0,\Omega} \|\nabla \widehat{\mathbf{u}} \|_{0,\Omega},
	\end{align} 
	which readily implies \eqref{dc1}.
\end{proof}

\begin{lemma} \label{d2_unique}
	For any given $\widehat{\mathbf{u}} \in \mathbf{V}_{div}$, the decoupled problem \eqref{dvariation-2} has a unique solution $\widehat\phi \in \Sigma$ such that it satisfies
	\begin{align}
		\vertiii{\mathbb{M}(\widehat{\mathbf{u}})}_{\Sigma} :=\vertiii{\widehat{{\phi}}}_{\Sigma} \lesssim \dfrac{C_P}{ \sqrt{\kappa}} \|g\|_{0,\Omega}.
	\end{align}
\end{lemma}
\begin{proof}
	Since $A_T$ is continuous and coercive with coercivity constant $1$, the uniqueness of the solution is evident by employing the Lax-Milgram Lemma. Further using the coercivity property, we obtain
	\begin{align}
		\vertiii{\widehat{{\phi}}}_{\Sigma}^2 \leq A_T(\widehat{\mathbf{u}}; \widehat{\phi}, \widehat{{\phi}}) \lesssim C_P \|g\|_{0,\Omega} \|\nabla \widehat{\phi}\|_{0,\Omega}, \label{ATC}
	\end{align} 
	which completes the proof. 
\end{proof}

\subsection{Stability of the coupled problem} 
In this section we focus on proving the uniqueness of the continuous solution. To this end, we first derive several auxiliary results that are crucial for demonstrating the stability of the problem \eqref{variation-1}.
\begin{lemma} \label{s-diff}
	For any given ${\phi}, \widehat{{\phi}} \in \Sigma$, we have the following estimate
	\begin{align}
		\vertiii{\mathbf{S}(\phi) - \mathbf{S}(\widehat{{\phi}})} \lesssim \dfrac{ C_q(1+C_P)} { \beta_0 \sqrt{ \kappa \mu_{min}}}  \big[ L_\mu \|\nabla S^1(\widehat{{\phi}})\|_{0,3,\Omega} + \alpha C_q(1+C_P) \|\mathbf{f}\|_{0, \Omega} \big] \vertiii{\phi - \widehat{{\phi}}}_{\Sigma}. \label{sdiff}
	\end{align}
\end{lemma}
\begin{proof}
	For any given $\phi, \widehat{{\phi}} \in \Sigma$, we define $\mathbf{S}(\phi)=(S^1(\phi), S^2(\phi)):=(\mathbf{u},p)$ and $\mathbf{S}(\widehat{{\phi}})= (S^1(\widehat{{\phi}}), S^2(\widehat{{\phi}})) :=(\widehat{\mathbf{u}}, \widehat{{p}})$. Recalling Eqn. \eqref{dvariation-1}, and by adding and subtracting suitable terms, we infer 
	\begin{align}
		A_V[\phi; (\mathbf{u},p)- (\widehat{\mathbf{u}}, \widehat{{p}}), (\mathbf{v}, q)] :&= \big( \alpha \mathbf{f}\phi, \mathbf{v}\big) - \big(\alpha \mathbf{f} \widehat{{\phi}}, \mathbf{v} \big) + A_V[\widehat{{\phi}}; (\widehat{\mathbf{u}}, \widehat{{p}}), (\mathbf{v}, q)] - A_V[\phi; (\widehat{\mathbf{u}}, \widehat{{p}}), (\mathbf{v}, q)] \nonumber \\
		&=: A_{V,1} + A_{V,2}. \label{Sdiff-1}  		
	\end{align}
	Applying the fact $\mathbf{f} \in [L^{2}(\Omega)]^2$, the Sobolev embedding theorem and the Poincar$\acute{\text{e}}$ inequality, we have
	\begin{align}
		|A_{V,1}| &\lesssim \alpha \|\mathbf{f}\|_{0,\Omega} \|\phi - \widehat{{\phi}} \|_{0,4,\Omega} \|\mathbf{v}\|_{0,4,\Omega}  \nonumber \\
		& \lesssim \alpha C_q^2 (1+C_P)^2 \|\mathbf{f}\|_{0, \Omega} \| \nabla(\phi - \widehat{{\phi}} )\|_{0,\Omega} \|\nabla \mathbf{v}\|_{0,\Omega}  \nonumber \\
		& \lesssim \dfrac{ \alpha C_q^2 (1+C_P)^2}{\sqrt{\mu_{min}}} \|\mathbf{f}\|_{0, \Omega} \| \nabla(\phi - \widehat{{\phi}} )\|_{0,\Omega} \vertiii{(\mathbf{v}, q)}. 
		\label{sdiff-2}
	\end{align} 
	Employing the Lipschitz continuity of $\mu$ and the Sobolev embedding theorem $H^1(\Omega) \hookrightarrow L^6(\Omega)$, yields 
	\begin{align}
		|A_{V,2}| &= a_V(\widehat{{\phi}}; \widehat{\mathbf{u}}, \mathbf{v}) - a_V(\phi; \widehat{\mathbf{u}}, \mathbf{v}) = \big(( \mu(\widehat{{\phi}}) - \mu(\phi)) \nabla \widehat{\mathbf{u}}, \nabla \mathbf{v} \big) \nonumber \\
		&\lesssim L_\mu  \| \phi - \widehat{{\phi}}\|_{0,6, \Omega} \|\nabla \widehat{\mathbf{u}}\|_{0,3, \Omega} \|\nabla \mathbf{v}\|_{0,\Omega} \nonumber \\
		& \lesssim (1+C_P)C_q L_\mu  \|\nabla (\phi - \widehat{{\phi}})\|_{0,\Omega} \|\nabla \widehat{\mathbf{u}}\|_{0,3, \Omega} \|\nabla \mathbf{v}\|_{0,\Omega} \nonumber \\
		& \lesssim \dfrac{ (1+C_P)C_q L_\mu}{\sqrt{\mu_{min}}}  \|\nabla (\phi - \widehat{{\phi}})\|_{0,\Omega} \|\nabla \widehat{\mathbf{u}}\|_{0,3, \Omega} \vertiii{(\mathbf{v},q)}. \label{sdiff-3}
	\end{align}
	Using the inf-sup condition for $A_V[\cdot; \cdot,\cdot]$ and the estimates \eqref{sdiff-2} and \eqref{sdiff-3}, we arrive at
	\begin{align}
		\beta_0 \vertiii{(\mathbf{u}- \widehat{\mathbf{u}}, p - \widehat{{p}})} & \leq \sup \limits_{(\mathbf{v}, q) \in \mathbf{V} \times Q} \dfrac{A_V[\phi; (\mathbf{u},p)- (\widehat{\mathbf{u}}, \widehat{{p}}), (\mathbf{v}, q)]}{\vertiii{(\mathbf{v}, q)}} \nonumber \\
		& \lesssim \dfrac{ (1+C_P)C_q }{\sqrt{\mu_{min}}} \big[ L_\mu  \|\nabla \widehat{\mathbf{u}}\|_{0,3, \Omega} + \alpha C_q (1+C_P) \|\mathbf{f}\|_{0, \Omega} \big]\| \nabla(\phi - \widehat{{\phi}} )\|_{0,\Omega}, \nonumber
	\end{align}
	thus, we can easily obtain the estimate \eqref{sdiff}.   
\end{proof}

\begin{lemma} \label{m-diff}
	For any given $ \mathbf{u}, \widehat{\mathbf{u}} \in \mathbf{V}_{div}$, the following result holds:
	\begin{align}
		\vertiii{ \mathbb M (\mathbf{u}) - \mathbb M (\widehat{\mathbf{u}})}_{\Sigma} \lesssim \dfrac{(1+C_P)^2 C_P C_q^2}{\kappa^{3/2}} \| g\|_{0,\Omega} \|\nabla(\mathbf{u}- \widehat{\mathbf{u}})\|_{0,\Omega}. \label{Mdiff}
	\end{align}
\end{lemma}
\begin{proof}
	For given $ \mathbf{u}, \widehat{\mathbf{u}} \in \mathbf{V}_{div}$, we define $\phi = \mathbb M(\mathbf{u})$ and $\widehat{{\phi}} = \mathbb M (\widehat{\mathbf{u}})$. We set $\eta_\phi= \phi -\widehat{{\phi}}$, and proceed as follows
	\begin{align}
		A_T (\mathbf{u}; \eta_\phi, \eta_\phi) &= A_T(\mathbf{u}; \phi, \eta_\phi) - A_T(\mathbf{u}; \widehat{{\phi}}, \eta_\phi) = (g, \eta_\phi) - A_T(\mathbf{u}; \widehat{{\phi}}, \eta_\phi)= A_T(\widehat{\mathbf{u}}; \widehat{{\phi}}, \eta_\phi) - A_T(\mathbf{u}; \widehat{{\phi}}, \eta_\phi) \nonumber \\
		&=c^S_T(\widehat{\mathbf{u}}; \widehat{{\phi}}, \eta_\phi) - c^S_T(\mathbf{u}; \widehat{{\phi}}, \eta_\phi)= \frac{1}{2}\big[\big( ( \widehat{\mathbf{u}} - \mathbf{u}) \cdot \nabla \widehat{{\phi}}, \eta_\phi \big)- \big( ( \widehat{\mathbf{u}} - \mathbf{u}) \cdot \nabla \eta_\phi, \widehat{\phi} \,\big) \big] \nonumber \\
		&\lesssim  \|\mathbf{u}- \widehat{\mathbf{u}}\|_{0,4, \Omega} \|\nabla \widehat{{\phi}}\|_{0,\Omega} \|\eta_\phi\|_{0,4, \Omega} +  \|\mathbf{u}- \widehat{\mathbf{u}}\|_{0,4, \Omega}  \|\nabla \eta_\phi\|_{0, \Omega} \| \widehat{{\phi}}\|_{0,4,\Omega}. \nonumber
		\intertext{Using the Sobolev embedding theorem $H^1(\Omega) \hookrightarrow L^4(\Omega)$, it holds that}
		|A_T (\mathbf{u}; \eta_\phi, \eta_\phi)| & \lesssim (1+C_P)^2 C_q^2 \|\nabla(\mathbf{u}- \widehat{\mathbf{u}})\|_{0,\Omega} \|\nabla \widehat{{\phi}}\|_{0,\Omega} \| \nabla \eta_\phi\|_{0,\Omega} \nonumber. 
	\end{align}
	Employing the coercivity of $A_T$ and Lemma \ref{d2_unique}, we obtain
	\begin{align}
		\vertiii{\eta_\phi}_{\Sigma} &\lesssim  \dfrac{(1+C_P)^2 C_q^2}{ \sqrt{\kappa}} \|\nabla(\mathbf{u}- \widehat{\mathbf{u}})\|_{0,\Omega} \|\nabla \widehat{{\phi}}\|_{0,\Omega}  \nonumber \\
		& \lesssim \dfrac{(1+C_P)^2 C_P C_q^2}{\kappa^{3/2}} \| g\|_{0,\Omega} \|\nabla(\mathbf{u}- \widehat{\mathbf{u}})\|_{0,\Omega}, 
	\end{align}
	from the above analysis, the estimate \eqref{Mdiff} is evident. 
\end{proof}

Hereafter, we choose $\rho:=\frac{ C_P }{ \sqrt{\kappa}} \| g\|_{0,\Omega}$,
and we further define $\widehat{\Sigma}:= \big\{ \phi \in \Sigma \,\,\text{such that}\,\, \vertiii{\phi}_\Sigma \leq  \rho \big\} \subset \Sigma$.
\begin{lemma} \label{cnts0}
	On $\widehat{\Sigma}$, we have the following results: \newline
	(i). For any $\phi \in \widehat{\Sigma}$, then  $\mathbb T (\phi) \in \widehat{\Sigma}$. \newline
	(ii). For any $\phi, \widehat{{\phi}} \in \widehat{\Sigma}$, the map $\mathbb T$ satisfies
	\begin{align}
		\vertiii{\mathbb T(\phi) - \mathbb T(\widehat{{\phi}}) }_\Sigma \lesssim  C_{\mathbb T} \rho \big[ L_\mu \|\nabla \widehat{\mathbf{u}}\|_{0,3,\Omega} + \alpha (1+C_P)C_q \|\mathbf{f}\|_{0, \Omega} \big] \vertiii{\phi - \widehat{{\phi}}}_{\Sigma}, \label{stab-0}
	\end{align}
	where the positive constant $C_\mathbb T$ is given by \eqref{cnts-2}.
\end{lemma}
\begin{proof}
	(i). For any $\mathbf{u} \in \mathbf{V}_{div}$, we define $\mathbb M (\mathbf{u}) =\phi \in \widehat{\Sigma}$ denotes the solution of \eqref{dvariation-2}. Applying the definition of $\mathbb T$ and Lemma \ref{d2_unique}, we get
	\begin{align}
		\vertiii{\mathbb T (\phi)}_\Sigma =  \vertiii{\mathbb M (\mathbf{u})}_\Sigma \lesssim \dfrac{C_P}{ \sqrt{\kappa}} \|g\|_{0,\Omega} \leq \rho.
	\end{align}
	(ii). For any $\phi, \widehat{{\phi}} \in \widehat{\Sigma}$, we define $\mathbf{S}(\phi)=(S^1(\phi), S^2(\phi)):= (\mathbf{u},p)$ and $\mathbf{S}(\widehat{{\phi}})=(S^1(\widehat{{\phi}}), S^2(\widehat{{\phi}})):=(\widehat{\mathbf{u}}, \widehat{{p}})$. Employing Lemmas \ref{s-diff} and \ref{m-diff}, we get
	\begin{align}
		\vertiii{ \mathbb T(\phi) - \mathbb T(\widehat{{\phi}}) }_\Sigma &= \vertiii{ \mathbb M (S^1(\phi)) - \mathbb M(S^1(\widehat{{\phi}})) }_\Sigma \nonumber \\ 
		& \lesssim \dfrac{(1+C_P)^2 C_P C_q^2}{\kappa^{3/2} \mu_{min}^{1/2}} \| g\|_{0,\Omega} \vertiii{ \mathbf{S}(\phi)- \mathbf{S}(\widehat{\phi})} \nonumber \\
		& \lesssim  \dfrac{ (1+C_P)^3 C_q^3 } {\beta_0 \kappa^{3/2} {\mu_{min}}} \rho \big[  L_\mu \|\nabla S^1(\widehat{{\phi}})\|_{0,3, \Omega} + \alpha (1+C_P)C_q \|\mathbf{f}\|_{0, \Omega} \big] \vertiii{\phi - \widehat{{\phi}}}_{\Sigma},  \label{cnts-2}
	\end{align}
	thus, the proof of \eqref{stab-0} is evident from the above analysis. 
\end{proof}

{We once again recall that the existence of a continuous solution can be derived following \cite{mfem0}, as discussed in Section 2. Therefore,} we are now prepared to establish the uniqueness of the continuous solution using the contraction mapping theorem, stated as follows:
\begin{theorem} \label{unique-1}
	Let $(\mathbf{u},p, \phi) \in \mathbf{V} \times Q \times \widehat{\Sigma}$ be any solution of the problem \eqref{variation-1}. Furthermore, we assume that $\mathbf{u} \in [W^1_3(\Omega)]^2$ and the data are such that the following bound holds
	\begin{align}
		C_{\mathbb T} \rho \big[ L_\mu \|\nabla {\mathbf{u}}\|_{0,3, \Omega} + \alpha (1+C_P)C_q \|\mathbf{f}\|_{0, \Omega} \big] <1, \label{unq0}
	\end{align}
	then the problem \eqref{variation-1} has a unique solution.
\end{theorem}

\begin{proof}
	Applying Lemma \ref{cnts0}, we easily obtain the continuity of the map $\mathbb T$ on ball $\widehat{\Sigma}$. Further the assumption \eqref{unq0} claims that $\mathbb T$ is a contraction map. 
\end{proof}

\section{Stabilized virtual element framework} \label{sec:3}
In this section we present a local projection-based stabilized virtual element formulation of problem (S). We begin with the mesh regularity assumption.

\noindent \textbf{(A1)} \textbf{Mesh assumption.} \label{mesh_reg}
We denote a sequence of decompositions of $\Omega \subset \mathbb R^2$ into non-overlapping simple polygonal element by $\{\Omega_h\}_{h>0}$. Each polygonal element is denoted by $E \in \Omega_h$ and has a finite  diameter $h_E$. The maximum diameter across all polygonal elements in $\{ \Omega_h\}_{h>0}$ is given by $h:= \sup_{E \in \Omega_h} h_E$. Following the standard VEM framework \cite{vem1,vem5}, we assume that there exists a positive constant $\gamma_0$, independent of the mesh size $h$, such that each element $E\in\Omega_h$ satisfies the following mesh regularity assumptions:
\begin{itemize}
	\item each element $E$ is star-shaped with respect to a ball $B_{E}$ of radius  $\geq \gamma_0 h_E$;
	\item each edge $e \in \partial E$ has a length $|e| \geq \gamma_0  h_E$.
\end{itemize}
It is noteworthy that the regularity conditions mentioned above may be relaxed in numerous physical applications, as explored by Beirao da Veiga et al. in \cite{vem44}. 

Let $k \in \mathbb N$ and $\mathbf{m} \in \mathbb N^2$. For any $E \in \Omega_h$, let $|E|$ and $\mathbf{x}_E=(x_{E,1},x_{E,2})$ denote the area of the element $E$ and its centroid, respectively. We further assume that $\mathbb P_k(E)$ represents the space of polynomials of degree $\leq k$ and $\mathbb P_{-1}(E)=\{0\}$. A standard basis associated with the polynomial space $\mathbb P_k(E)$ is $\mathbb M_{k}(E)$, known as the set of normalized monomials, given as follows
\begin{align*}
	\mathbb M_k(E)= \Big\{ \Big(\dfrac{ \mathbf{x}-\mathbf{x}_E}{h_E} \Big)^{\mathbf{m}}, |\mathbf{m}| \leq k \Big\},
\end{align*}
where $\mathbf{m}=(m_1,m_2) \in \mathbb{N}^2$ is a multi-index such that $|\mathbf{m}|=m_1+m_2$, and $\mathbf{x}^{\mathbf{m}}= x_1^{m_1}x_2^{m_2}$.

We now introduce the standard polynomial projectors. For any $E \in \Omega_h$
\begin{itemize}
	\item the $H^1$-energy projection operator  $\Pi^{\nabla,E}_k: H^1(E) \rightarrow \mathbb{P}_k(E)$, given by
	\begin{align}
		\begin{cases}
			\int_E \nabla (w - \Pi^{\nabla,E}_k w) \cdot \nabla m_k \, dE = 0, \qquad \forall \, w \in H^1(E)\,\, \text{and} \,\, m_k \in \mathbb{P}_k(E),\\
			\int_{\partial E} (w - \Pi^{\nabla,E}_k w) \, ds=0,
		\end{cases}
		\label{H1proj}
	\end{align}
	with extension for the vector field $\boldsymbol{\Pi}^{\nabla,E}_k: [H^1(E)]^2 \rightarrow [\mathbb{P}_k(E)]^2$.
	\item  the $L^2$-projection operator $\Pi^{0,E}_k: L^2(E) \rightarrow \mathbb{P}_k(E)$, defined by
	\begin{align}
		\int_E (w - \Pi^{0,E}_k w) m_k \, dE = 0, \qquad \forall \,\, w \in L^2(E)\,\, \text{and} \,\, m_k \in \mathbb{P}_k(E), \label{l2proj}
	\end{align}
	with extension for vector field $\boldsymbol{\Pi}^{0,E}_k: [L^2(E)]^2 \rightarrow [\mathbb{P}_k(E)]^2$.
\end{itemize} 
\begin{remark} \label{stabilityP}
	(Stability of $\Pi^{0,E}_k$ in $L^q(E)$ for any $q\geq 2$) Utilizing the polynomial inverse inequality presented in \cite{scott}, we deduce that for any polynomial $p_k \in \mathbb P_k(E)$:
	\begin{align}
		\|p_k\|_{0,q,E} \lesssim h^{2/q-2/p}_E \|p_k\|_{0,p,E} \qquad \qquad \text{for any} \,\,\,p,q \in [0, \infty]. \label{inverse2}
	\end{align} 
	Therefore, for any $v \in L^2(E)$, we have
	\begin{align}
		\| \Pi^{0,E}_k v\|_{0,q,E} &\lesssim h^{2/q-1}_E \| \Pi^{0,E}_k v\|_{0,E}.
		\intertext{Using the stability of $\Pi^{0,E}_k$ in $L^2(E)$, it holds}
		\| \Pi^{0,E}_k v\|_{0,q,E} &\lesssim h^{2/q-1}_E \| v\|_{0,E}.
		\intertext{We now choose $r>0$ such that $\frac{1}{q}+\frac{1}{r}=\frac{1}{2}$, and then employing the H$\ddot{\text{o}}$lder inequality }
		\| \Pi^{0,E}_k v\|_{0,q,E} & \lesssim h^{\frac{2}{q}-1}_E \| 1\|_{0,r,E} \| v\|_{0,q,E} \nonumber \\
		& \lesssim h^{\frac{2}{q}-1}_E h^{2/r}_E \|v\|_{0,q,E} \nonumber \\
		& \lesssim h^{\frac{2}{q}-1}_E h^{2(\frac{1}{2}- \frac{1}{q})}_E \|v\|_{0,q,E} \nonumber \\
		& \lesssim  \|v\|_{0,q,E}.
	\end{align}
	Thus, the $L^2$-projection operator satisfies the stability property with respect to $L^q(E)$ norm for any $q \geq 2$.
\end{remark}
\subsection{Virtual element spaces }  
This section introduces the $H^1$-conforming virtual space for velocity vector field, temperature field, and  pressure field following the VEM framework \cite{vem25}. For each $E \in \Omega_h$, we define
\begin{align}
	V_h(E):= \Big\{ w_h \in H^1(E) \cap C^0( \partial{E}) \,\, \text{such that} \,\,\, &(i) \,\,\Delta w_h \in \mathbb P_k(E), \nonumber \\ &(ii)\, \,  w_{h|e} \in \mathbb P_k(e) \,\, \forall \,\, e \in \partial E,  \nonumber \\  &(iii) \,\, \big(w_h-\Pi_k^{\nabla,E}w_h, m_\alpha \big)_{E}=0, \, \forall \, m_\alpha \in \mathbb{M}^\ast_{k-1}(E) \cup \mathbb{M}^\ast_{k}(E) \Big\}, \nonumber
\end{align}
where $\mathbb{M}^\ast_{k}(E)$ is the set of monomials of degree equal to $k$.

Furthermore, we introduce following sets of linear operators $\mathbf{DF_1}$, $ \mathbf{DF_2}$, and $\mathbf{DF_3}: V_h(E) \rightarrow \mathbb R$, which constitute the degrees of freedom for the local virtual space $V_h(E)$
\begin{itemize}
	\item $\mathbf{DF_1}(w_h)$: the values of $w_h$ evaluated at each vertices of $E$;
	\item $\mathbf{DF_2}(w_h)$: the values of $w_h$ evaluated at $(k-1)$ distinct points of each edge $e \in \partial E$;
	\item $\mathbf{DF_3}(w_h)$: the internal moments of $w_h$ up to order $k-2$, 
	\begin{align*}
		\frac{1}{|E|} \, \int_E w_h m_\alpha \, dE \quad \forall \, m_\alpha \in \mathbb{M}_{k-2}(E).
	\end{align*}	
\end{itemize}
Therefore, the global virtual space is given as follows
\begin{align}
	V_h:=\{ w_h \in H^1(\Omega) \, \, \text{such that}\, \, w_{h|E} \in V_h(E) \, \,  \forall\, E \in \Omega_h \}.
\end{align}
We now define the global discrete velocities space $\mathbf{V}_{h}$, the pressure space $Q_{h}$ and the temperature space $\Sigma_h$, given as follows:
\begin{align}
	\mathbf{V}_{h} &:= \{ \mathbf{w}_h \in [H^1(\Omega)]^2 \cap \mathbf{V} \,\, \,\, \text{such that}\,\,\,\, \mathbf{w}_{h|E} \in [V_h(E)]^2 \,\,\,\, \forall \,\, E \in \Omega_h\}, \label{dspace-1} \\
	Q_{h} &:= \{ q_h \in H^1(\Omega)\cap Q \,\, \,\, \text{such that}\,\,\,\, q_{h|E} \in V_h(E) \,\, \,\,\forall \,\, E \in \Omega_h\}, \label{dspace-2} \\
	\Sigma_{h} &:= \{ \psi_h \in H^1(\Omega) \cap \Sigma \,\,\,\, \text{such that}\,\, \,\, \psi_{h|E} \in V_h(E) \,\,\,\, \forall \,\, E \in \Omega_h\}. \label{dspace-3}
\end{align} 

\begin{remark}
	The local space $\mathbf{V}_h(E)$ preserves the polynomial inclusion as $[\mathbb{P}_k(E)]^2 \subseteq \mathbf{V}_{h}(E)$, and the following operators are computable through the degrees of freedom $\mathbf{DF_1}$--$\mathbf{DF_3}$
	\begin{align*}
		\boldsymbol{\Pi}^{\nabla,E}_k: \mathbf{V}_{h}(E) \rightarrow [\mathbb{P}_k(E)]^2, \qquad	\boldsymbol{\Pi}^{0,E}_k: \mathbf{V}_{h}(E) \rightarrow [\mathbb{P}_k(E)]^2, \qquad  
		\boldsymbol{\Pi}^{0,E}_{k-1}: \nabla \mathbf{V}_{h}(E) \rightarrow [\mathbb{P}_{k-1}(E)]^{2 \times 2}.
	\end{align*}
\end{remark}
Before going to the next step, we note that for any element $E \in \Omega_{h}$, the local counterparts of the continuous forms $a_V, b,a_T$ and $c^S_T$ are denoted by $a^E_V, b^E, a_T^E$ and $c_T^{S,E}$, respectively. Furthermore
\begin{align*}
	a_{V}(\phi; \mathbf{v},\mathbf{w}) &= \sum_{E \in \Omega_{h}} a^E_{V}(\phi; \mathbf{v},\mathbf{w}),  \qquad \quad \,\,\,\, b(\mathbf{v},q) = \sum_{E \in \Omega_{h}} b^{E}(\mathbf{v},q), \nonumber \\
	a_{T}(\phi, \psi) &= \sum_{E \in \Omega_{h}} a^E_{T}(\phi, \psi),  \qquad \,\, \,\,\,
	c^S_T(\mathbf{w}; \phi,\psi) = \sum_{E \in \Omega_{h}} c^{S,E}_T(\mathbf{w}; \phi, \psi).
\end{align*}
Following this, we will apply the same decomposition to the external load terms in the subsequent analysis.

\subsection{Virtual element forms} 
Hereafter we discuss the virtual element approximation of the continuous forms and external load terms, which are easily computable by the degrees of freedom \textbf{DF}'s.

\noindent $\bullet$ Following the standard VEM literature \cite{vem25,vem5,vem04}, for any $\mathbf{v}_h, \mathbf{w}_h \in \mathbf{V}_h, q_h \in Q_h$, and $\phi_h, \psi_h \in \Sigma_h$,  the local discrete forms corresponding to continuous forms $a_V, b, a_T$ and $c^S_T$ are given as follows:
\begin{align}
	a^E_{V,h}( \phi_h; \mathbf{v}_h, \mathbf{w}_h) :&=  \int_E \mu(\Pi^{0,E}_k \phi_h) \boldsymbol{\Pi}^{0,E}_{k-1} \varepsilon (\mathbf{v}_h): \boldsymbol{\Pi}^{0,E}_{k-1} \varepsilon (\mathbf{w}_h)\, dE \,\,+ \nonumber \\ & \qquad  \mu(\Pi^{0,E}_0 \phi_h) S^E_ \nabla \big((\mathbf{I}-\boldsymbol{\Pi}^{\nabla,E}_k) \mathbf{v}_h,(\mathbf{I}-\boldsymbol{\Pi}^{\nabla,E}_k) \mathbf{w}_h \big),  \nonumber\\ 
	b^E_{h}(\mathbf{v}_h, q_h) :&= \int_E \Pi^{0,E}_{k-1} \nabla \cdot \mathbf{v}_h \Pi^{0,E}_{k} q_h\, dE, \nonumber \\
	a^E_{T,h}(\phi_h, \psi_h):&= \int_{E} \kappa \boldsymbol{\Pi}^{0,E}_{k-1} \nabla \phi_h \cdot \boldsymbol{\Pi}^{0,E}_{k-1} \nabla \psi_h \, dE + \kappa S^E_1 \big((I-\Pi^{\nabla,E}_k) \phi_h,(I- \Pi^{\nabla,E}_k) \psi_h \big), \nonumber \\
	c^{E}_{T,h}(\mathbf{v}_h; \phi_h,\psi_h):&= \int_E  \big[\boldsymbol{\Pi}^{0,E}_{k} \mathbf{v}_h \cdot	\boldsymbol{\Pi}^{0,E}_{k-1} \nabla \phi_h \big] \Pi^{0,E}_k \psi_h\, dE,  \nonumber \\
	c^{S,E}_{T,h}(\mathbf{v}_h; \phi_h,\psi_h):&= \dfrac{1}{2} \big[c^{E}_{T,h}(\mathbf{v}_h; \phi_h,\psi_h) - c^{E}_{T,h}(\mathbf{v}_h; \psi_h,\phi_h) \big].  \nonumber 
\end{align}
Moreover, the global forms can be given as follows
\begin{align*}
	a_{V,h}( \phi_h; \mathbf{v}_h,\mathbf{w}_h) &= \sum_{E \in \Omega_{h}} a^E_{V,h}(\phi_h; \mathbf{v}_h,\mathbf{w}_h),  \qquad \quad  b_{h}(\mathbf{v}_h,q_h) = \sum_{E \in \Omega_{h}} b^{E}_{h}(\mathbf{v}_h,q_h), \nonumber \\
	a_{T,h}(\phi_h, \psi_h) &= \sum_{E \in \Omega_{h}} a^E_{T,h}(\phi_h, \psi_h),  \qquad 
	c^S_{T,h}(\mathbf{v}_h; \phi_h,\psi_h) = \sum_{E \in \Omega_{h}} c^{S,E}_{T,h}(\mathbf{v}_h; \phi_h, \psi_h).
\end{align*}
For each $E \in \Omega_h$, the local VEM stabilization terms $S^E_\nabla(\cdot,\cdot): \mathbf{V}_{h}(E) \times \mathbf{V}_{h}(E) \rightarrow \mathbb{R}$ and $S^E_1(\cdot, \cdot): \Sigma_h(E) \times \Sigma_h(E) \rightarrow \mathbb R$ are the VEM computable symmetric positive definite bilinear forms such that 
\begin{align}
	\theta_{1\ast} \ \|\nabla \mathbf{v}_h\|_{0,E}^2 \leq S^E_\nabla(\mathbf{v}_h, \mathbf{v}_h)	 \leq
	\theta_1^\ast  \|\nabla \mathbf{v}_h\|_{0,E}^2 \qquad  \forall\, \mathbf{v}_h \in \mathbf{V}_{h}(E) \cap ker(\boldsymbol{\Pi}^{\nabla,E}_{k}), \label{vem-a} \\
	\theta_{2\ast} \ \| \nabla \psi_h\|_{0,E}^2 \leq S^E_1(\psi_h, \psi_h) \leq
	\theta_2^\ast  \| \nabla \psi_h\|_{0,E}^2 \qquad  \forall\, \psi_h \in \Sigma_{h}(E) \cap ker(\Pi^{\nabla,E}_{k}), \label{vem-b}
\end{align}
where the positive constants $\theta_{1\ast} \leq \theta_1^\ast$ and $\theta_{2\ast} \leq\theta_2^\ast$ are independent of the decomposition of $\Omega$. \newline
\noindent $\bullet$ Discrete load terms: For given $\phi_h \in \Sigma_h$, we define
\begin{align}
	{F}_{h,\phi_h}: \mathbf{V}_h \rightarrow \mathbb R \quad \qquad {F}_{h,\phi_h}(\mathbf{v}_h)&:= \sum_{E \in \Omega_h} \big( \alpha \mathbf{f} \Pi^{0,E}_k \phi_h, \boldsymbol{\Pi}^{0,E}_k \mathbf{v}_h\big)_E, \nonumber \\
	G_{h}: \Sigma_h \rightarrow \mathbb{R} \quad \qquad G_{h}(\psi_h)&:= \sum_{E \in \Omega_h} \big( g, \Pi^{0,E}_k \psi_h\big)_E. \nonumber 
\end{align}
Following these, we set for all $(\mathbf{v}_h,p_h), (\mathbf{w}_h, q_h) \in \mathbf{V}_h \times Q_h$, and $\phi_h, \psi_h \in \Sigma_h$,
\begin{align}
	A_{V,h}[\phi_h;(\mathbf{v}_h, p_h), (\mathbf{w}_h,q_h) ] &:= a_{V,h}(\phi_h; \mathbf{v}_h, \mathbf{w}_h)- b_h(\mathbf{w}_h, p_h) + b_h(\mathbf{v}_h, q_h), \\
	A_{T,h}(\mathbf{v}_h; \phi_h, \psi_h) &:= a_{T,h}(\phi_h, \psi_h)+ c^S_{T,h}(\mathbf{v}_h; \phi_h, \psi_h).
\end{align}
Thus, the classical virtual element problem for the coupled Stokes-Temperature equation $(S)$ is given as follows:
\begin{align}
	\begin{cases}
		\text{Find}\,\, (\mathbf{u}_h, p_h,\phi_h) \in \mathbf{V}_h \times Q_h \times \Sigma_h,\,\, \text{such that} \\
		A_{V,h}[ \phi_h; (\mathbf{u}_h, p_h), (\mathbf{v}_h, q_h)]= F_{h,\phi_h}(\mathbf{v}_h) \quad \forall \, (\mathbf{v}_h, q_h) \in \mathbf{V}_h \times Q_h, \\
		A_{T,h}(\mathbf{u}_h; \phi_h, \psi_h)=G_{h}(\psi_h) \quad \forall \, \psi_h \in \Sigma_h.
	\end{cases}
	\label{std_vem}
\end{align}

Since the equal-order virtual element approximations are employed for the study, the solution domain is significantly affected by non-physical oscillations due to the violation of discrete inf-sup condition. To address these issues, a stabilized form of problem \eqref{std_vem} is necessary.

\subsection{ Stabilized virtual element problem }
\label{sec:3.4}
We finally introduce fluctuation-controlling operators to control the fluctuations of the gradients of pressure and temperature, and to mitigate the violation of the divergence-free constraints. 
We define a fluctuation operator $\widehat{r}_h : L^2(E) \rightarrow L^2(E)$ such that $\widehat{r}_h := \Pi^{0,E}_{k} - \Pi^{0,E}_{k-1}$ with extension for vector field  $\widehat{\mathbf{r}}_h : [L^2(E)]^2 \rightarrow [L^2(E)]^2$, with $\widehat{\mathbf{r}}_h := \boldsymbol{\Pi}^{0,E}_k - \boldsymbol{\Pi}^{0,E}_{k-1}$. Employing the above fluctuation-controlling operators, we define the following stabilization terms:

\begin{itemize}
	\item The continuity equation is stabilized by the following stabilization term:
	\begin{align}
		\mathcal{L}_{1,h}(\mathbf{v}_h,\mathbf{w}_h) &:= \sum_{E \in \Omega_h}\tau_{1,E} \big[\big( \widehat{r}_h(\nabla \cdot \mathbf{v}_h), \widehat{r}_h(\nabla \cdot \mathbf{w}_h) \big)_E  + S^E_\nabla\big((\mathbf{I}-\boldsymbol{\Pi}^{\nabla,E}_k) \mathbf{v}_h,(\mathbf{I}-\boldsymbol{\Pi}^{\nabla,E}_k) \mathbf{w}_h \big)\big].  \label{div_1}
	\end{align}
	\item Pressure stabilizing term:
	\begin{align}
		\mathcal{L}_{2,h}(p_h,q_h) &:= \sum_{E \in \Omega_h}\tau_{2,E} \big[\big( \widehat{\mathbf{r}}_h(\nabla p_h), \widehat{\mathbf{r}}_h(\nabla q_h) \big)_E  + S^E_2\big((I-\Pi^{\nabla,E}_{k-1}) p_h, (I-\Pi^{\nabla,E}_{k-1}) q_h \big)\big].  \label{press_1}
	\end{align}
	Moreover, there exist two positive constants $\theta_{3\ast} \leq \theta_3^\ast$ independent of the mesh size such that there holds
	\begin{align}
		\theta_{3\ast} \ \|\nabla q_h\|_{0,E}^2 \leq S^E_2(q_h, q_h) \leq
		\theta_3^\ast  \|\nabla q_h\|_{0,E}^2 \qquad  \forall\, q_h \in Q_{h}(E) \cap ker({\Pi}^{\nabla,E}_{k-1}). \label{vem-c}
	\end{align}
	\item Stabilization term for transport equation for temperature:
	\begin{align}
		\mathcal{L}_{3,h}(\phi_h,\psi_h) &:= \sum_{E \in \Omega_h}\tau_{3,E} \big[\big( \widehat{\mathbf{r}}_h(\nabla \phi_h), \widehat{\mathbf{r}}_h(\nabla \psi_h) \big)_E  + S^E_1\big((I-\Pi^{\nabla,E}_{k}) \phi_h, (I-\Pi^{\nabla,E}_{k}) \psi_h \big)\big],  \label{trans_1}
	\end{align}
	where $\tau_{i,E}$ are stabilization parameters for $i=1, 2$ and $3$, {which depend on the mesh size and will be specified later.}
\end{itemize} 
Let $\mathcal{L}_{h}[(\mathbf{v}_h, p_h), (\mathbf{w}_h,q_h)]:= \mathcal{L}_{1,h}(\mathbf{v}_h,\mathbf{w}_h)+ \mathcal{L}_{2,h}(p_h, q_h)$. Therefore, the stabilized coupled virtual element problem for Stokes-Temperature equation $(S)$ is given as follows:
\begin{align}
	\begin{cases}
		\text{Find}\,\, (\mathbf{u}_h, p_h,\phi_h) \in \mathbf{V}_h \times Q_h \times \Sigma_h,\,\, \text{such that} \\
		A_{V,h}[ \phi_h; (\mathbf{u}_h, p_h), (\mathbf{v}_h, q_h)]+ \mathcal{L}_{h}[(\mathbf{u}_h, p_h), (\mathbf{v}_h,q_h)]= F_{h,\phi_h}(\mathbf{v}_h) \quad \forall \, (\mathbf{v}_h, q_h) \in \mathbf{V}_h \times Q_h, \\
		A_{T,h}(\mathbf{u}_h; \phi_h, \psi_h)+ 	\mathcal{L}_{3,h}(\phi_h,\psi_h)=G_{h}(\psi_h) \quad \forall \, \psi_h \in \Sigma_h.
	\end{cases}
	\label{nvem}
\end{align}
\begin{remark}
	From the stabilized problem \eqref{nvem}, it is noteworthy that the classical virtual element problem \eqref{std_vem} primarily requires the stabilization of the continuity equation to improve the violation of divergence-free constraints, the stabilization of pressure due to violation of discrete inf-sup condition, and the stabilization of transport equation to provide a reasonable control over the gradient of solution. While the choice of $\mathcal{L}_{3,h}$ can be relaxed for the dominated-diffusion regimes. Our study considers the stabilized problem \eqref{nvem}, demonstrating that the inclusion of $\mathcal{L}_{3,h}$ does not compromise the robustness of the proposed method in diffusion-dominated transport regimes.
\end{remark}

\section{Well-posedness} \label{sec:4}
This section is dedicated to demonstrating the well-posedness of the stabilized virtual element problem. To achieve this, we begin by defining the discrete fixed-point problem  which is equivalent to the problem \eqref{nvem}.
\subsection{Discrete fixed-point problem}
We define the discrete operator $\mathbf{S}_h: \Sigma_h \rightarrow \mathbf{V}_h \times Q_h$ such that
\begin{align}
	\phi_h \rightarrow \mathbf{S}_h(\phi_h)=\big( S^1_h(\phi_h), S^2_h(\phi_h)\big):=(\mathbf{u}_h,p_h), \label{dfix-1}
\end{align}
where $(\mathbf{u}_h,p_h) \in \mathbf{V}_h \times Q_h$ is the solution of the decoupled problem, given by
\begin{align}
	\begin{cases}
		\text{For given $\phi_h \in \Sigma_h$, find}\,\, (\mathbf{u}_h, p_h) \in \mathbf{V}_h \times Q_h,\,\, \text{such that} \\
		A_{V,h}[ \phi_h; (\mathbf{u}_h, p_h), (\mathbf{v}_h, q_h)]+ \mathcal{L}_{h}[(\mathbf{u}_h, p_h), (\mathbf{v}_h,q_h)]= F_{h,\phi_h}(\mathbf{v}_h) \quad \forall \, (\mathbf{v}_h, q_h) \in \mathbf{V}_h \times Q_h.
	\end{cases}
	\label{dd_couple-1}
\end{align}
Furthermore, we define the operator $\mathbb M_h: \mathbf{V}_h \rightarrow \Sigma_h$ such that
\begin{align}
	\mathbf{u}_h \rightarrow \mathbb M_h(\mathbf{u}_h)=\phi_h, \label{dfix-2}
\end{align}
where $\phi_h \in \Sigma_h$ represents the solution of the decoupled problem, defined by
\begin{align}
	\begin{cases}
		\text{For given $ \mathbf{u}_h \in \mathbf{V}_h$, find}\,\, \phi_h \in \Sigma_h,\,\, \text{such that} \\
		A_{T,h}(\mathbf{u}_h; \phi_h, \psi_h)+ 	\mathcal{L}_{3,h}(\phi_h,\psi_h)=G_{h}(\psi_h) \quad \forall \, \psi_h \in \Sigma_h.
	\end{cases}
	\label{dd_couple-2}
\end{align}
Following the continuous case, we define the operator $\mathbb T_h: \Sigma_h \rightarrow \Sigma_h$ such that
\begin{align}
	\phi_h \rightarrow \mathbb T_h (\phi_h)= \mathbb M_h(S^1_h(\phi_h)) \qquad \,\, \text{for all} \,\, \phi_h \in \Sigma_h.
\end{align}
Thus, the discrete fixed-point problem equivalent to the stabilized problem \eqref{nvem} can be given as follows:
\begin{align}
	\begin{cases}
		\text{ Find}\,\, \phi_h \in \Sigma_h,\,\, \text{such that} \\
		\mathbb T_h(\phi_h)=\phi_h.
	\end{cases}
	\label{ddfix}
\end{align}
\subsection{Well-posedness of the discrete decoupled problems}
We now define the mesh-dependent norms over $\mathbf{V}_h \times Q_h$ and $\Sigma_h$, which will be utilized in the subsequent analysis.
\begin{align}
	\vertiii{(\mathbf{v}_h,q_h)}^2_h&= \mu_{min} \|\nabla \mathbf{v}_h\|^2_{0,\Omega} + \|q_h\|^2_{0,\Omega} + \mathcal{L}_{h}[(\mathbf{v}_h,q_h),(\mathbf{v}_h,q_h)] \qquad \,\, \text{for all} \,\, (\mathbf{v}_h,q_h) \in \mathbf{V}_h \times Q_h,  \\
	\vertiii{\psi_h}^2_{\Sigma_h}&=  \kappa \|\nabla \psi_h\|^2_{0,\Omega} + \mathcal{L}_{3,h}(\psi_h,\psi_h) \qquad \,\, \text{for all} \,\, \psi_h \in \Sigma_h.
\end{align}

\begin{remark} \label{estimate}
	For all $E \in \Omega_h$, and for any $ \psi_h \in H^1(E)$, employing the orthogonality of the projectors, we can derive the following results:
	\begin{align}
		\|(\mathbf{I}- \boldsymbol{\Pi}^{0,E}_{k-1}) \nabla  \psi_h \|_{0,E} &\lesssim \| \nabla ({I}- {\Pi}^{\nabla,E}_{k})  \psi_h\|_{0,E} \lesssim \| \nabla ({I}- {\Pi}^{\nabla,E}_{k-1})  \psi_h\|_{0,E}, \\
		\|\nabla{\Pi}^{\nabla,E}_k  \psi_h \|_{0,E} & \lesssim \| \nabla  \psi_h\|_{0,E}.
	\end{align}
\end{remark} 
Hereafter, we will first establish the weak inf-sup condition for the discrete form $b_h$. In this sequel, we recall the following results:
\begin{lemma} (Inverse inequality \cite{vem28}) \label{inverse}
	Under the assumption \textbf{(A1)}, for any virtual function ${\psi}_h \in V_h(E)$ for all $E \in \Omega_{h}$, it yields that
	\begin{align}
		|{\psi}_h|_{1,E} \lesssim h^{-1}_E \|{\psi}_h\|_{0,E} \label{inverse1}.
	\end{align}
\end{lemma}
\begin{lemma} \label{Beiga_1}
	Under the mesh regularity assumption \textbf{(A1)}, there exists a positive constant $\widehat{\beta}_0$ independent of the decomposition of $\Omega$, such that for all $k\geq 2$, there holds
	\begin{align}
		\sup \limits_{\mathbf{0} \neq \mathbf{w}_h \in \mathbf{V}_{h}} \dfrac{b_{h}(\mathbf{w}_h, q_h)}{\|\mathbf{w}_h\|_{1,\Omega}} \geq \widehat{\beta}_0 \|\Pi^{\nabla}_{k-1} q_h\|_{0,\Omega}  \qquad \forall \, q_h \in Q_h. \label{inf-sup}
	\end{align}
\end{lemma} 
\begin{proof}
	The proof of Lemma \ref{Beiga_1} follows from \cite{vem2}. 
\end{proof}
\begin{lemma}(weak inf-sup condition) \label{infsup}
	Under the mesh regularity assumption \textbf{(A1)}, there exist positive constants $\widehat{\beta}_1$ and $\widehat{\beta}_2$ independent of the mesh size, satisfying
	\begin{align}
		\sup \limits_{\mathbf{0} \neq \mathbf{w}_h \in \mathbf{V}_h} \dfrac{b_{h}(\mathbf{w}_h, q_h)}{\|\mathbf{w}_h\|_{1,\Omega}} \geq \widehat{\beta}_1 \|q_h\|_{0,\Omega} - \widehat{\beta}_2 [\mathcal{L}_{2,h}(q_h,q_h)]^{\frac{1}{2}} \qquad \forall \, q_h \in Q_h.  \label{inf-sup1}
	\end{align}
\end{lemma}
\begin{proof}
	We will establish the proof of Theorem \ref{infsup} following \cite[Theorem 4.4]{vem028}. We recall for any $q_h \in Q_h$ there exists $\mathbf{z} \in \mathbf{V}$ such that
	\begin{align}
		\big( \nabla \cdot \mathbf{z}, q_h\big) \geq C \|\mathbf{z} \|_{1,\Omega} \|q_h\|_{0,\Omega} \qquad \text{for all} \,\, q_h \in Q_h. \label{weak-01}
	\end{align}	
	Let $\mathbf{z}_I \in \mathbf{V}_{h}$ denote the interpolant of $\mathbf{z} \in \mathbf{V}$. Therefore, we have the following estimates
	\begin{align}
		\|\mathbf{z} - \mathbf{z}_I\|_{0,\Omega} \leq C h \|\mathbf{z}\|_{1,\Omega} \qquad \text{and} \quad \,\, \|\mathbf{z}_I\|_{1,\Omega} \leq C \|\mathbf{z}\|_{1,\Omega}. \label{weak-11}
	\end{align}	
	Using the definition of projection operators and \eqref{weak-11}, we infer
	\begin{align}
		\sup \limits_{\mathbf{w}_h \in \mathbf{V}_h} \dfrac{b_h(\mathbf{w}_h, q_h)}{\|\mathbf{w}_h\|_{1,\Omega}} & = \sup \limits_{\mathbf{w}_h \in \mathbf{V}_h} \dfrac{ \sum_{E \in \Omega_h} \big( \Pi^{0,E}_{k-1} \nabla \cdot \mathbf{w}_h, \Pi^{0,E}_k q_h \big)}{\|\mathbf{w}_h\|_{1,\Omega}} \nonumber \\
		& \geq \dfrac{ \sum_{E \in \Omega_h} \big( \Pi^{0,E}_{k-1} \nabla \cdot \mathbf{z}_I, \Pi^{0,E}_k q_h \big)}{\|\mathbf{z}_I\|_{1,\Omega}} \nonumber \\
		& \geq \dfrac{ \sum_{E \in \Omega_h} \big(  \nabla \cdot \mathbf{z}_I, \Pi^{\nabla,E}_{k-1} q_h \big)}{C \|\mathbf{z}\|_{1,\Omega}}. \label{eq41}
	\end{align}	
	We now use the following settings: 
	\begin{align}
		\big( \nabla \cdot \mathbf{z}_I, \Pi^{\nabla,E}_{k-1} q_h \big) = -\big(\nabla \cdot \mathbf{z}_I, q_h -\Pi^{\nabla,E}_{k-1}q_h \big) - \big( \nabla \cdot (\mathbf{z} - \mathbf{z}_I), q_h \big) + \big(\nabla \cdot \mathbf{z}, q_h \big). \label{eq51}
	\end{align}
	Combining Eqn. \eqref{eq51} and the inequality \eqref{eq41}, we obtain 
	\begin{align}
		\sup \limits_{\mathbf{w}_h \in \mathbf{V}_h} \dfrac{b_h(\mathbf{w}_h, q_h)}{\|\mathbf{w}_h\|_{1,\Omega}} & \geq - \dfrac{| \sum_{E \in \Omega_h} \big(  \nabla \cdot \mathbf{z}_I, q_h -\Pi^{\nabla,E}_{k-1}q_h \big)|}{C \|\mathbf{z}\|_{1,\Omega}} - \dfrac{| \sum_{E \in \Omega_h} \big( \nabla \cdot (\mathbf{z} - \mathbf{z}_I), q_h \big)|}{C \|\mathbf{z}\|_{1,\Omega}} + \dfrac{ \big(  \nabla \cdot \mathbf{z},  q_h \big)}{C \|\mathbf{z}\|_{1,\Omega}}. \nonumber
	\end{align}
	Employing the integration by parts for the second term and \eqref{weak-01} for the last term, we have
	\begin{align}
		\sup \limits_{\mathbf{w}_h \in \mathbf{V}_h} \dfrac{b_h(\mathbf{w}_h, q_h)}{\|\mathbf{w}_h\|_{1,\Omega}} & \geq - C \Big[\| (I- \Pi^{\nabla}_{k-1}) q_h \|_{0,\Omega} + 
		h \| \nabla q_h \|_{0,\Omega}  \Big] + C \| q_h \|_{0,\Omega}.\label{weak-21} 
	\end{align}
	Using the Poincar$\acute{\text{e}}$ inequality \cite[Lemma 2.2]{vem28}, we infer
	\begin{align}
		\| (I- \Pi^{\nabla}_{k-1}) q_h \|^2_{0,\Omega} &= \sum_{E \in \Omega_h} \| (I- \Pi^{\nabla,E}_{k-1}) q_h \|^2_{0,E} \leq C \sum_{E \in \Omega_h} h^2_E \| \nabla(I- \Pi^{\nabla,E}_{k-1}) q_h \|^2_{0,E}. \label{w4}
	\end{align}
	
	\noindent For $k=1$, combining \eqref{weak-21} and \eqref{w4}, we conclude that
	\begin{align}
		\sup \limits_{\mathbf{w}_h \in \mathbf{V}_h} \dfrac{b_h(\mathbf{w}_h, q_h)}{\|\mathbf{w}_h\|_{1,\Omega}} & \geq - C  h \| \nabla (I- \Pi^{\nabla}_{0}) q_h \|_{0,\Omega}  + C \| q_h \|_{0,\Omega}. \label{weak-03}
	\end{align}
	Employing the triangle inequality and Lemma \ref{inverse}, we obtain
	\begin{align}
		h^2 \| \nabla q_h \|^2_{0,\Omega} &\leq C \sum_{E \in \Omega_h} \big[h^2 \| \nabla \Pi^{\nabla,E}_{k-1} q_h \|^2_{0,E} + 	h^2 \| \nabla(I-\Pi^{\nabla,E}_{k-1} )q_h\|^2_{0,E} \big] \nonumber \\
		&\leq C \big[ \| \Pi^{\nabla}_{k-1} q_h \|^2_{0,\Omega} + 	h^2 \| \nabla(I-\Pi^{\nabla}_{k-1} )q_h\|^2_{0,\Omega} \big]. \label{w5}
	\end{align}
	For $k \geq 2$, combining \eqref{weak-21}, \eqref{w4} and \eqref{w5}, we arrive at
	\begin{align}
		\sup \limits_{\mathbf{w}_h \in \mathbf{V}_h} \dfrac{b_h(\mathbf{w}_h, q_h)}{\|\mathbf{w}_h\|_{1,\Omega}} 
		&\geq - C \Big[ \|  \Pi^{\nabla}_{k-1} q_h \|_{0,\Omega} + h \| \nabla(I- \Pi^{\nabla}_{k-1}) q_h \|_{0,\Omega}  \Big] + C \| q_h \|_{0,\Omega}. \label{weak-22}
	\end{align}
	Using Lemma \ref{Beiga_1}, we obtain for $k \geq 2$
	\begin{align}
		\sup \limits_{\mathbf{w}_h \in \mathbf{V}_h} \dfrac{b_h(\mathbf{w}_h, q_h)}{\|\mathbf{w}_h\|_{1,\Omega}} & \geq - C  h \| \nabla (I- \Pi^{\nabla}_{k-1}) q_h \|_{0,\Omega}  + \widehat \beta_1 \| q_h \|_{0,\Omega}. \label{weak-04}
	\end{align}	
	Employing the triangle inequality, Remark \ref{estimate} and the result \eqref{vem-c}, there holds that
	\begin{align}
		h^2\| \nabla (I- \Pi^{\nabla}_{k-1}) q_h \|^2_{0,\Omega} & \leq	h^2 \sum_{E \in \Omega_h } \| \nabla (I- \Pi^{\nabla,E}_{k-1}) q_h \|^2_{0,E} \nonumber \\
		&\leq 4 h^2 \sum_{E \in \Omega_h} \Big[\|\nabla q_h- \boldsymbol{\Pi}^{0,E}_k \nabla q_h\|^2_{0,E} + \| \boldsymbol{\Pi}^{0,E}_k \nabla q_h - \boldsymbol{\Pi}^{0,E}_{k-1}\nabla q_h\|^2_{0,E} \,+ \nonumber \\ & \qquad \| \nabla q_h - \boldsymbol{\Pi}^{0,E}_{k-1}\nabla q_h\|^2_{0,E}  + \| \nabla q_h- \nabla \Pi^{\nabla,E}_{k-1} q_h\|^2_{0,E}\Big] \nonumber \\
		& \leq 4h^2 \sum_{E \in \Omega_h} \Big[ \|\boldsymbol{\Pi}^{0,E}_k \nabla q_h- \boldsymbol{\Pi}^{0,E}_{k-1} \nabla q_h\|^2_{0,E} + 3 \| \nabla q_h- \nabla \Pi^{\nabla,E}_{k-1} q_h\|^2_{0,E} \Big] \nonumber \\
		&\leq \frac{ 12 h^2}{ \min_{E \in \Omega_h}\tau_{2,E}} \sum_{E \in \Omega_h} \Big[ \tau_{2,E} \|\boldsymbol{\Pi}^{0,E}_k \nabla q_h- \boldsymbol{\Pi}^{0,E}_{k-1} \nabla q_h\|^2_{0,E}\, + \nonumber \\ & \qquad \frac{\tau_{2,E}}{\theta_{3\ast}} S^E_2 \big( (I-\Pi^{\nabla,E}_{k-1})q_h,(I-\Pi^{\nabla,E}_{k-1})q_h \big) \Big] \nonumber \\
		& \leq \frac{12 h^2}{ \min_{E \in \Omega_h}\tau_{2,E}} \max \Big\{1, \frac{1}{\theta_{3\ast}} \Big\} \mathcal{L}_{2,h}(q_h,q_h). \label{w6}
		\intertext{It is obvious that \eqref{w6} holds for all $k\geq 1$. Furthermore, to ensure that the constant on the right-hand side of \eqref{w6} remains independent of the mesh size, we assume {$\tau_{2,E} \sim h^2_E$.} Consequently, \eqref{w6} yields}
		h^2 \| \nabla (I- \Pi^{\nabla}_{k-1}) q_h \|^2_{0,\Omega} & \leq C \max \Big\{1, \frac{1}{\theta_{3\ast}} \Big\} \mathcal{L}_{2,h}(q_h,q_h), \label{w7}
	\end{align}
	thus,  the estimate \eqref{inf-sup1} readily follows from the estimates \eqref{weak-03}, \eqref{weak-04} and \eqref{w7}. 
\end{proof}
\begin{remark} \label{inf11}
	From \eqref{w6}, we conclude that Lemma \ref{infsup} holds when {$\tau_{2,E} \sim h^2_E$} for all $E \in \Omega_h$.
\end{remark}
\begin{lemma}{\big(well-posedness of \eqref{dd_couple-1}\big)} \label{wellposed-1}
	Under the mesh regularity assumption \textbf{(A1)}, for any $(\mathbf{u}_h, p_h) \in \mathbf{V}_h \times Q_h$ there exists a positive constant $\widetilde{\beta}$ that is independent of the mesh size $h$, such that for given $\phi_h \in \Sigma_h$, the following holds 
	\begin{align}
		\sup \limits_{(\mathbf{0},0) \neq (\mathbf{v}_h, q_h) \in \mathbf{V}_h \times Q_h} \dfrac{A_{V,h} [ \phi_h; (\mathbf{u}_h, p_h), (\mathbf{v}_h, q_h)] + \mathcal{L}_h[ (\mathbf{u}_h, p_h), (\mathbf{v}_h, q_h)]}{\vertiii{(\mathbf{v}_h, q_h)}_h}  \gtrsim \widetilde{\beta} \vertiii{(\mathbf{u}_h, p_h)}_h, \label{wellpos-0}
	\end{align}
	furthermore, problem \eqref{dd_couple-1} has a unique solution $(\mathbf{u}_h, p_h) \in \mathbf{V}_h \times Q_h$ such that it holds
	\begin{align}
		\vertiii{\mathbf{S}_h(\phi_h)}_h:=\vertiii{(\mathbf{u}_h,p_h)}_h \lesssim \frac{\alpha C_q^2(1+C_P)^2}{\widetilde{\beta}\sqrt{ \kappa \mu_{min}}} \|\mathbf{f}\|_{0, \Omega} \vertiii{ \phi_h}_{\Sigma_h}.  \label{wellpos-s0}
	\end{align}
\end{lemma}
\begin{proof}
	We prove the estimate \eqref{wellpos-0} in the following steps: \newline
	$\bullet$ Using the estimates \eqref{vem-a}, the definition of projectors and Remark \ref{estimate}, we infer that
	\begin{align}
		A_{V,h}[ \phi_h; (\mathbf{u}_h, p_h), (\mathbf{u}_h, p_h)] &=a_{V,h}(\mathbf{u}_h,\mathbf{u}_h) \nonumber \\ 
		& \gtrsim \sum_{E \in \Omega_{h}} \min\{1, \theta_{1\ast}\} \mu_{min} \big[  \| \boldsymbol{\Pi}^{0,E}_{k-1} \varepsilon (\mathbf{u}_h) \|^{2}_{0,E}  +  \|\nabla (I-\boldsymbol{\Pi}^{\nabla,E}_k) \mathbf{u}_h\|^{2}_{0,E} \big]    \nonumber \\
		& \gtrsim \sum_{E \in \Omega_{h}} \min\{1, \theta_{1\ast}\} \mu_{min} \big[  \| \boldsymbol{\Pi}^{0,E}_{k-1} \varepsilon (\mathbf{u}_h) \|^{2}_{0,E}  +  \|\varepsilon( (I-\boldsymbol{\Pi}^{\nabla,E}_k) \mathbf{u}_h)\|^{2}_{0,E} \big]    \nonumber \\
		& \gtrsim \sum_{E \in \Omega_{h}} \min\{1, \theta_{1\ast}\} \mu_{min} \big[  \| \boldsymbol{\Pi}^{0,E}_{k-1} \varepsilon (\mathbf{u}_h) \|^{2}_{0,E}  +  \|( \mathbf{I}-\boldsymbol{\Pi}^{0,E}_{k-1}) \varepsilon(\mathbf{u}_h)\|^{2}_{0,E} \big]    \nonumber \\
		& \gtrsim \sum_{E \in \Omega_h } \min\{1, \theta_{1\ast}\} \mu_{min} \|\varepsilon (\mathbf{u}_h)\|^{2}_{0,E} \gtrsim \sum_{E \in \Omega_h } C_{\theta_\ast} \mu_{min} \|\nabla \mathbf{u}_h\|^{2}_{0,E}, \label{wellpos-1}
	\end{align}
	where $C_{\theta_\ast}=C_k\min\{1, \theta_\ast\}$, with $C_k>0$ is the constant from Korn's inequality. \newline
	$\bullet$ Lemma \ref{infsup} ensures that for each $q_h \in Q_h$,
	there exists $\widehat{\mathbf{w}}_h \in \mathbf{V}_{h}$ such that 
	\begin{align}
		b_{h}(\widehat{\mathbf{w}}_h, q_h) \geq \|\widehat{\mathbf{w}}_h\|_{1,\Omega} \big[\widehat{\beta}_1 \|q_h\|_{0,\Omega} - \widehat{\beta}_2 [\mathcal{L}_{2,h}(q_h,q_h)]^{\frac{1}{2}} \big]. \label{wellpos-2}
	\end{align}
	Taking $\mathbf{w}_h$ = $\frac{\|p_h\|_{0,\Omega}}{\|\nabla \widehat{\mathbf{w}}_h\|_{0,\Omega}} \widehat{\mathbf{w}}_h $, this gives $\|\nabla \mathbf{w}_h\|_{0,\Omega}= \|p_h\|_{0,\Omega}$. Then we have
	\begin{align}
		b_{h}(\mathbf{w}_h, p_h) \geq \|p_h\|_{0,\Omega} \big[ \widehat{\beta}_1 \|p_h\|_{0,\Omega} - \widehat{\beta}_2 [\mathcal{L}_{2,h}(p_h,p_h)]^{\frac{1}{2}} \big]. \label{wellpos-3}
	\end{align}
	We now use $(\mathbf{v}_h, q_h) = (-\mathbf{w}_h, 0)$, and employing the estimate \eqref{wellpos-3} and the Young inequality, we obtain
	\begin{align}
		A_{V,h}[ \phi_h;&(\mathbf{u}_h, p_h), (\mathbf{v}_h, q_h)] + \mathcal{L}_h[(\mathbf{u}_h, p_h), (\mathbf{v}_h,q_h)] =- a_{V,h}(\phi_h; \mathbf{u}_h,\mathbf{w}_h) + b_{h}(\mathbf{w}_h,p_h)  - \mathcal{L}_{1,h}(\mathbf{u}_h, \mathbf{w}_h)\nonumber  \nonumber\\
		& \gtrsim - \max \{1, \theta_1^\ast \} \mu_{max} \|\nabla \mathbf{u}_h\|_{0, \Omega} \|\nabla \mathbf{w}_h\|_{0, \Omega} + \|p_h\|_{0,\Omega} \big[\widehat{\beta}_1 \|p_h\|_{0,\Omega} - \widehat{\beta}_2 [\mathcal{L}_{2,h}(p_h,p_h)]^{\frac{1}{2}} \big] \nonumber \\ &\qquad -  \max \limits_{E\in\Omega_h} (1+\theta_1^\ast) \tau_{1,E} \|\nabla \mathbf{u}_h\|_{0, \Omega} \|\nabla \mathbf{w}_h\|_{0, \Omega}  \nonumber \\
		& \gtrsim - C_{\theta1^\ast}  \sqrt{\mu_{min}} \|\nabla \mathbf{u}_h\|_{0, \Omega} \|\nabla \mathbf{w}_h\|_{0, \Omega} +  \widehat{\beta}_1 \|p_h\|^2_{0,\Omega} - \widehat{\beta}_2 \|p_h\|_{0,\Omega} [\mathcal{L}_{2,h}(p_h,p_h)]^{\frac{1}{2}}  \nonumber \\ 
		& \gtrsim - C^2_{\theta1^\ast} t_1 \mu_{min} \|\nabla \mathbf{u}_h\|^2_{0, \Omega} - \dfrac{1}{8t_1} \| p_h\|^2_{0, \Omega} + \widehat{\beta}_1 \|p_h\|^2_{0,\Omega} - \dfrac{\widehat{\beta}_2^2}{t_2}  \|p_h\|^2_{0,\Omega} - \dfrac{t_2}{8} \mathcal{L}_{2,h}(p_h,p_h) \nonumber \\
		& \gtrsim - \frac{C^2_{\theta1^\ast}}{\widehat{\beta}_1} \mu_{min} \|\nabla \mathbf{u}_h\|^2_{0, \Omega} + \dfrac{\widehat{\beta}_1}{4} \|p_h\|^2_{0,\Omega} - \dfrac{\widehat{\beta}_2^2}{5\widehat{\beta}_1} \mathcal{L}_{2,h}(p_h,p_h), \label{wellpos-4}
	\end{align} 
	where $C_{\theta1^\ast}:= \frac{\max \{1, \theta_1^\ast \} \mu_{max} +\, \max_{E\in\Omega_h} (1+\theta_1^\ast) \tau_{1,E}}{\sqrt{\mu_{min}}}$, and the last line is obtained by employing $t_1= \frac{1}{\widehat{\beta}_1}$ and $t_2=\frac{8 \widehat{\beta}_2^2}{5 \widehat{\beta}_1}$.  \newline
	$\bullet$ Employing $(\mathbf{v}_h, q_h) =(\mathbf{u}_h, p_h) - \Theta (\mathbf{w}_h, 0)$ with a constant $\Theta >0$. Recalling \eqref{wellpos-1} and \eqref{wellpos-4}, we infer
	\begin{align}
		A_{V,h}[ \phi_h;&(\mathbf{u}_h, p_h), (\mathbf{v}_h, q_h)] + \mathcal{L}_h[(\mathbf{u}_h, p_h), (\mathbf{v}_h,q_h)] \nonumber \\
		&=A_{V,h}[ \phi_h;(\mathbf{u}_h, p_h), (\mathbf{u}_h, p_h)]- \Theta A_{V,h}[ \phi_h;(\mathbf{u}_h, p_h), (\mathbf{w}_h, 0)] \nonumber \\ & \qquad+ \mathcal{L}_h[(\mathbf{u}_h, p_h), (\mathbf{u}_h,p_h)] - \Theta \mathcal{L}_h[(\mathbf{u}_h, p_h), (\mathbf{w}_h,0)]\nonumber \\ 
		& \gtrsim \big[C_{\theta_\ast} -  \frac{ \Theta C^2_{\theta1^\ast}}{\widehat{\beta}_1}  \big] \mu_{min} \|\nabla \mathbf{u}_h\|^2_{0, \Omega} + \frac{\Theta \widehat{\beta}_1}{4} \|p_h\|^2_{0, \Omega}  + \mathcal{L}_{1,h}(\mathbf{u}_h, \mathbf{u}_h) + \big(1 - \frac{\Theta \widehat{\beta}^2_2}{5\widehat{\beta}_1}\big) \mathcal{L}_{2,h}(p_h,p_h).  \label{wellpos-5}
	\end{align}  
	Choosing $\Theta = \min \Big\{ \frac{3 C_{\theta_\ast} \widehat{\beta}_1}{4 C^2_{\theta1^\ast}},  \frac{ 15 \widehat{\beta}_1}{4 \widehat{\beta}^2_2} \Big\}$, then we obtain
	\begin{align}
		A_{V,h}&[ \phi_h;(\mathbf{u}_h, p_h), (\mathbf{v}_h, q_h)] + \mathcal{L}_h[(\mathbf{u}_h, p_h), (\mathbf{v}_h,q_h)] \nonumber \\
		& \gtrsim  \frac{C_{\theta_\ast}}{4} \mu_{min} \|\nabla \mathbf{u}_h\|^2_{0, \Omega} + \frac{1}{4} \min \Big\{ \frac{3 C_{\theta_\ast} \widehat{\beta}^2_1}{4 C^2_{\theta1^\ast}},  \frac{ 15 \widehat{\beta}_1^2}{4 \widehat{\beta}^2_2} \Big\} \|p_h\|^2_{0, \Omega} + \mathcal{L}_{1,h}(\mathbf{u}_h, \mathbf{u}_h) +  \frac{1}{4} \mathcal{L}_{2,h}(p_h,p_h) \nonumber \\
		& \gtrsim \Lambda_0 \vertiii{(\mathbf{u}_h, p_h)}_h^2, \label{wellpos-6} 		
	\end{align}
	where $\Lambda_0 = \min \Big\{ \frac{1}{4}, \frac{C_{\theta_\ast}}{4},  \frac{3 C_{\theta_\ast} \widehat{\beta}^2_1}{16 C^2_{\theta1^\ast}},  \frac{ 15 \widehat{\beta}_1^2}{16 \widehat{\beta}^2_2} \Big\}$. \newline
	$\bullet$ We prove that $\vertiii{(\mathbf{v}_h, q_h)}_h \lesssim \vertiii{(\mathbf{u}_h, p_h)}_h$. Concerning this, we use $\tau_0= (1+\theta_1^\ast)\max_{E\in\Omega_h} \tau_{1,E}$ and proceed as follows
	\begin{align}
		\vertiii{(\mathbf{v}_h, q_h)}_h^2 &= \mu_{min} \|\nabla (\mathbf{u}_h-\Theta \mathbf{w}_h)\|^2_{0, \Omega} + \|p_h\|^2 + \mathcal{L}_{1,h}((\mathbf{u}_h-\Theta \mathbf{w}_h), (\mathbf{u}_h-\Theta \mathbf{w}_h)) + \mathcal{L}_{2,h}(p_h,p_h) \nonumber \\
		& \lesssim  2\mu_{min} \|\nabla \mathbf{u}_h\|^2_{0, \Omega} + (2 \mu_{min} \Theta^2 + 1) \|p_h\|^2_{0, \Omega} +  \mathcal{L}_{1,h}(\mathbf{u}_h, \mathbf{u}_h) + \Theta^2 \mathcal{L}_{1,h}(\mathbf{w}_h, \mathbf{w}_h) \nonumber \\ & \qquad + 2 \Theta |\mathcal{L}_{1,h}(\mathbf{u}_h, \mathbf{w}_h)|+ \mathcal{L}_{2,h}(p_h,p_h) \nonumber \\
		& \lesssim  2\mu_{min} \|\nabla \mathbf{u}_h\|^2_{0, \Omega} + (2 \mu_{min} \Theta^2 + 1) \|p_h\|^2_{0, \Omega} + 2 \mathcal{L}_{1,h}(\mathbf{u}_h, \mathbf{u}_h) + 2\Theta^2  \mathcal{L}_{1,h}(\mathbf{w}_h, \mathbf{w}_h)  + \mathcal{L}_{2,h}(p_h,p_h) \nonumber \\
		& \lesssim  2\mu_{min} \|\nabla \mathbf{u}_h\|^2_{0, \Omega} + (2 \mu_{min} \Theta^2 + 2 \Theta^2 \tau_0 + 1) \|p_h\|^2_{0, \Omega} + 2 \mathcal{L}_{1,h}(\mathbf{u}_h, \mathbf{u}_h) + \mathcal{L}_{2,h}(p_h,p_h) \nonumber \\
		&  \lesssim  \max \{ 2, (2 \mu_{min} \Theta^2 + 2 \Theta^2 \tau_0 + 1)\} 
		\big[ \mu_{min} \|\nabla \mathbf{u}_h\|^2_{0, \Omega} + \|p_h\|^2_{0, \Omega} +  \mathcal{L}_{1,h}(\mathbf{u}_h, \mathbf{u}_h) + \mathcal{L}_{2,h}(p_h,p_h) \big] \nonumber \\
		& \lesssim \Lambda^2 \vertiii{(\mathbf{u}_h,p_h)}^2_h, \label{wellpos-7}
	\end{align}
	where $\Lambda := \sqrt{2 \mu_{min} \Theta^2 + 2 \Theta^2 \tau_0 + 2}$. Consequently, by merging \eqref{wellpos-6} and \eqref{wellpos-7}, we establish the validity of the estimate \eqref{wellpos-0} with $\widetilde{\beta}=\Lambda_0/\Lambda$. Moreover, using the continuity of the forms $A_{V,h}$ and $\mathcal{L}_h$, along with the estimate \eqref{wellpos-0}, guarantees the uniqueness of the solution to the problem  \eqref{dd_couple-1}. \newline
	$\bullet$ Additionally, employing \eqref{wellpos-0}, the Sobolev embedding theorem and Remark \ref{stabilityP}, we obtain
	\begin{align}
		\widetilde{\beta} \vertiii{(\mathbf{u}_h,p_h)}_h \vertiii{(\mathbf{v}_h,q_h)}_h &\lesssim   A_{V,h} [ \phi_h; (\mathbf{u}_h, p_h), (\mathbf{v}_h, q_h)] + \mathcal{L}_h[ (\mathbf{u}_h, p_h), (\mathbf{v}_h, q_h)]  \nonumber \\
		& \lesssim \alpha \sum_{E \in \Omega_h} \|\mathbf{f}\|_{0, E} \|\Pi_k^{0,E}\phi_h\|_{0,4,E} \| \boldsymbol{\Pi}^{0,E}_k\mathbf{v}_h\|_{0,4,E} \nonumber \\
		& \lesssim \alpha \sum_{E \in \Omega_h} \|\mathbf{f}\|_{0, E} \|\phi_h\|_{0,4,E} \| \mathbf{v}_h\|_{0,4,E} \nonumber \\
		&\lesssim \frac{\alpha C_q^2(1+C_P)^2}{\sqrt{\mu_{min}}} \|\mathbf{f}\|_{0, \Omega} \| \nabla \phi_h\|_{0,\Omega} \vertiii{(\mathbf{v}_h,q_h)}_h, 
	\end{align}
	the estimate \eqref{wellpos-s0} readily follows from the above analysis. 
\end{proof}

\begin{lemma}	\big(Well-posedness of \eqref{dd_couple-2}\big) \label{lm-ddM}
	Under the assumption \textbf{(A1)}, for any given $\widehat{\mathbf{u}}_h \in \mathbf{V}_h$,
	the discrete problem \eqref{dd_couple-2} has a unique solution $\widehat{\phi}_h \in \Sigma_h$ such that
	\begin{align}
		\vertiii{\mathbb M_h(\widehat{\mathbf{u}}_h)}_{\Sigma_h} :=\vertiii{\widehat{\phi}_h}_{\Sigma_h} \lesssim \frac{C_P} { C_\theta \sqrt{\kappa}} \|g\|_{0,\Omega}, \label{dd2-f0}
	\end{align}
	where $C_\theta:= \min\{1, \theta_{2\ast}\}$.
\end{lemma}
\begin{proof}
	For a given $\widehat{\mathbf{u}}_h \in \mathbf{V}_h$ and any $\psi_h \in \Sigma_h$, we investigate the coercivity of $A_{T,h} + \mathcal{L}_{3,h}$. Concerning this, we use the estimate \eqref{vem-b} and Remark \ref{estimate}:
	\begin{align}
		A_{T,h}(\widehat{\mathbf{u}}_h; \psi_h, \psi_h)&= a_{T,h}(\psi_h,\psi_h) + c^S_{T,h}(\widehat{\mathbf{u}}_h; \psi_h,\psi_h) \nonumber \\
		& \gtrsim  \sum_{E \in \Omega_h} \big[\kappa \|\boldsymbol{\Pi}^{0,E}_{k-1} \nabla \psi_h\|^2_{0,E} + \theta_{2\ast} \kappa \| \nabla (I - \Pi^{\nabla,E}_k) \psi_h\|_{0,E}^2 \big]  \nonumber \\
		& \gtrsim  \sum_{E \in \Omega_h} \kappa \min \{1, \theta_{2\ast} \} \big[ \|\boldsymbol{\Pi}^{0,E}_{k-1} \nabla \psi_h\|^2_{0,E} +  \| (\boldsymbol{I}-\boldsymbol{\Pi}^{0,E}_{k-1}) \nabla \psi_h\|_{0,E}^2 \big] \nonumber \\
		& \gtrsim   \kappa \min \{1, \theta_{2\ast} \}  \| \nabla \psi_h\|^2_{0,\Omega}. \label{dd2-1}
	\end{align}
	Combining \eqref{dd2-1} with $\mathcal{L}_{3,h}$, it gives
	\begin{align}
		A_{T,h}(\widehat{\mathbf{u}}_h; \psi_h, \psi_h) + \mathcal{L}_{3,h}(\psi_h,\psi_h) 
		& \gtrsim   C_\theta \vertiii{\psi_h}_{\Sigma_h}^2. \label{dd2-3} 
	\end{align}
	Thus, using the above analysis and the continuity of $A_{T,h}$ and $\mathcal{L}_{3,h}$, the uniqueness of the discrete solution to the problem \eqref{dd_couple-2} is obtained by employing the Lax-Milgram Lemma. Furthermore, it also satisfies
	\begin{align}
		C_\theta \vertiii{ \widehat{\phi}_h}_{\Sigma_h}^2 &\lesssim A_{T,h}(\widehat{\mathbf{u}}_h; \widehat{\phi}_h, \widehat{\phi}_h) + \mathcal{L}_{3,h}(\widehat{\phi}_h,\widehat{\phi}_h) = G_h(\widehat{\phi}_h)  \nonumber \\
		&\lesssim C_P \|g\|_{0, \Omega} \| \nabla \widehat{\phi}_h\|_{0,\Omega},  \label{dd2-4c}
	\end{align}
	which completes the proof of the result \eqref{dd2-f0}. 
\end{proof}

\subsection{Well-posedness of the stabilized coupled problem}
Hereafter, we establish the existence and uniqueness of the discrete solution to the stabilized virtual element problem \eqref{nvem}. To achieve this, we first derive the following results, which ensure that the assumptions of Brouwer’s fixed point theorem are satisfied. 

\begin{lemma} \label{Sh-diff}
	For any given ${\phi}_h, \widehat{{\phi}}_h \in \Sigma_h$, the following result holds:
	\begin{align}
		\vertiii{\mathbf{S}_h(\phi_h) - \mathbf{S}_h(\widehat{{\phi}}_h)}_h & \lesssim\dfrac{C_q(1+C_P)}{ \widetilde{\beta}\sqrt{ \kappa \mu_{min}}}\Big[ \alpha C_q(1+C_P)\|\mathbf{f}\|_{0, \Omega}  +   L_\mu C_{\mathbf{S}, \widehat{\mathbf{u}}_h} \,+   \theta_1^\ast  L_\mu \|\nabla \widehat{\mathbf{u}}_h \|_{0,3,\Omega} \Big] \vertiii{\phi_h - \widehat{{\phi}}_h}_{\Sigma_{h}}.		
		\label{shdiff}
	\end{align}
\end{lemma}

\begin{proof}
	For given $\phi_h, \widehat{{\phi}}_h \in \Sigma_h$, we define  $\mathbf{S}_h(\phi_h):=(\mathbf{u}_h,p_h)$ and $\mathbf{S}_h(\widehat{{\phi}}_h):= (\widehat{\mathbf{u}}_h, \widehat{{p}}_h)$. For any $(\mathbf{v}_h, q_h) \in \mathbf{V}_h \times Q_h$, we have
	\begin{align}
		A_{V,h}[\phi_h;& (\mathbf{u}_h,p_h)- (\widehat{\mathbf{u}}_h, \widehat{{p}}_h), (\mathbf{v}_h, q_h)] + \mathcal{L}_{h}[ (\mathbf{u}_h,p_h)- (\widehat{\mathbf{u}}_h, \widehat{{p}}_h), (\mathbf{v}_h, q_h)] \nonumber \\
		&= F_{h,\phi_h}(\mathbf{v}_h) - F_{h,\widehat{\phi}_h}(\mathbf{v}_h) + A_{V,h}[\widehat{{\phi}}_h; (\widehat{\mathbf{u}}_h, \widehat{{p}}_h), (\mathbf{v}_h, q_h)]- A_{V,h}[\phi_h; (\widehat{\mathbf{u}}_h, \widehat{{p}}_h), (\mathbf{v}_h, q_h)] \nonumber \\
		&=: A_{V,h,1} + A_{V,h,2}. \label{Shdiff-1}  		
	\end{align}
	Concerning the first term of \eqref{Shdiff-1}, we proceed as follows
	\begin{align}
		A_{V,h,1} &= \sum_{E \in \Omega_h} \big[ \big(\alpha \mathbf{f} \Pi^{0,E}_k \phi_h, \boldsymbol{\Pi}^{0,E}_k \mathbf{v}_h\big) - \big(\alpha \mathbf{f} \Pi^{0,E}_k \widehat{\phi}_h, \boldsymbol{\Pi}^{0,E}_k \mathbf{v}_h\big) \big] = \sum_{E \in \Omega_h} \big(\alpha \mathbf{f} \Pi^{0,E}_k (\phi_h - \widehat{\phi}_h), \boldsymbol{\Pi}^{0,E}_k \mathbf{v}_h\big) \nonumber \\
		& \lesssim  \sum_{E \in \Omega_h}  \alpha \|\mathbf{f}\|_{0,E} \| \Pi^{0,E}_k(\phi_h - \widehat{{\phi}}_h )\|_{0,4,E} \|\boldsymbol{\Pi}^{0,E}_k\mathbf{v}_h\|_{0,4,E} \nonumber \\
		& \lesssim  \dfrac{ \alpha C_q^2(1+C_P)^2}{\sqrt{\mu_{min}}} \|\mathbf{f}\|_{0, \Omega} \| \nabla(\phi_h - \widehat{{\phi}}_h )\|_{0,\Omega}  \vertiii{(\mathbf{v}_h, q_h)}_h. 
		\label{Shdiff-2}
	\end{align} 
	And 
	\begin{align}
		A_{V,h,2} &= a_{V,h}(\widehat{{\phi}}_h; \widehat{\mathbf{u}}_h, \mathbf{v}_h) - a_{V,h}(\phi_h; \widehat{\mathbf{u}}_h, \mathbf{v}_h) \nonumber \\
		&= \sum_{E \in \Omega_h} \Big[\big([ \mu( \Pi^{0,E}_k\widehat{{\phi}}_h) - \mu( \Pi^{0,E}_k \phi_h)] \boldsymbol{\Pi}^{0,E}_{k-1}\nabla \widehat{\mathbf{u}}_h, \boldsymbol{\Pi}^{0,E}_{k-1}\nabla \mathbf{v}_h \big) \nonumber \\ & \qquad + \big[ \mu( \Pi^{0,E}_0 \widehat{{\phi}}_h) - \mu( \Pi^{0,E}_0 \phi_h) \big] S^E_ \nabla \big((\mathbf{I}-\boldsymbol{\Pi}^{\nabla,E}_k) \mathbf{\widehat{\mathbf{u}}}_h,(\mathbf{I}-\boldsymbol{\Pi}^{\nabla,E}_k) \mathbf{v}_h \big)  \Big] \nonumber \\
		&=: a_{v,1} + a_{v,2}. \label{Shdiff-3}
	\end{align}
	We use the Lipschitz continuity of $\mu$, the Sobolev embeddings $H^1(E) \hookrightarrow L^6(E)$ and Remark \ref{stabilityP}:
	\begin{align}
		|a_{v,1}| &\lesssim \sum_{E \in \Omega_h} \| [\mu( \Pi^{0,E}_k\widehat{{\phi}}_h) - \mu( \Pi^{0,E}_k \phi_h)] \boldsymbol{\Pi}^{0,E}_{k-1}\nabla \widehat{\mathbf{u}}_h \|_{0,E} \| \boldsymbol{\Pi}^{0,E}_{k-1}\nabla \widehat{\mathbf{v}}_h\|_{0,E} \nonumber \\
		& \lesssim \sum_{E \in \Omega_h} L_\mu \|  \Pi^{0,E}_k(\widehat{{\phi}}_h -  \phi_h)\|_{0,6,E} \| \boldsymbol{\Pi}^{0,E}_{k-1}\nabla \widehat{\mathbf{u}}_h \|_{0,3,E} \|\nabla \mathbf{v}_h\|_{0,E} \nonumber \\
		& \lesssim \sum_{E \in \Omega_h} L_\mu \| \widehat{{\phi}}_h -  \phi_h\|_{0,6,E} \| \boldsymbol{\Pi}^{0,E}_{k-1} \nabla \widehat{\mathbf{u}}_h \|_{0,3,E} \|\nabla \mathbf{v}_h\|_{0,E} \nonumber \\  
		& \lesssim L_\mu \| \widehat{{\phi}}_h -  \phi_h\|_{0,6,\Omega} \| \boldsymbol{\Pi}^{0}_{k-1} \nabla \widehat{\mathbf{u}}_h \|_{0,3,\Omega}\|\nabla \mathbf{v}_h\|_{0,\Omega} \nonumber \\
		& \lesssim L_\mu C_q(1+C_P) C_{\mathbf{S}, \widehat{\mathbf{u}}_h} \| \nabla(\widehat{{\phi}}_h -  \phi_h)\|_{0,\Omega} \|\nabla \mathbf{v}_h\|_{0,\Omega}, \label{Shdiff-4}
	\end{align}
	where $C_{\mathbf{S}, \widehat{\mathbf{u}}_h}:= \| \boldsymbol{\Pi}^{0}_{k-1} \nabla \widehat{\mathbf{u}}_h \|_{0,3,\Omega}=\big(\sum_{E \in \Omega_h} \| \boldsymbol{\Pi}^{0,E}_{k-1} \nabla \widehat{\mathbf{u}}_h \|^3_{0,3,E}\big)^{1/3}$ is finite, as $\boldsymbol{\Pi}^{0,E}_{k-1} \nabla \widehat{\mathbf{u}}_h$ is a tensor polynomial. 
	\newline
	{	Employing the estimate \eqref{vem-a}, the inequality \eqref{inverse2} and Bramble--Hilbert Lemma \cite[ Lemma 4.3.8]{scott}, we obtain
		\begin{align}
			|a_{v,2}| &\lesssim \sum_{E \in \Omega_h } \theta_1^\ast L_\mu |  \Pi^{0,E}_0(\widehat{{\phi}}_h -  \phi_h)| \|\nabla(\mathbf{I} -  \boldsymbol{\Pi}^{\nabla,E}_{k}) \widehat{\mathbf{u}}_h \|_{0,E} \|\nabla (\mathbf{I} -  \boldsymbol{\Pi}^{\nabla,E}_{k})\mathbf{v}_h\|_{0,E} \nonumber \\
			&\lesssim \sum_{E \in \Omega_h } \theta_1^\ast L_\mu  h^{-1/3}_E \| \widehat{{\phi}}_h -  \phi_h\|_{0,6,E} \,h^{1/3}_E\|\nabla \widehat{\mathbf{u}}_h \|_{0,3,E} \|\nabla \mathbf{v}_h\|_{0,E} \nonumber \\
			&\lesssim \sum_{E \in \Omega_h }  \theta_1^\ast L_\mu \|(\widehat{{\phi}}_h -  \phi_h)\|_{0,6,E} \|\nabla \widehat{\mathbf{u}}_h \|_{0,3,E}  \|\nabla \mathbf{v}_h\|_{0,E} \nonumber \\
			& \lesssim \theta_1^\ast  L_\mu C_q(1+C_P) \|\nabla \widehat{\mathbf{u}}_h \|_{0,3,\Omega} \| \nabla(\widehat{{\phi}}_h -  \phi_h)\|_{0,\Omega} \|\nabla \mathbf{v}_h\|_{0,\Omega}. \label{Shdiff-5}
	\end{align}}
Substituting the estimates \eqref{Shdiff-4} and \eqref{Shdiff-5} in \eqref{Shdiff-3}, it gives
\begin{align}
	A_{V,h,2} \lesssim \dfrac{L_\mu C_q(1+C_P)}{\sqrt{\kappa \mu_{min}}} \big[  C_{\mathbf{S}, \widehat{\mathbf{u}}_h} +   \theta_1^\ast \|\nabla \widehat{\mathbf{u}}_h \|_{0,3,\Omega} \big] \vertiii{\phi_h - \widehat{{\phi}}_h}_{\Sigma_{h}} \vertiii{(\mathbf{v}_h,q)}_h. \label{Shdiff-6}
\end{align}
Using the discrete inf-sup condition \eqref{wellpos-0} and the estimates \eqref{Shdiff-2} and \eqref{Shdiff-6}, we obtain the result \eqref{shdiff}. 
\end{proof}

\begin{lemma} \label{Mh-diff}
For any given $ \mathbf{u}_h, \widehat{\mathbf{u}}_h \in \mathbf{V}_h$, the following result holds:
\begin{align}
\vertiii{\mathbb M (\mathbf{u}_h) - \mathbb M (\widehat{\mathbf{u}}_h)}_{\Sigma_h} 
\lesssim \dfrac{(1+C_P)^2 C_q^2C_P}{ C_\theta^2 \kappa^{3/2} \sqrt{\mu_{min}}} \| g\|_{0,\Omega} \vertiii{(\mathbf{u}_h - \widehat{\mathbf{u}}_h, p_h - \widehat{{p}}_h)}_h. \label{Mhdiff}
\end{align}
\end{lemma}
\begin{proof}
We begin with the coercivity property \eqref{dd2-4c} for $\phi_h-\widehat{{\phi}}_h \in \Sigma_h$:
\begin{align}
C_\theta \vertiii{\phi_h - \widehat{\phi}_h}_{\Sigma_h}^2 &\lesssim A_{T,h}(\mathbf{u}_h; \phi_h - \widehat{\phi}_h, \phi_h - \widehat{\phi}_h) + \mathcal{L}_{3,h}(\phi_h - \widehat{\phi}_h,\phi_h - \widehat{\phi}_h) \nonumber \\
&\lesssim G_h(\phi_h-\widehat{{\phi}}_h)  - A_{T,h}(\mathbf{u}_h; \widehat{\phi}_h, \phi_h - \widehat{\phi}_h) - \mathcal{L}_{3,h}( \widehat{\phi}_h,\phi_h - \widehat{\phi}_h) \nonumber \\
&\lesssim  A_{T,h}(\widehat{\mathbf{u}}_h; \widehat{\phi}_h, \phi_h - \widehat{\phi}_h)  - A_{T,h}(\mathbf{u}_h; \widehat{\phi}_h, \phi_h - \widehat{\phi}_h) \qquad \text{(using \eqref{dd_couple-2})} \nonumber \\
&\lesssim  c^S_{T,h}(\widehat{\mathbf{u}}_h; \widehat{\phi}_h, \phi_h - \widehat{\phi}_h)  - c^S_{T,h}(\mathbf{u}_h; \widehat{\phi}_h, \phi_h - \widehat{\phi}_h) \nonumber \\
&\lesssim \dfrac{1}{2} \sum_{E \in \Omega_h} \big[\big( \boldsymbol{\Pi}^{0,E}_k (\widehat{\mathbf{u}}_h - \mathbf{u}_h) \cdot \boldsymbol{\Pi}^{0,E}_{k-1} \nabla \widehat{\phi}_h, \Pi^{0,E}_k (\phi_h - \widehat{\phi}_h) \big) + \nonumber \\ & \qquad \qquad \big( \boldsymbol{\Pi}^{0,E}_k ( \mathbf{u}_h - \widehat{\mathbf{u}}_h) \cdot \boldsymbol{\Pi}^{0,E}_{k-1} \nabla (\phi_h - \widehat{\phi}_h), \Pi^{0,E}_k \widehat{\phi}_h \big) \big] \nonumber \\
& \lesssim  \sum_{E \in \Omega_h} \big[ \| \boldsymbol{\Pi}^{0,E}_k (\widehat{\mathbf{u}}_h - \mathbf{u}_h)\|_{0,4,E} \|\boldsymbol{\Pi}^{0,E}_{k-1} \nabla \widehat{\phi}_h\|_{0,E} \|\Pi^{0,E}_k (\phi_h - \widehat{\phi}_h)\|_{0,4,E} \,+  \nonumber \\ & \qquad \qquad \| \boldsymbol{\Pi}^{0,E}_k (\widehat{\mathbf{u}}_h - \mathbf{u}_h)\|_{0,4,E} \|\boldsymbol{\Pi}^{0,E}_{k-1} \nabla(\phi_h - \widehat{\phi}_h)\|_{0,E} \|\Pi^{0,E}_k \widehat{\phi}_h\|_{0,4,E}\big].  \nonumber 
\intertext{Using Remark \ref{stabilityP}, the Sobolev embedding theorem and \eqref{dd2-f0}, it gives}
C_\theta \vertiii{\phi_h - \widehat{\phi}_h}_{\Sigma_h}^2 & \lesssim   C_q^2 \| \widehat{\mathbf{u}}_h - \mathbf{u}_h\|_{1,\Omega} \big[\| \nabla \widehat{\phi}_h\|_{0,\Omega} \|\phi_h - \widehat{\phi}_h\|_{1,\Omega} +\| \widehat{\phi}_h\|_{1,\Omega} \|\nabla(\phi_h - \widehat{\phi}_h)\|_{0,\Omega} \big] \nonumber \\
& \lesssim \frac{(1 + C_P)^2 C_q^2 C_P}{C_\theta\kappa} \|g\|_{0,\Omega}\| \nabla ( \mathbf{u}_h - \widehat{\mathbf{u}}_h)\|_{0,\Omega}  \| \nabla (\phi_h - \widehat{\phi}_h)\|_{0,\Omega},
\end{align}
which completes the proof. 
\end{proof}

Hereafter, we define a ball $\widehat{\Sigma}_h$ such that
\begin{align}
\widehat{\Sigma}_h := \big\{ \phi_h \in \Sigma_h \,\, \text{such that} \,\,\, \vertiii{\phi_h}_{\Sigma_h} \leq \rho_0 \big\} \subset \Sigma_h, \label{dset}
\end{align}
where $\rho_0 :=\frac{C_P} { C_\theta \sqrt{\kappa}} \|g\|_{0,\Omega}$.
\begin{lemma}
For any $\phi_h \in \widehat{\Sigma}_h$, 
we have $\mathbb T_h(\phi_h) \in \widehat{\Sigma}_h$.
\end{lemma}

\begin{proof}
The proof follows from Lemma \ref{lm-ddM}. 
\end{proof}

\begin{lemma} \label{lm-tdiff}
For any $\phi_h, \widehat{{\phi}}_h \in \widehat{\Sigma}_h$, the discrete operator $\mathbb T_h$ satisfies the following:
\begin{align}
\vertiii{\mathbb{T}_h(\phi_h) - \mathbb{T}_h(\widehat{{\phi}}_h)}_{\Sigma_h} &\lesssim 	\ \dfrac{(1+C_P)^3C_q^3 \rho_0}{ \widetilde{\beta}C_\theta \kappa^{3/2} \mu_{min}} \Big[ \alpha C_q(1+C_P)\|\mathbf{f}\|_{0, \Omega}  +   L_\mu C_{\mathbf{S}, \widehat{\mathbf{u}}_h} \,+  \nonumber \\ & \qquad \qquad \qquad \qquad  \theta_1^\ast  L_\mu  \|\nabla \widehat{\mathbf{u}}_h \|_{0,3,\Omega} \Big] \vertiii{\phi_h - \widehat{{\phi}}_h}_{\Sigma_{h}}.
\label{Thdiff}
\end{align}
\end{lemma}	

\begin{proof}
The result \eqref{Thdiff} is evident from the definition of $\mathbb T_h$ and Lemmas \ref{Sh-diff} and \ref{Mh-diff}. 
\end{proof}

We now establish the existence of discrete solutions, as outlined below:
\begin{theorem} \label{d-exist}
Let $\widehat{\Sigma}_h$ be defined by \eqref{dset}. Then the stabilized virtual element problem \eqref{nvem} has at least one solution $(\mathbf{u}_h,p_h,\phi_h) \in \mathbf{V}_h \times Q_h \times \Sigma_h$ with $\phi_h \in \widehat{\Sigma}_h$, and there holds
\begin{align}
\vertiii{(\mathbf{u}_h,p_h)}_h \lesssim\dfrac{ \alpha C_q^2 (1+C_P)^2}{\beta_0 \sqrt{ \kappa \mu_{min}}} \|\mathbf{f}\|_{0, \Omega} \vertiii{\widehat{\phi}_h}_{\Sigma_h}  \qquad \text{and} \qquad 	\vertiii{\phi_h}_{\Sigma_h} \lesssim \frac{C_P} { C_\theta \sqrt{\kappa}} \|g\|_{0,\Omega}.  \label{th-f0}
\end{align}	
\end{theorem}

\begin{proof}
From Lemma \ref{lm-tdiff}, the  operator $\mathbb T_h: \widehat{\Sigma}_h \rightarrow \widehat{\Sigma}_h$ is continuous. Further, the Brouwer fixed point theorem ensures that $\mathbb T_h$ admits at least one fixed point. Additionally, \eqref{th-f0} follow directly from Lemmas \ref{wellposed-1} and \ref{lm-ddM}. 
\end{proof}

Next, we address the uniqueness of the discrete solution using the contraction mapping theorem:
\begin{theorem} \label{sunique}
Let $(\mathbf{u}_h,p_h,\phi_h) \in \mathbf{V}_h \times Q_h \times \Sigma_h$ with $\phi_h \in \widehat{\Sigma}_h$ be {any} solution of the stabilized problem \eqref{nvem}. {Furthermore, we assume that $\mathbf{u}_h \in [W^{1,3}(\Omega_h)]^2$ and the data are such that the following bound holds}
\begin{align}
\dfrac{(1+C_P)^3C_q^3 \rho_0}{ \widetilde{\beta}C_\theta \kappa^{3/2} \mu_{min}} \Big[ \alpha C_q(1+C_P)\|\mathbf{f}\|_{0, \Omega} + L_\mu C_{\mathbf{S}, {\mathbf{u}}_h} +   \theta_1^\ast  L_\mu  \|\nabla {\mathbf{u}}_h \|_{0,3,\Omega} \Big] <1,
\label{Thdiff2}
\end{align}
then, the stabilized coupled problem \eqref{nvem} has a unique solution.
\end{theorem}
\begin{proof}
The proof follows from Theorem \ref{d-exist} and \eqref{Thdiff2}.  
\end{proof}

\section{Error estimates} \label{sec:5}
In this section we derive the error estimates for the proposed virtual element problem \eqref{nvem} in the energy norm. We first introduce the following polynomial and interpolation approximations. 
\begin{lemma}
	\label{lemmaproj1} (\cite{scott})
	Under the assumption \textbf{(A1)}, for any $\psi \in H^s(E)$ defined on $E\in \Omega_h$, it holds that
	\begin{align}
		\|\psi -\Pi^{\nabla,E}_k \psi\|_{m,E} &\lesssim  h^{s-m}_E |\psi|_{s,E} \quad s, m \in \mathbb{N}, \,\,\, s \geq 1,\,\, \, m \leq s \leq k+1,\label{eqproj1}	 \\
		\|\psi - \Pi^{0,E}_k \,\psi \|_{m,E} &\lesssim h^{s-m}_E |\psi|_{s,E} \quad s, m \in \mathbb{N},\,\,\, m \leq s \leq k+1.
	\end{align}
\end{lemma}

\begin{lemma}	\label{lemmaproj2}
	(\cite{vem31}) Under the assumption \textbf{(A1)}, for any $\psi \in H^{s+1}(E)$ there exists $\psi_I \in V_h(E)$ for all $E\in \Omega_h$ such that
	\begin{equation}
		\| \psi - \psi_I\|_{0,E} + h_E |\psi - \psi_I|_{1,E} \leq C h^{1+s}_E \, \|\psi \|_{s+1,E},\quad 0 \leq s\leq k, \label{projlll2}
	\end{equation}
	where $C$ is a positive constant depending only on $k$ and $\gamma_0$.
\end{lemma}

Let $(\mathbf{u}_I, p_I,\phi_I) \in \mathbf{V}_h \times Q_h \times \Sigma_h$ denote the virtual interpolant of $(\mathbf{u}, p, \phi) \in \mathbf{V} \times Q \times \Sigma$. We further introduce the following notations, which will be employed in the subsequent analysis.
\begin{align*}
	e^\mathbf{u}_I := \mathbf{u}-\mathbf{u}_I, \quad e^\mathbf{u}_h := \mathbf{u}_h-\mathbf{u}_I, \quad 
	e^p_I := p-p_I, \quad e^p_h := p_h- p_I, \quad e^\phi_I:=\phi - \phi_I, \quad e^\phi_h:= \phi_h - \phi_I.
\end{align*}

\noindent \textbf{(A2) Regularity assumptions for error analysis:} We introduce the following conditions:
\begin{itemize}
	\item  The solution $(\mathbf{u}, p, \phi)$ satisfies $\mathbf{u} \in \mathbf{V} \cap [H^{k+1}(\Omega_h)]^2, \quad p \in Q \cap H^{k}(\Omega_h), \quad \phi \in \Sigma \cap H^{k+1}(\Omega_h)$;
	\item The loads $\mathbf{f}$ and $g$ satisfy $\mathbf{f} \in [H^{k}(\Omega_h)]^2$ , \quad $g \in H^{k}(\Omega_h)$;
	\item  The viscosity $\mu$ satisfies $\mu \in W^{k,\infty}(\Omega_h)$.
\end{itemize}
\begin{lemma} \label{lm-err1}
	{Under the assumptions \textbf{(A1)} and \textbf{(A2)}, let $(\mathbf{u}, p) \in \mathbf{V} \times Q$ and $(\mathbf{u}_h, p_h) \in \mathbf{V}_h \times Q_h$ be the solution of the problems \eqref{dvariation-1} and  \eqref{dd_couple-1}, respectively, with $\phi \in \widehat{\Sigma}$ and $\phi_h \in \widehat{\Sigma}_h$.} Then, the following estimate holds:
	\begin{align}
		\vertiii{(\mathbf{u}_h- \mathbf{u}_I, p_h-p_I)}_h &\lesssim  \frac{1}{\widetilde{\beta}}\Big[ \frac {\big( L_\mu {\|\nabla \mathbf{u}\|_{0,3,\Omega}} + \alpha \|\mathbf{f}\|_{0, \Omega} \big)}{ \sqrt{\mu_{min}}} \|\nabla(\phi_h - \phi_I)\|_{0,\Omega} + \frac{h^k}{\sqrt{\mu_{min}}} \big[ \alpha \big( \|\mathbf{f}\|_{0, \Omega} + \|\mathbf{f}\|_{k,\Omega}\big) \nonumber \\ & \qquad + L_\mu {\| \nabla \mathbf{u}\|_{0,3,\Omega}} \big] \|\phi\|_{k+1,\Omega} + h^k \big( h^{-1} \max_{E\in\Omega_h} \tau_{2,E}^{1/2} + \mu_{min}^{-1/2} \big) \|p\|_{k,\Omega} \, +  \nonumber \\ & \qquad h^k \big( 1 + \max_{E\in\Omega_h} \tau_{1,E}^{1/2}+ (\mu_{max} +  \max_{E\in\Omega_h}\|\mu\|_{k,\infty, E}) \mu_{min}^{-1/2} \big) \|\mathbf{u}\|_{k+1,\Omega}  \Big]. \label{lm-err1f}
	\end{align} 
\end{lemma}

\begin{proof}
	Let $(\mathbf{u}_I, p_I) \in \mathbf{V}_h \times Q_h$ denote the virtual interpolant of $(\mathbf{u}, p) \in \mathbf{V} \times Q$. For given $\phi_h \in \widehat{\Sigma}_h$, using the estimate \eqref{wellpos-0}, we infer
	\begin{align}
		\widetilde{\beta}\,	&\vertiii{(\mathbf{u}_h - \mathbf{u}_I, p_h-p_I)}_h \vertiii{(\mathbf{v}_h , q_h)}_h \nonumber \\ & \lesssim  A_{V,h}[ \phi_h; (\mathbf{u}_h - \mathbf{u}_I, p_h-p_I), (\mathbf{v}_h, q_h)] + \mathcal{L}_h[(\mathbf{u}_h - \mathbf{u}_I, p_h-p_I), (\mathbf{v}_h, q_h)] \nonumber \\
		& \lesssim |F_{h,\phi_h}(\mathbf{v}_h) - F_{\phi}(\mathbf{v}_h) | + |A_V[\phi;(\mathbf{u},p),(\mathbf{v}_h, q_h)] - A_{V,h} [\phi_h; ( \mathbf{u}_I, p_I), (\mathbf{v}_h, q_h)] | \nonumber \\ & \qquad + |\mathcal{L}_h [ ( \mathbf{u}_I, p_I), (\mathbf{v}_h, q_h)]| \nonumber \\
		&=: \eta_{F} + \eta_{A_V} + \eta_{\mathcal{L}_h}. \label{sterr-2} 
	\end{align}
	We estimate the above in the following steps: \newline
	\textbf{Step 1.} By adding and subtracting suitable terms, we obtain
	\begin{align}
		\eta_F&= \big|\sum_{E \in \Omega_h} \big[ \big(\alpha \mathbf{f} \Pi^{0,E}_k \phi_h, \boldsymbol{\Pi}^{0,E}_k \mathbf{v}_h \big) - \big( \alpha \mathbf{f} \phi,  \mathbf{v}_h \big) \big] \big| \nonumber \\
		&= \big| \sum_{E \in \Omega_h} \big[ \big(\alpha \mathbf{f} (\Pi^{0,E}_k \phi_h - \phi), \boldsymbol{\Pi}^{0,E}_k \mathbf{v}_h \big) + \big( \alpha \mathbf{f} \phi, \boldsymbol{\Pi}^{0,E}_k \mathbf{v}_h- \mathbf{v}_h \big) \big] \big| \nonumber \\
		&= \big| \sum_{E \in \Omega_h} \alpha \big[ \big( \mathbf{f} (\Pi^{0,E}_k \phi_h - \phi), \boldsymbol{\Pi}^{0,E}_k \mathbf{v}_h \big) + \big( \mathbf{f} \phi - \boldsymbol{\Pi}^{0,E}_k(\mathbf{f} \phi), \boldsymbol{\Pi}^{0,E}_k \mathbf{v}_h- \mathbf{v}_h \big) \big] \big|. \nonumber 
		\intertext{We use the stability of the projection operator in $L^2(E)$ and $L^4(E)$, the Sobolev embedding theorems and Lemma \ref{lemmaproj1}:}
		\eta_F &\lesssim \sum_{E \in \Omega_h} \alpha \big[ \|\mathbf{f}\|_{0,E}  \|\Pi^{0,E}_k \phi_h - \phi\|_{0,4,E} \|\boldsymbol{\Pi}^{0,E}_k \mathbf{v}_h \|_{0,4,E} + \| \mathbf{f} \phi - \boldsymbol{\Pi}^{0,E}_k(\mathbf{f} \phi)\|_{0,E} \|\boldsymbol{\Pi}^{0,E}_k \mathbf{v}_h- \mathbf{v}_h \|_{0,E} \big]	\nonumber \\
		&\lesssim \sum_{E \in \Omega_h} \alpha \big[ \|\mathbf{f}\|_{0,E} \big( \|\Pi^{0,E}_k (\phi - \phi_h)\|_{0,4,E} +  \| \phi- \Pi^{0,E}_k \phi \|_{0,4,E} \big)\|\boldsymbol{\Pi}^{0,E}_k \mathbf{v}_h \|_{0,4,E} + h^{k}_E |\mathbf{f} \phi|_{k-1,E} \| \nabla \mathbf{v}_h \|_{0,E} \big]	\nonumber \\
		&\lesssim \sum_{E \in \Omega_h} \alpha \big[ \|\mathbf{f}\|_{0,E} \big( \|\phi - \phi_h\|_{0,4,E} + h_E^k | \phi |_{k,4,E} \big)\| \mathbf{v}_h \|_{0,4,E} +  h^{k}_E |\mathbf{f}|_{k-1,4,E}  |\phi |_{k-1,4,E} \| \nabla \mathbf{v}_h \|_{0,E} \big]	\nonumber \\
		&\lesssim \alpha \big[ \|\mathbf{f}\|_{0,\Omega} \big( \| \phi - \phi_h\|_{0,4,\Omega} + h^k \| \phi \|_{k,4,\Omega} \big)\|  \mathbf{v}_h \|_{0,4,\Omega} + h^{k}  |\mathbf{f}|_{k-1,4,\Omega}  |\phi |_{k-1,4,\Omega} \| \nabla \mathbf{v}_h \|_{0,\Omega} \big]	\nonumber \\
		&\lesssim \alpha \big[ \|\mathbf{f}\|_{0,\Omega} \big( \| \phi - \phi_h\|_{1,\Omega} + h^k \| \phi \|_{k+1,\Omega} \big)\|  \mathbf{v}_h \|_{1,\Omega} +  h^k \|\mathbf{f}\|_{k,\Omega}  \|\phi \|_{k,\Omega} \| \nabla \mathbf{v}_h \|_{0,\Omega} \big]	\nonumber \\
		&\lesssim \frac{\alpha}{\sqrt{\mu_{min}}} \big[  \|\mathbf{f}\|_{0,\Omega} \big( h^k \|\phi\|_{k+1,\Omega} + \|\nabla (\phi_h - \phi_I)\|_{0,\Omega} \big) + h^{k} \|\mathbf{f}\|_{k,\Omega}  \|\phi \|_{k+1,\Omega}  \big] \vertiii{(\mathbf{v}_h,q_h)}_{h}. \label{sterr-3}
	\end{align}
	
	\noindent \textbf{Step 2.} By applying the definitions of the forms $A_V$ and $A_{V,h}$, we arrive at
	\begin{align}
		\eta_{A_V} & \lesssim | a_V(\phi; \mathbf{u}, \mathbf{v}_h) - a_{V,h}(\phi_h; \mathbf{u}_I, \mathbf{v}_h) | + | b(\mathbf{v}_h, p) - b_h(\mathbf{v}_h, p_I) | +  | b(\mathbf{u},q_h) - b_h(\mathbf{u}_I, q_h) | \nonumber \\ 
		& =:	\eta_{A_V,a} + \eta_{A_V,b1} + \eta_{A_V,b2}. \label{sterr-4}
	\end{align}
	$\bullet$ Concerning $\eta_{A_V,a}$, we proceed as follows
	\begin{align}
		\eta_{A_V,a} &= \big|\sum_{E \in \Omega_h} \big[ \big( \mu(\phi)\varepsilon(\mathbf{u}), \varepsilon(\mathbf{v}_h) \big) -  \big( \mu(\Pi^{0,E}_k \phi_h) \boldsymbol{\Pi}^{0,E}_{k-1} \varepsilon(\mathbf{u}_I), \boldsymbol{\Pi}^{0,E}_{k-1} \varepsilon(\mathbf{v}_h)  \big) \nonumber \\ & \qquad - \mu(\Pi^{0,E}_0 \phi_h) S^E_\nabla \big( (\mathbf{I} - \boldsymbol{\Pi}^{\nabla,E}_k) \mathbf{u}_I, (\mathbf{I} - \boldsymbol{\Pi}^{\nabla,E}_k) \mathbf{v}_h\big) \big] \big| \nonumber \\
		&= \big| \sum_{E \in \Omega_h} \big[ \big( \mu(\phi) (\varepsilon(\mathbf{u}) - \boldsymbol{\Pi}^{0,E}_{k-1} \varepsilon(\mathbf{u})), \varepsilon(\mathbf{v}_h) \big) + \big( \mu(\phi) \boldsymbol{\Pi}^{0,E}_{k-1} \varepsilon(\mathbf{u}), \varepsilon(\mathbf{v}_h) -   \boldsymbol{\Pi}^{0,E}_{k-1} \varepsilon(\mathbf{v}_h)  \big) + \nonumber \\ & \qquad \big( (\mu(\phi)- \mu(\Pi^{0,E}_k \phi_h)) \boldsymbol{\Pi}^{0,E}_{k-1} \varepsilon(\mathbf{u}), \boldsymbol{\Pi}^{0,E}_{k-1} \varepsilon(\mathbf{v}_h)\big) + \big( \mu(\Pi^{0,E}_k \phi_h) \boldsymbol{\Pi}^{0,E}_{k-1} (\varepsilon(\mathbf{u})-\varepsilon(\mathbf{u}_I)), \boldsymbol{\Pi}^{0,E}_{k-1} \varepsilon(\mathbf{v}_h)\big)  \nonumber \\ & \qquad  -\mu(\Pi^{0,E}_0 \phi_h) S^E_\nabla \big( (\mathbf{I} - \boldsymbol{\Pi}^{\nabla,E}_k) \mathbf{u}_I, (\mathbf{I} - \boldsymbol{\Pi}^{\nabla,E}_k) \mathbf{v}_h\big) \big] \big| \nonumber \\
		& =: |a_1 + a_2 + a_3 + a_4 +a_5|. \label{sterr-5}
	\end{align}
	Applying the Cauchy-Schwarz inequality, the assumption \eqref{mu-reg} and Lemma \ref{lemmaproj1}, we infer
	\begin{align}
		|a_1| &\lesssim \sum_{E \in \Omega_h }  \mu_{max} \|\varepsilon(\mathbf{u}) - \boldsymbol{\Pi}^{0,E}_{k-1} \varepsilon(\mathbf{u})\|_{0,E} \|\varepsilon(\mathbf{v}_h) \|_{0,E} \nonumber \\
		&  \lesssim \sum_{E \in \Omega_h }  \frac{h^k_E \mu_{max}}{\sqrt{\mu_{min}}}  \| \mathbf{u} \|_{k+1,E} \vertiii{(\mathbf{v}_h, q_h)}_{h,E}. \label{sterr-6}
	\end{align} 
	Concerning $a_2$, we use the definition of $\boldsymbol{\Pi}^{0,E}_{k-1}$ as follows:
	\begin{align*}
		a_2=  \sum_{E \in \Omega_h} \Big( \mu(\phi) \boldsymbol{\Pi}^{0,E}_{k-1} \varepsilon(\mathbf{u}) - \boldsymbol{\Pi}^{0,E}_{k-1}(\mu(\phi) \varepsilon(\mathbf{u})), \varepsilon(\mathbf{v}_h)  - \boldsymbol{\Pi}^{0,E}_{k-1} \varepsilon(\mathbf{v}_h) \Big).
	\end{align*}
	Using the triangle inequality, the assumption \eqref{mu-reg} and Lemmas \ref{lemmaproj1} and \ref{lemmaproj2}, we obtain
	\begin{align}
		|a_2| &\lesssim \sum_{E \in \Omega_h} \| \mu(\phi) \boldsymbol{\Pi}^{0,E}_{k-1} \varepsilon(\mathbf{u}) - \boldsymbol{\Pi}^{0,E}_{k-1}(\mu(\phi) \varepsilon(\mathbf{u}))\|_{0,E} \|\varepsilon(\mathbf{v}_h)  - \boldsymbol{\Pi}^{0,E}_{k-1} \varepsilon(\mathbf{v}_h) \|_{0,E} \nonumber \\
		&\lesssim \sum_{E \in \Omega_h} \big[ \| \mu(\phi) \boldsymbol{\Pi}^{0,E}_{k-1} \varepsilon(\mathbf{u}) - \mu(\phi) \varepsilon(\mathbf{u})\|_{0,E} + \|\mu(\phi) \varepsilon(\mathbf{u})  - \boldsymbol{\Pi}^{0,E}_{k-1}(\mu(\phi) \varepsilon(\mathbf{u}))\|_{0,E}   \big] \|\varepsilon(\mathbf{v}_h) \|_{0,E} \nonumber \\
		&\lesssim \sum_{E \in \Omega_h} \big[ \mu_{max} h^k_E \| \mathbf{u}\|_{k+1,E} + h^k_E |\mu(\phi) \varepsilon(\mathbf{u})|_{k,E} \big] \|\nabla \mathbf{v}_h \|_{0,E} \nonumber \\
		&\lesssim \sum_{E \in \Omega_h} \frac{h^k_E}{\sqrt{\mu_{min}}} \big[ \mu_{max}  \| \mathbf{u}\|_{k+1,E} +  \|\mu \|_{k,\infty,E} \|\mathbf{u} \|_{k+1,E} \big] \vertiii{(\mathbf{v}_h,q_h)}_{h,E}. \label{sterr-7}
	\end{align}
	Employing the Lipschitz continuity of $\mu$ and Remark \ref{stabilityP}, it gives
	\begin{align}
		|a_3| &\lesssim \sum_{E \in \Omega_h } \|\mu(\phi)- \mu(\Pi^{0,E}_k \phi_h)\|_{0,6,E} \|\boldsymbol{\Pi}^{0,E}_{k-1} \varepsilon(\mathbf{u})\|_{0,3,E} \|\boldsymbol{\Pi}^{0,E}_{k-1} \varepsilon(\mathbf{v}_h) \|_{0,E} \nonumber \\ 
		&\lesssim \sum_{E \in \Omega_h } L_\mu \|\phi- \Pi^{0,E}_k \phi_h\|_{0,6,E} \| \varepsilon(\mathbf{u})\|_{0,3,E} \| \varepsilon(\mathbf{v}_h) \|_{0,E} \nonumber \\ 
		&\lesssim \sum_{E \in \Omega_h } L_\mu  \big[ \|\phi- \Pi^{0,E}_k \phi\|_{0,6,E} + \| \Pi^{0,E}_k (\phi - \phi_h)\|_{0,6,E} \big] \| \nabla \mathbf{u}\|_{0,3,E} \| \nabla \mathbf{v}_h \|_{0,E} \nonumber \\
		&\lesssim \sum_{E \in \Omega_h } L_\mu  \big[ h_E^{k}|\phi|_{k,6,E} + \| \phi - \phi_h\|_{0,6,E} \big] \| \nabla \mathbf{u}\|_{0,3,E}\| \nabla \mathbf{v}_h \|_{0,E}. \nonumber 
		\intertext{Applying the H$\ddot{\text{o}}$lder inequality and the Sobolev embedding theorem $H^{k+1}(\Omega) \hookrightarrow W^{k}_4(\Omega)$, yields}
		|a_3| & \lesssim  L_\mu  \big[ h^{k} |\phi|_{k,6,\Omega} + \| \phi - \phi_h\|_{0,6,\Omega} \big] \| \nabla \mathbf{u}\|_{0,3,\Omega} \| \nabla \mathbf{v}_h \|_{0,\Omega} \nonumber  \\
		& \lesssim  L_\mu  \big[ h^{k}\|\phi\|_{k+1,\Omega} + \| \phi - \phi_h\|_{1,\Omega} \big] \| \nabla \mathbf{u}\|_{0,3,\Omega} \| \nabla \mathbf{v}_h \|_{0,\Omega} \nonumber  \\
		& \lesssim \frac{L_\mu}{\sqrt{\mu_{min}}} \big[ h^{k} \|\phi\|_{k+1,\Omega} + \| \nabla (\phi_h - \phi_I)\|_{0,\Omega} \big]\| \nabla \mathbf{u}\|_{0,3,\Omega} \vertiii{(\mathbf{v}_h, q_h)}_{h}. \label{sterr-8}
	\end{align}
	Following the estimation of $a_1$ and Lemma \ref{lemmaproj2}, we have
	\begin{align}
		|a_4| &\lesssim \sum_{E \in \Omega_h }  \mu_{max} \| \boldsymbol{\Pi}^{0,E}_{k-1}(\varepsilon(\mathbf{u}) - \varepsilon(\mathbf{u}_I))\|_{0,E} \|\boldsymbol{\Pi}^{0,E}_{k-1}\varepsilon(\mathbf{v}_h) \|_{0,E} \nonumber \\
		&  \lesssim \sum_{E \in \Omega_h }  \frac{h^k_E \mu_{max}}{\sqrt{\mu_{min}}}  \| \mathbf{u} \|_{k+1,E} \vertiii{(\mathbf{v}_h, q_h)}_{h,E}. \label{sterr-9a}
	\end{align} 
	Concerning $a_5$, we use \eqref{vem-a}, Lemma \ref{inverse}, the stability of the projectors and Lemmas \ref{lemmaproj1} and \ref{lemmaproj2}:
	\begin{align}
		|a_5| &\lesssim \sum_{E \in \Omega_h} \theta_1^\ast \mu_{max} \|\nabla (\mathbf{I} - \boldsymbol{\Pi}^{\nabla,E}_k) \mathbf{u}_I\|_{0,E} \|\nabla (\mathbf{I} - \boldsymbol{\Pi}^{\nabla,E}_k) \mathbf{v}_h\|_{0,E}  \nonumber \\
		&\lesssim \sum_{E \in \Omega_h} \theta_1^\ast \mu_{max} \big[\|\nabla (\mathbf{u} - \mathbf{u}_I ) \|_{0,E} + \|\nabla (\mathbf{I} - \boldsymbol{\Pi}^{\nabla,E}_k) \mathbf{u}\|_{0,E} + \|\nabla \boldsymbol{\Pi}^{\nabla,E}_k (\mathbf{u} - \mathbf{u}_I)\|_{0,E} \big] \|\nabla \mathbf{v}_h\|_{0,E}  \nonumber \\
		&\lesssim \sum_{E \in \Omega_h} \theta_1^\ast \mu_{max} \big[h^k_E \|\mathbf{u} \|_{k+1,E}  + h^{-1}_E \|\boldsymbol{\Pi}^{\nabla,E}_k (\mathbf{u} - \mathbf{u}_I)\|_{0,E} \big] \|\nabla \mathbf{v}_h\|_{0,E}  \nonumber \\
		&\lesssim \sum_{E \in \Omega_h} \frac{h^k_E \mu_{max}}{\sqrt{\mu_{min}}} \|\mathbf{u}\|_{k+1,E} \vertiii{(\mathbf{v}_h,q_h)}_{h,E}. \label{sterr-9}
	\end{align}
	$\bullet$ Concerning $\eta_{A_V,b1}$, we utilize the definition of $\Pi^{0,E}_{k-1}$ and add and subtract appropriate terms as follows:
	\begin{align}
		\eta_{A_V,b1} &= \big| \sum_{E \in \Omega_h} \big[ \big( \nabla \cdot \mathbf{v}_h, p \big) - \big( {\Pi}^{0,E}_{k-1}\nabla \cdot \mathbf{v}_h, \Pi^{0,E}_k p_I\big)  \big] \big| = \big| \sum_{E \in \Omega_h} \big[ \big( \nabla \cdot \mathbf{v}_h, p \big) - \big( \nabla \cdot \mathbf{v}_h, \Pi^{0,E}_{k-1} p_I\big)  \big] \big| \nonumber \\ 
		& = \big| \sum_{E \in \Omega_h} \big[ \big( \nabla \cdot \mathbf{v}_h, p - \Pi^{0,E}_{k-1} p \big) + \big( \nabla \cdot \mathbf{v}_h, \Pi^{0,E}_{k-1} (p-p_I)\big)  \big] \big| \nonumber \\
		&\lesssim \sum_{E \in \Omega_h}  \big[ \| \nabla \mathbf{v}_h\|_{0,E} \|p-\Pi^{0,E}_{k-1} p\|_{0,E} + \|  \nabla \mathbf{v}_h\|_{0,E} \|\Pi^{0,E}_{k-1} (p-p_I)\|_{0,E}  \big] \nonumber \\
		& \lesssim \sum_{E \in \Omega_h } \frac{h^k_E}{\sqrt{\mu_{min}}} \|p\|_{k,E} \vertiii{(\mathbf{v}_h, q_h)}_{h,E}. \label{sterr-10}
	\end{align}
	$\bullet$ Following the estimation of $\eta_{A_V,b1}$ and Lemma \ref{inverse}, it holds that
	\begin{align}
		\eta_{A_V,b2} 
		&= \big|  \sum_{E \in \Omega_h} \big[ \big( \nabla \cdot \mathbf{u} - \nabla \cdot \mathbf{u}_I, q_h\big) + \big( \nabla \cdot \mathbf{u}_I, q_h-\Pi^{0,E}_{k-1} q_h  \big) \big] \big|  \nonumber \\
		&= \big| \sum_{E \in \Omega_h} \big[ \big( \nabla \cdot \mathbf{u} - \nabla \cdot \mathbf{u}_I, q_h\big) + \big( \nabla \cdot \mathbf{u}_I - \Pi^{0,E}_{k-1} \nabla \cdot \mathbf{u}_I, q_h-\Pi^{0,E}_{k-1} q_h \big) \big] \big| \nonumber \\
		& \lesssim \sum_{E \in \Omega_h} \big[ h^k_E \| \mathbf{u}\|_{k+1,E} \|q_h\|_{0,E} + h^{k+1}_E \|\mathbf{u}\|_{k+1,E} \| \nabla q_h\|_{0,E} \big] \nonumber \\
		& \lesssim \sum_{E \in \Omega_h}  h^k_E \| \mathbf{u}\|_{k+1,E} \vertiii{(\mathbf{v}_h, q_h)}_{h,E} . \label{sterr-11}
	\end{align}
	Adding the estimates \eqref{sterr-6}--\eqref{sterr-11}, it satisfies
	\begin{align}
		\eta_{A_V} & \lesssim \Big[h^k \big( 1 + \frac{\mu_{max} + \max_{E\in\Omega_h} \|\mu\|_{k,\infty, E}}{\sqrt{\mu_{min}}}\big) \|u\|_{k+1,\Omega} + \frac{h^k}{\sqrt{\mu_{min}}} \|p\|_{k,\Omega}\, + \nonumber \\ & \qquad  \frac{L_\mu}{\sqrt{\mu_{min}}} \big( h^{k}\|\phi\|_{k+1,\Omega} + \| \nabla (\phi_h - \phi_I)\|_{0,\Omega} \big) {\| \nabla \mathbf{u}\|_{0,3,\Omega}} \Big]\vertiii{(\mathbf{v}_h, q_h)}_h. \label{eta-Av}
	\end{align}
	\textbf{Step 3.} Concerning the third term of \eqref{sterr-2}, we proceed as follows
	\begin{align}
		\eta_{\mathcal{L}_h} &= |\mathcal{L}_h[(\mathbf{u}_I, p_I), (\mathbf{v}_h, q_h)]| \lesssim   |\mathcal{L}_{1,h}(\mathbf{u}_I, \mathbf{v}_h)| + |\mathcal{L}_{2,h}(p_I, q_h)| \nonumber \\
		&=: \eta_{\mathcal{L}_{h1}} + \eta_{\mathcal{L}_{h2}}. \label{sterr-12}
	\end{align} 
	$\bullet$  We use the triangle inequality and Lemmas \ref{lemmaproj1} and \ref{lemmaproj2}:
	\begin{align}
		\eta_{\mathcal{L}_{h1}} &= \sum_{E \in \Omega_h} \tau_{1,E} \big[ \big(\widehat{r}_h(\nabla \cdot\mathbf{u}_I),  \widehat{r}_h(\nabla \cdot \mathbf{v}_h)\big) + S^E_\nabla \big((\mathbf{I}-  \boldsymbol{\Pi}^{\nabla,E}_{k}) \mathbf{u}_I, (\mathbf{I}-  \boldsymbol{\Pi}^{\nabla,E}_{k}) \mathbf{v}_h\big) \big] \nonumber \\
		&\lesssim  \sum_{E \in \Omega_h} \tau^{1/2}_{1,E} \big[ \| \widehat{r}_h(\nabla \cdot e^\mathbf{u}_I)\|_{0,E} + \|\widehat{r}_h(\nabla \cdot \mathbf{u})\|_{0,E} + \theta_1^\ast \|\nabla(\mathbf{I}-  \boldsymbol{\Pi}^{\nabla,E}_{k}) \mathbf{u}_I\|_{0,E} \big] \vertiii{(\mathbf{v}_h, q_h)}_{h,E}  \nonumber \\
		&\lesssim \sum_{E \in \Omega_h} \tau^{1/2}_{1,E} \big[ \|\nabla e^\mathbf{u}_I\|_{0,E} + \|(\mathbf{I} -  \boldsymbol{\Pi}^{0,E}_k)\nabla \cdot \mathbf{u}\|_{0,E} + \|(\mathbf{I} -  \boldsymbol{\Pi}^{0,E}_{k-1}) \nabla \cdot \mathbf{u}\|_{0,E} \,+ \nonumber \\ &\qquad \qquad \quad \theta_1^\ast h^k_E \| \mathbf{u}\|_{k+1,E} \big]\vertiii{(\mathbf{v}_h, q_h)}_{h,E}\nonumber \\
		&\lesssim \sum_{E \in \Omega_h} h^k_E \tau^{1/2}_{1,E}  \| \mathbf{u}\|_{k+1,E} \vertiii{(\mathbf{v}_h, q_h)}_{h,E}. \label{sterr-13}
	\end{align}
	$\bullet$  Following the estimation of $\eta_{\mathcal{L}_{h1}}$, we infer  
	\begin{align}
		\eta_{\mathcal{L}_{h2}} &\lesssim \sum_{E \in \Omega_h} \tau^{1/2}_{2,E} h^{k-1}_E \|p\|_{k,E} \vertiii{(\mathbf{v}_h, q_h)}_{h,E}. \label{sterr-14}    
	\end{align}
	Combining the results \eqref{sterr-13} and \eqref{sterr-14}, it holds that
	\begin{align}
		\eta_{\mathcal{L}_h} \lesssim h^k \big( \max_{E\in\Omega_h} \tau_{1,E}^{1/2} \| \mathbf{u}\|_{k+1,\Omega} + h^{-1} \max_{E\in\Omega_h}  \tau_{2,E}^{1/2} \|p\|_{k,\Omega}\big)\vertiii{(\mathbf{v}_h, q_h)}_{h}. \label{eta-lh}
	\end{align}
	Finally, combining the results \eqref{sterr-3}, \eqref{eta-Av} and \eqref{eta-lh}, we readily obtain the estimate \eqref{lm-err1f}. 
\end{proof}
\begin{remark}
	Form \eqref{eta-lh}, it is notable that the stabilization parameter $ \tau_{2,E}$ must be chosen as $\tau_{2,E} \sim h^2_E $ to achieve optimal convergence estimates in the energy norm, this choice is once again similar to Remark \ref{inf11}. Furthermore, the error estimate \eqref{lm-err1f} becomes quasi-uniform in the energy norm as $\mu \rightarrow 0$, which is similar to \cite{mfem43,mfem31,vem028}. The investigation of uniform error estimates in the energy norm will be the focus of future work.
	
\end{remark}
\begin{lemma} \label{lm-err2}
	{	Under the assumptions \textbf{(A1)} and \textbf{(A2)}, let  $ \phi \in  \Sigma$ and $ \phi_h \in \Sigma_h$ be the solution of the problems \eqref{dvariation-2} and \eqref{dd_couple-2}, respectively, for given $\mathbf{u} \in \mathbf{V}$ and $\mathbf{u}_h \in \mathbf{V}_h$ with $\phi \in \widehat{\Sigma}$ and $\phi_h \in \widehat{\Sigma}_h$.} Then, the following estimate holds:
	\begin{align} 
		\vertiii{\phi_h-\phi_I}_{\Sigma_h} & \lesssim \dfrac{1}{C_\theta }\Big[ h^k \big(\sqrt{\kappa} + \max_{E\in\Omega_h} \tau_{3,E}^{1/2}\big) \| \phi \|_{k+1,\Omega}  + \frac{h^k}{\sqrt{\kappa}} \big[ h \|g\|_{k,\Omega} + \big( \|\mathbf{u}\|_{k,\Omega}  + \| \nabla \mathbf{u}\|_{0,\Omega}\big) \|\phi\|_{k+1,\Omega}  \nonumber\\ & + \| \nabla \phi\|_{0,\Omega} \| \mathbf{u} \|_{k+1,\Omega}  \big] + \frac{1}{\sqrt{\kappa}} \| \nabla \phi\|_{0,\Omega} \|\nabla (\mathbf{u}_h - \mathbf{u}_I) \|_{0,\Omega} \Big]. \label{lm-err2f} 
	\end{align}
\end{lemma}

\begin{proof}
	Let $\phi_I \in \Sigma_h$ represent the virtual interpolant  $\phi \in \Sigma$.	For given $\mathbf{u}_h \in \mathbf{V}_h$, we proceed by following the coercivity of $A_{T,h}$
	\begin{align}
		C_\theta \vertiii{\phi_h-\phi_I}^2_{\Sigma_h} &\lesssim A_{T,h}\big( \mathbf{u}_h; \phi_h-\phi_I, e^\phi_h\big) + \mathcal{L}_{3,h}\big( \phi_h-\phi_I, e^\phi_h\big) \nonumber \\
		& \lesssim |G_h(e^\phi_h) - G(e^\phi_h) |+ |A_{T}\big( \mathbf{u}; \phi, e^\phi_h\big) -  A_{T,h}\big( \mathbf{u}_h; \phi_I, e^\phi_h\big)| + |\mathcal{L}_{3,h}\big( \phi_I, e^\phi_h\big) | \nonumber \\
		&=: \zeta_{G} + \zeta_{A_T} + \zeta_{\mathcal{L}_3}. \label{temp-3}
	\end{align}
	We will estimate the above $\zeta$'s in the following steps: \newline
	\textbf{Step 1.} Using the regularity condition on $g$ and Lemma \ref{lemmaproj1}, it holds that
	\begin{align}
		\zeta_{G} &= \big| \sum_{E \in \Omega_h } \big[\big(g, \Pi^{0,E}_k e^\phi_h \big) - \big( g, e^\phi_h\big) \big]\big| =  \big| \sum_{E \in \Omega_h } \big(g - \Pi^{0,E}_k g, \Pi^{0,E}_k e^\phi_h - e^\phi_h\big)\big| \nonumber \\
		& \lesssim  \sum_{E \in \Omega_h } \|g - \Pi^{0,E}_k g\|_{0,E} \|e^\phi_h - \Pi^{0,E}_k e^\phi_h\|_{0,E} \nonumber \\
		& \lesssim \sum_{E \in \Omega_h } \frac{h_E^{k+1}}{\sqrt{\kappa}} \|g\|_{k,E} \vertiii{e^\phi_h}_{\Sigma_h,E}
		\lesssim \frac{h^{k+1}}{\sqrt{\kappa}} \|g\|_{k,\Omega} \vertiii{e^\phi_h}_{\Sigma_h}. \label{temp-4}
	\end{align}
	\textbf{Step 2.} Concerning $\zeta_{A_T}$, we proceed as follows
	\begin{align}
		\zeta_{A_T} &\lesssim | a_T(\phi,e^\phi_h) - a_{T,h}(\phi_I, e^\phi_h)|	+ |c^S_T(\mathbf{u}; \phi, e^\phi_h) - c^S_{T,h}(\mathbf{u}_h; \phi_I, e^\phi_h)| \nonumber \\
		&=: \zeta_{A_T,a} + \zeta_{A_T,c}. \label{temp-5}
	\end{align}
	$\bullet$ We use the definition of $\boldsymbol{\Pi}^{0,E}_{k-1}$, the estimate \eqref{vem-a} and Lemma \ref{lemmaproj1}:
	\begin{align}
		\zeta_{A_T,a} &= \big| \sum_{E \in \Omega_h } \big[ \big( \kappa \nabla \phi, \nabla e^\phi_h\big) - \big( \kappa \boldsymbol{\Pi}^{0,E}_{k-1} \nabla \phi_I, \boldsymbol{\Pi}^{0,E}_{k-1} \nabla e^\phi_h \big) - \kappa S_1\big( (I-\Pi^{\nabla,E}_k) \phi_I, (I- \Pi^{\nabla,E}_k) e^\phi_h \big) \big] \big| \nonumber \\
		& = \Big| \sum_{E \in \Omega_h } \big[ \kappa \big( \nabla \phi - \boldsymbol{\Pi}^{0,E}_{k-1} \nabla \phi_I, \nabla e^\phi_h\big)  - \kappa S_1\big( (I-\Pi^{\nabla,E}_k) \phi_I, (I- \Pi^{\nabla,E}_k) e^\phi_h \big) \big] \Big|  \nonumber \\ 
		& \lesssim \sum_{E \in \Omega_h } \big[ \kappa \| \nabla \phi - \boldsymbol{\Pi}^{0,E}_{k-1} \nabla \phi_I\|_{0,E} \|\nabla e^\phi_h \|_{0,E}  + \kappa \theta_2^\ast \|\nabla(I-\Pi^{\nabla,E}_k) \phi_I\|_{0,E} \| \nabla(I- \Pi^{\nabla,E}_k) e^\phi_h \|_{0,E} \big]  \nonumber \\ 
		& \lesssim \sum_{E \in \Omega_h } \big[ \kappa h^k_E \| \phi \|_{k+1,E}\|\nabla e^\phi_h \|_{0,E}  + \kappa \theta_2^\ast h^k_E \|\phi\|_{k+1,E} \| \nabla e^\phi_h \|_{0,E} \big]  \nonumber \\ 
		& \lesssim \sum_{E \in \Omega_h } \sqrt{\kappa} h^k_E \| \phi \|_{k+1,E} \vertiii{ e^\phi_h }_{\Sigma_h,E} \lesssim \sqrt{\kappa} h^k \| \phi \|_{k+1,\Omega} \vertiii{ e^\phi_h }_{\Sigma_h}. \label{temp-6} 
	\end{align}
	$\bullet$  {By adding and subtracting appropriate terms and using the skew-symmetry of the trilinear form $c_{T,h}^S(\cdot; \cdot,\cdot)$, we have}
	\begin{align}
		\zeta_{A_T,c}&= |c^S_{T}(\mathbf{u}; \phi, e^\phi_h) - c^S_{T,h}(\mathbf{u}_h; \phi_I, e^\phi_h) | \nonumber\\
		&=|c^S_{T}(\mathbf{u}; \phi, e^\phi_h) - c^S_{T,h}(\mathbf{u}; \phi, e^\phi_h)+ c^S_{T,h}(\mathbf{u}; \phi, e^\phi_h) - c^S_{T,h}(\mathbf{u}_h; \phi_I, e^\phi_h) | \nonumber\\
		&=|c^S_{T}(\mathbf{u}; \phi, e^\phi_h) - c^S_{T,h}(\mathbf{u}; \phi, e^\phi_h)+ c^S_{T,h}(\mathbf{u}; e^\phi_I, e^\phi_h) + c^S_{T,h}(\mathbf{u}-\mathbf{u}_h; \phi_I, e^\phi_h) | \nonumber\\
		&\lesssim |c^S_{T}(\mathbf{u}; \phi, e^\phi_h) - c^S_{T,h}(\mathbf{u}; \phi, e^\phi_h)| + | c^S_{T,h}(\mathbf{u}; e^\phi_I, e^\phi_h) | + |c^S_{T,h}( \mathbf{u}-\mathbf{u}_h; \phi_I, e^\phi_h) | \nonumber\\
		&=: \zeta_{c,1} + \zeta_{c,2} + \zeta_{c,3}. \label{temp-7}
	\end{align} 
	Regarding $\zeta_{c,1}$, we utilize \cite[Lemma 4.3]{mvem9}:
	\begin{align}
		\zeta_{c,1} &\lesssim \frac{h^k}{\sqrt{\kappa}} \big[  \|\mathbf{u}\|_{k,\Omega}  \|\phi\|_{k+1,\Omega} +  \| \nabla \mathbf{u}\|_{0,\Omega}  \| \phi \|_{k+1,\Omega} + \| \mathbf{u} \|_{k+1,\Omega}  \| \nabla \phi\|_{0,\Omega} \big] \vertiii{e^\phi_h}_{\Sigma_h}. \label{temp-8}
	\end{align}
	We use Remark \ref{stabilityP}, the Sobolev embedding theorem and Lemma \ref{lemmaproj2}:
	\begin{align}
		\zeta_{c,2} & \lesssim \sum_{E \in \Omega_h }  \big[ \|\boldsymbol{\Pi}^{0,E}_k \mathbf{u}\|_{0,4,E} \|\boldsymbol{\Pi}^{0,E}_{k-1} \nabla e^\phi_I\|_{0,E} \| \Pi^{0,E}_k e^\phi_h\|_{0,4,E} \, +  \|\boldsymbol{\Pi}^{0,E}_k \mathbf{u}\|_{0,4,E}  \| \boldsymbol{\Pi}^{0,E}_{k-1} \nabla e^\phi_h\|_{0,E} \|{\Pi}^{0,E}_k e^\phi_I\|_{0,4,E}\big] \nonumber \\
		& \lesssim \sum_{E \in \Omega_h }  \big[ h^k_E \|\mathbf{u}\|_{0,4,E} \|\phi\|_{k+1,E} \|e^\phi_h\|_{0,4,E} +  \|e^\phi_I\|_{0,4,E} \|\mathbf{u}\|_{0,4,E}  \|  \nabla e^\phi_h\|_{0,E} \big] \nonumber \\
		& \lesssim  \big[ h^k \|\mathbf{u}\|_{0,4,\Omega} \|\phi\|_{k+1,\Omega} \|e^\phi_h\|_{0,4,\Omega} + \|e^\phi_I\|_{0,4,\Omega} \|\mathbf{u}\|_{0,4,\Omega}  \|  \nabla e^\phi_h\|_{0,\Omega} \big] \nonumber \\
		&\lesssim  \frac{h^{k}}{\sqrt{\kappa}} \|\nabla \mathbf{u}\|_{0,\Omega} \| \phi\|_{k+1,\Omega} \vertiii{ e^\phi_h}_{\Sigma_h}. \label{temp-9}
	\end{align}
	Following the estimation of $\zeta_{c,2}$, we arrive at
	\begin{align}
		\zeta_{c,3} &\lesssim \big[ \|(\mathbf{u} - \mathbf{u}_h) \|_{0,4,\Omega} \| \nabla \phi_I\|_{0,\Omega} \| e^\phi_h\|_{0,4,\Omega} +\|(\mathbf{u} - \mathbf{u}_h) \|_{0,4,\Omega} \| \phi_I\|_{0,4,\Omega} \|\nabla e^\phi_h\|_{0,\Omega} \big]  \nonumber \\ 
		& \lesssim  \frac{1}{\sqrt{\kappa}}\big[ \|\nabla (\mathbf{u}_h - \mathbf{u}_I) \|_{0,\Omega} + h^k \|\mathbf{u}\|_{k+1,\Omega} \big] \| \nabla \phi\|_{0,\Omega} \vertiii{e^\phi_h}_{\Sigma_h}. \label{temp-10}
	\end{align}
	Adding the estimates \eqref{temp-7}--\eqref{temp-9}, it yields that
	\begin{align}
		\zeta_{A_T,c} &\lesssim  \Big[\frac{h^k}{\sqrt{\kappa}} \big[\big( \|\mathbf{u}\|_{k,\Omega}   + \| \nabla \mathbf{u}\|_{0,\Omega}\big) \|\phi\|_{k+1,\Omega}  + \| \nabla \phi\|_{0,\Omega} \| \mathbf{u} \|_{k+1,\Omega}  \big] \nonumber \\ & \qquad +  \frac{1}{\sqrt{\kappa}} \| \nabla \phi\|_{0,\Omega} \|\nabla (\mathbf{u}_h - \mathbf{u}_I) \|_{0,\Omega}  \Big]\vertiii{e^\phi_h}_{\Sigma_h}. \label{zeta-ATc}
	\end{align}
	\textbf{Step 3.} Following the estimation of $\eta_{\mathcal{L}_{h1}}$ in \eqref{sterr-13}, it yields that
	\begin{align}
		\zeta_{\mathcal{L}_3} \lesssim \sum_{E \in \Omega_h } h^k_E \tau_{3,E}^{1/2} \|\phi\|_{k+1,E} \vertiii{e^\phi_h}_{\Sigma_h,E} \lesssim  h^k \max_{E\in\Omega_h} \tau_{3,E}^{1/2} \|\phi\|_{k+1,\Omega} \vertiii{e^\phi_h}_{\Sigma_h}. \label{temp-11} 
	\end{align}
	Finally, adding the estimates \eqref{temp-4}, \eqref{temp-6}, \eqref{zeta-ATc} and \eqref{temp-11}, we readily arrive at the result \eqref{lm-err2f}. 
\end{proof}
Hereafter, we show the error estimate for the stabilized coupled problem \eqref{nvem} in the energy norm:
\begin{theorem} \label{lm-err3}
	{Under the assumptions \textbf{(A1)}, \textbf{(A2)} and the assumptions of Theorems \ref{unique-1} and \ref{sunique},} let  $(\mathbf{u}, p, \phi) \in \mathbf{V} \times Q \times \Sigma$ and $(\mathbf{u}_h, p_h, \phi_h) \in \mathbf{V}_h \times Q_h \times \Sigma_h$ be the solution of the problems \eqref{variation-1} and \eqref{nvem}, respectively, with $\phi \in \widehat{\Sigma}$ and $\phi_h \in \widehat{\Sigma}_h$. We assume the stabilization parameters are $\tau_{1,E} \sim \mathcal{O}(1)$, $\tau_{2,E} \sim h^2_E$ and $\tau_{3,E} \sim h_E$ for all $E \in \Omega_h$. Furthermore, we consider there exists a constant $C_\ast$ depending on the constants $\theta^\ast_1$ and $\theta^\ast_2$, the mesh regularity constants, the Poincar$\acute{\text{e}}$ constant $C_P$ and the embedding constant $C_q$ such that if the data satisfies
	\begin{align}
		\widetilde{\beta} \mu_{min} - C_\ast  \frac { \big( C_{\mathbb T}^{-1} + \alpha \rho \|\mathbf{f}\|_{0, \Omega} \big)}{ C_\theta \kappa^{3/2} } >0, \label{err_c}
	\end{align}
	then, we have the following error estimate:
	\begin{align}
		\vertiii{(\mathbf{u}- \mathbf{u}_h, p-p_h)}_h + \vertiii{\phi-\phi_h}_{\Sigma_h}&\lesssim   \mathscr{R} \big(\alpha,\mu, \kappa, C_\theta, L_\mu, \mathbf{f}, g, \mathbf{u}, p, \phi \big) h^k, \label{lm-err3f} 
	\end{align}
	where $\mathscr{R} \big(\alpha,\mu, \kappa, C_\theta, L_\mu, \mathbf{f}, g, \mathbf{u}, p, \phi \big)$ is independent of the mesh size and depends on $\alpha$, $\mu$, $\kappa$, $C_\theta$, $L_\mu$, $\mathbf{f}$, $g$, $\mathbf{u}$, $p$, and $\phi$.
\end{theorem}

\begin{proof}
	Employing the definition of stabilization parameters and Lemma \ref{lm-err1}, we have
	\begin{align}
		\vertiii{(\mathbf{u}_h- \mathbf{u}_I, p_h-p_I)}_h &\lesssim  \frac{1}{\widetilde{\beta}}\Big[ \frac { \big( L_\mu {\|\nabla \mathbf{u}\|_{0,3,\Omega}} + \alpha \|\mathbf{f}\|_{0, \Omega} \big)}{\sqrt{ \kappa \mu_{min}}} \vertiii{\phi_h - \phi_I}_{\Sigma_h} + \mathscr{R} \big(\alpha,\mu, L_\mu, \mathbf{f}, \mathbf{u}, p, \phi \big) h^k \Big]. \label{finl-1}
	\end{align} 
	Using Lemma \ref{lm-err2}, we infer
	\begin{align}
		\vertiii{(\mathbf{u}_h- \mathbf{u}_I, p_h-p_I)}_h &\lesssim  \frac{1}{\widetilde{\beta}}\Big[ \frac { \big( L_\mu {\|\nabla \mathbf{u}\|_{0,3,\Omega}} + \alpha \|\mathbf{f}\|_{0, \Omega} \big)}{ C_\theta \kappa \mu_{min}} \| \nabla \phi\|_{0,\Omega}\vertiii{(\mathbf{u}_h- \mathbf{u}_I, p_h-p_I)}_h \nonumber \\ &  \qquad + \mathscr{R} \big(\alpha,\mu, \kappa,C_\theta, L_\mu, \mathbf{f}, g, \mathbf{u}, p, \phi \big) h^k \Big].
	\end{align} 
	{We now use the estimate \eqref{unq0} and the definition of $\widehat{\Sigma}$, we arrive at
		\begin{align}
			\vertiii{(\mathbf{u}_h- \mathbf{u}_I, p_h-p_I)}_h &\lesssim  \frac{1}{\widetilde{\beta}}\Big[ \frac { \big( C_{\mathbb T}^{-1} + \alpha \rho \|\mathbf{f}\|_{0, \Omega} \big)}{ C_\theta \kappa^{3/2} \mu_{min}} \vertiii{(\mathbf{u}_h- \mathbf{u}_I, p_h-p_I)}_h + \mathscr{R} \big(\alpha,\mu, \kappa,C_\theta, L_\mu, \mathbf{f}, g, \mathbf{u}, p, \phi \big) h^k \Big]. \label{finl-2}
	\end{align} }
	Recalling \eqref{err_c}, we obtain 
	\begin{align}
		\vertiii{(\mathbf{u}_h- \mathbf{u}_I, p_h-p_I)}_h &\lesssim   \mathscr{R} \big(\alpha,\mu, \kappa,C_\theta, L_\mu, \mathbf{f}, g, \mathbf{u}, p, \phi \big) h^k. \label{finl-3}
	\end{align} 
	Therefore, combining Lemma \ref{lm-err2} and \eqref{finl-3}, yields
	\begin{align}
		\vertiii{(\mathbf{u}_h- \mathbf{u}_I, p_h-p_I)}_h + \vertiii{\phi_h-\phi_I}_{\Sigma_h}&\lesssim  \mathscr{R} \big(\alpha,\mu, \kappa, C_\theta, L_\mu, \mathbf{f}, g, \mathbf{u}, p, \phi \big) h^k. \label{lm-err4f}
	\end{align}	
	We now use the definition of $\vertiii{\cdot}_h$ and Lemma \ref{lemmaproj2}:
	\begin{align}
		\vertiii{(\mathbf{u}- \mathbf{u}_I, p-p_I)}_h \lesssim  \mathscr{R}(\mu) h^k \big( \|\mathbf{u} \|_{k+1,\Omega} + \|p\|_{k,\Omega}\big). \label{sterr-1}
	\end{align}	
	It is evident from the definition of $\vertiii{\cdot}_\Sigma$ and Lemmas \ref{lemmaproj1} and \ref{lemmaproj2} that
	\begin{align}
		\vertiii{\phi - \phi_I}_\Sigma \lesssim h^k \mathscr{R}(\kappa) \|\phi\|_{k+1,\Omega}. \label{temp-2}
	\end{align}	
	Finally, combining the results	\eqref{lm-err4f}, \eqref{sterr-1} and \eqref{temp-2}, we readily obtain the result \eqref{lm-err3f}. 
\end{proof}

\section{Numerical results} 
\label{sec:6}

In this section we present four examples to examine the numerical behavior of the proposed stabilized VEM \eqref{nvem}. In addition, we will demonstrate the effectiveness of the proposed scheme in the presence of convection-dominated transport regimes. These examples demonstrate the method's efficiency and underscore its advantages in addressing realistic problems.

\subsection{Meshes and computational error norms} \label{sec-error}
We employ uniform and distorted convex and non-convex meshes to study the numerical experiments. For the first three examples, we utilize meshes $\Omega_1$--$\Omega_4$, as shown in Figure \ref{samp}, with the mesh sizes $h=1/5, 1/10, 1/20, 1/40$ and $1/80$. Further, we use the fixed-point iterative scheme to solve the nonlinear coupled problem \eqref{nvem}. For all examples, the stopping criteria of the iterative scheme is fixed at $10^{-7}$. We investigate the convergence behavior on unit square $\Omega=(0,1)^2$. 
Additionally, we will follow that the stabilization parameters are $\tau_{1,E} \sim \mathcal{O}(1)$, $\tau_{2,E} \sim h_E^2$ and $\tau_{3,E} \sim h_E$.  Our numerical results are discussed for VEM orders $k=1$ and $k=2$.

\begin{figure}[h]
	\centering
	\subfloat[$\Omega_1$]{\includegraphics[height=3cm, width=4cm]{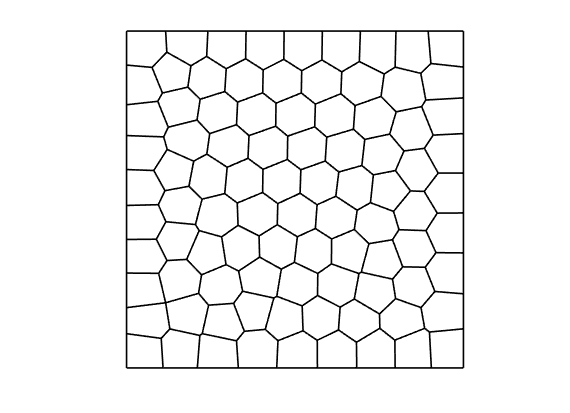}}
	\subfloat[$\Omega_2$]{\includegraphics[height=3cm, width=3cm]{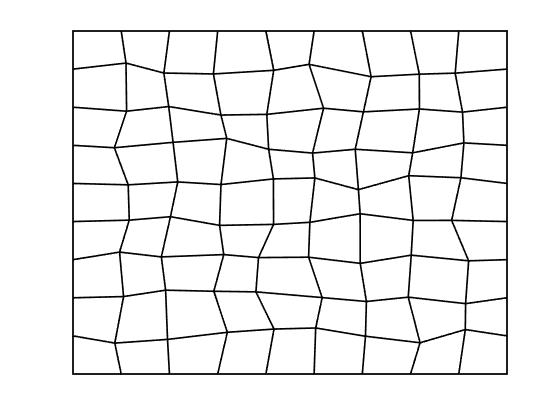}}~~
	\subfloat[$\Omega_3$]{\includegraphics[height=3cm, width=3cm]{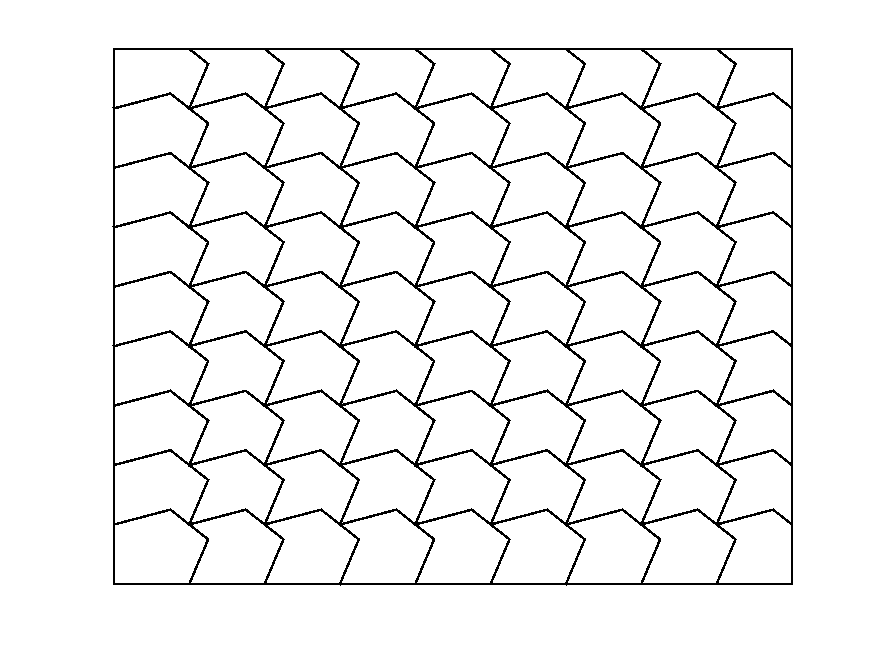}}~~
	\subfloat[$\Omega_4$]{\includegraphics[height=3cm, width=3cm]{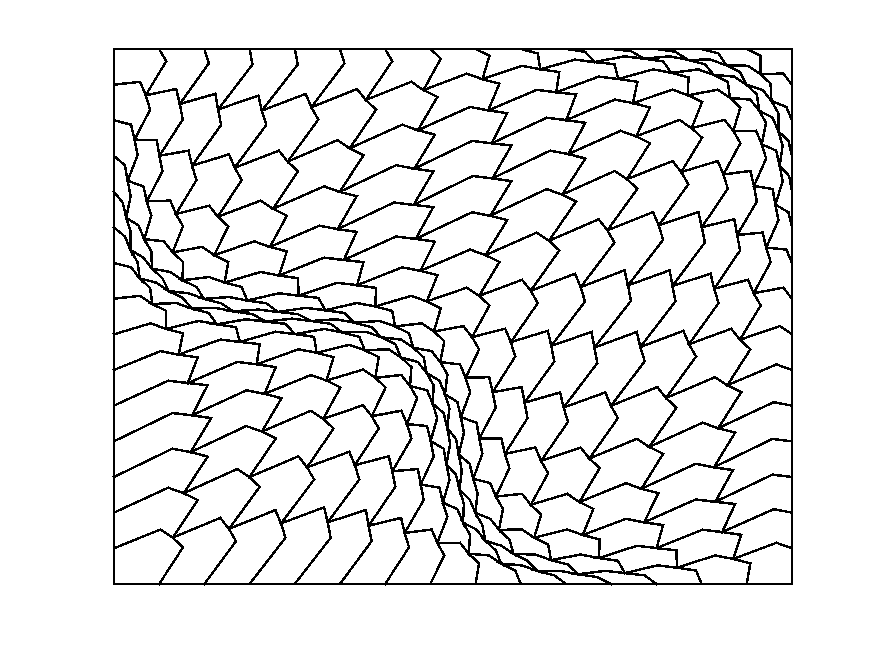}}
	\caption{ Computational meshes.}
	\label{samp} 
\end{figure}

The computable VEM errors, precisely the $H^1(\Omega)$ semi-norm and $L^2(\Omega)$ norm for the velocity vector field, are defined as follows:  \newline

$\bullet$ Velocity $H^1$-semi norm: \qquad	$E^{\mathbf{u}}_{H^1} := \sqrt{\sum\limits_{E\in\Omega_h}  \| \nabla (\mathbf{u} - \boldsymbol{\Pi}^{\nabla,E}_k \mathbf{u}_h ) \|_E^2}$, 

$\bullet$ Velocity $L^2$-norm: \qquad \qquad \,\,\,$ E^\mathbf{u}_{L^2} := \sqrt{\sum\limits_{E\in\Omega_h} \| \mathbf{u}-\boldsymbol{\Pi}^{0,E}_k \mathbf{u}_h\|^2_E}$, \newline
where $\mathbf{u}$ and $\mathbf{u}_h$ represent the exact and VEM-stabilized velocity field, respectively. Following these definitions, the errors between the exact temperature $\phi$ and VEM-stabilized temperature $\phi_h$  are denoted by $E^{\phi}_{H^1}$ in the $H^1$ semi-norm and $E^{\phi}_{L^2}$ in the $L^2$ norm. Additionally, the pressure error in the $L^2(\Omega)$ norm is denoted by $E^{p}_{L^2}$, with $p$ and $p_h$, representing the exact and VEM-stabilized pressures, respectively. 

\begin{figure}[h]
	\centering
	\subfloat[$\Omega_1$]{\includegraphics[height=6cm, width=7.5cm]{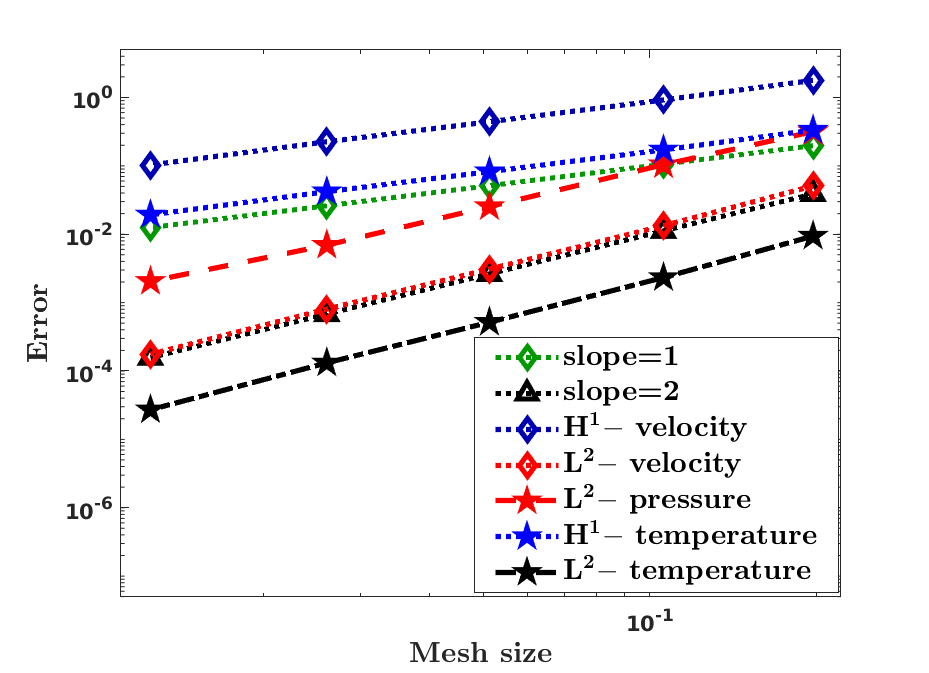}}
	\subfloat[$\Omega_1$]{\includegraphics[height=6cm, width=7.5cm]{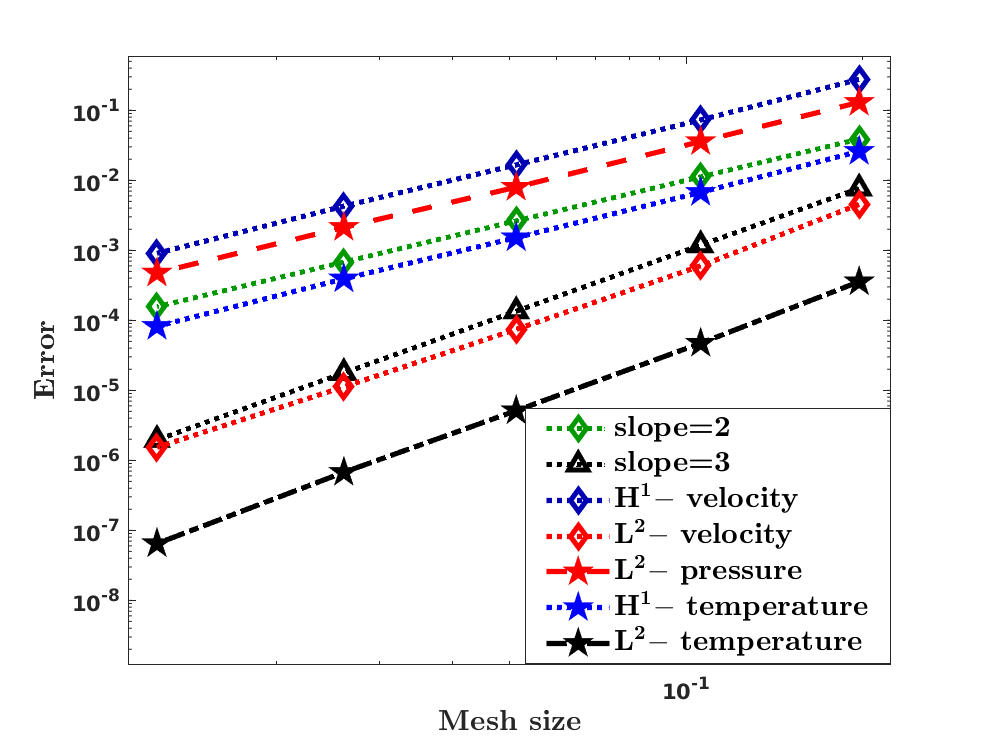}}\\
	\subfloat[$\Omega_2$]{\includegraphics[height=6cm, width=7.5cm]{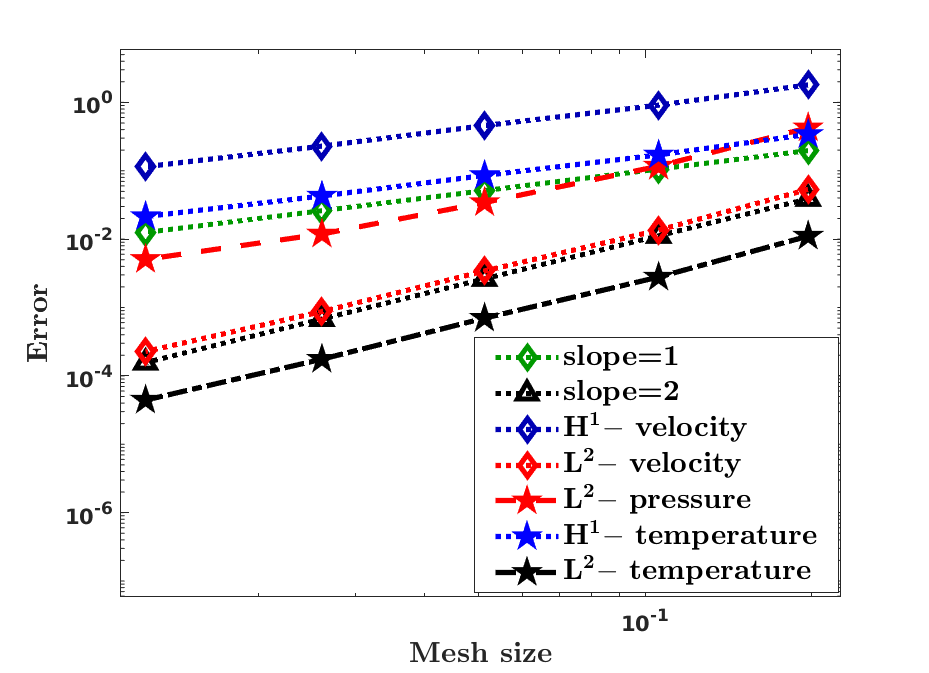}}
	\subfloat[$\Omega_2$]{\includegraphics[height=6cm, width=7.5cm]{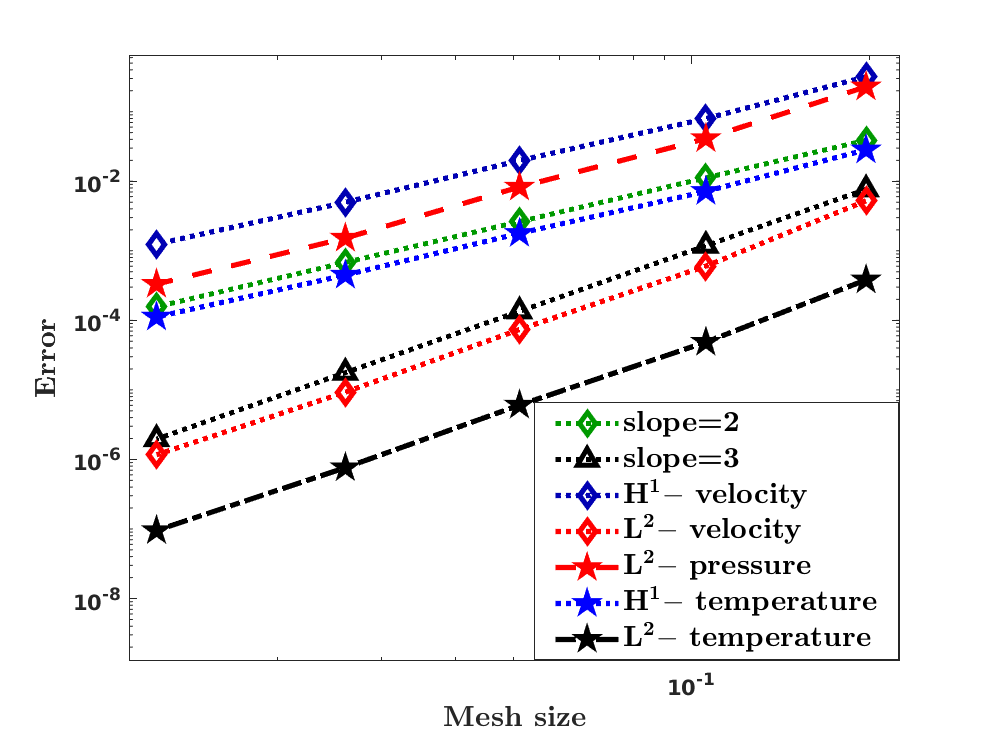}}\\
	\caption{ Example 1. Convergence curves for VEM order $k=1$ (left) and $k=2$ (right).}
	\label{ex1_cng1} 
\end{figure}

\subsection{Example 1}
\label{case1}
In this example, we present a benchmark problem for the convergence study discussed in \cite{mfem39} with a known solution.  We consider the parameters of problem (S) as follows:
\begin{align}
	\mu(r)=\frac{1}{(1-0.5r)^2},  \qquad \alpha=1, \qquad \kappa=1. \nonumber 
\end{align}
We choose the source terms $\mathbf{f}$ and $g$ such that the exact solution of problem (S) is given as follows: 
\begin{empheq}[left= \empheqlbrace]{align}
	\mathbf{u}(x,y) &= \big[ \sin(2\pi x) \cos(2 \pi y), \, -\cos(2\pi x)\sin(2\pi y) \big]^T,\nonumber \\
	p(x,y) &= \sin(2\pi x) \sin(2\pi y), \nonumber\\
	\phi(x,y) &= 15 - 15 \exp(-xy(x-1)(y-1)). \nonumber 
\end{empheq} 
We emphasize that $0 \leq \phi \leq 1$ in $\Omega$, and that $\mu^\prime(r) \leq 8$ for $0<r<1$, this gives the Lipschitz continuity of $\mu$.

We perform Example 1 considering the dominated-diffusion regime for the transport equation using the computational meshes $\Omega_1$ and $\Omega_2$ for VEM order $k=1$ and $2$. In Figure \ref{ex1_cng1}, we present the convergence curves obtained using the proposed VEM. In this context, our primary focus is on validating that the proposed stabilized VEM meets theoretical expectations. For a regular problem, the expected optimal convergence rates are  $\mathcal{O}(h^k)$ for the $H^1$ semi-norm of velocity and temperature, and $\mathcal{O}(h^{k})$ for the $L^2$ norm of the pressure. While we have not explicitly derived error estimates for the velocity and temperature in the $L^2$-norm, but the interpolation estimates suggest the best possibility of the optimal rates in the $L^2$-norm is $\mathcal{O}(h^{k+1})$. This expectation also applies to convection-dominated transport (temperature) problems.

From Figure \ref{ex1_cng1}, we notice that the convergence behavior of the proposed VEM aligns well with  theoretical results for VEM order $k=1$ and $2$. Additionally, we observe the $L^2$-norm of pressure error shows superconvergence behavior {of order $\mathcal{O}(h^2)$} on regular Voronoi mesh $\Omega_1$ with the lowest equal-order VEM {$k=1$}, although we have not derived superconvergence estimates for pressure, whereas it is optimal {of order $\mathcal{O}(h)$} for the distorted mesh $\Omega_2$.

\begin{figure}[h]
	\centering
	\subfloat[$\Omega_3$]{\includegraphics[height=6cm, width=7.5cm]{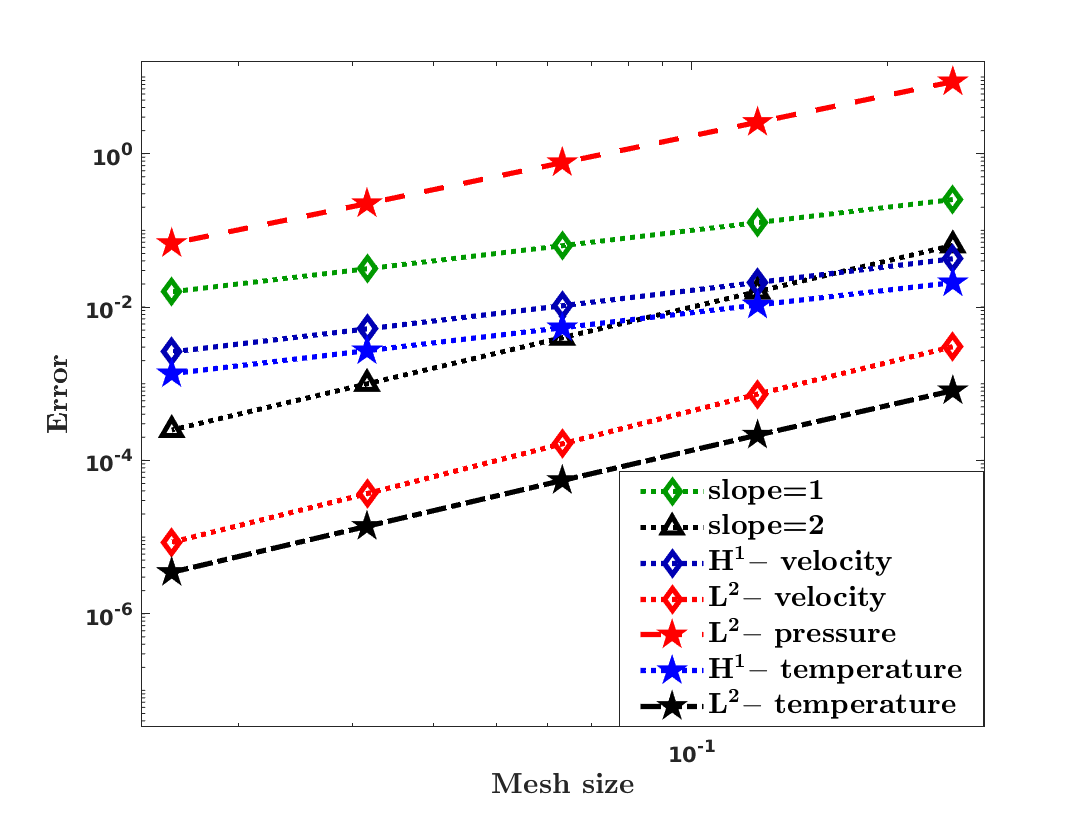}}
	\subfloat[$\Omega_3$]{\includegraphics[height=6cm, width=7.5cm]{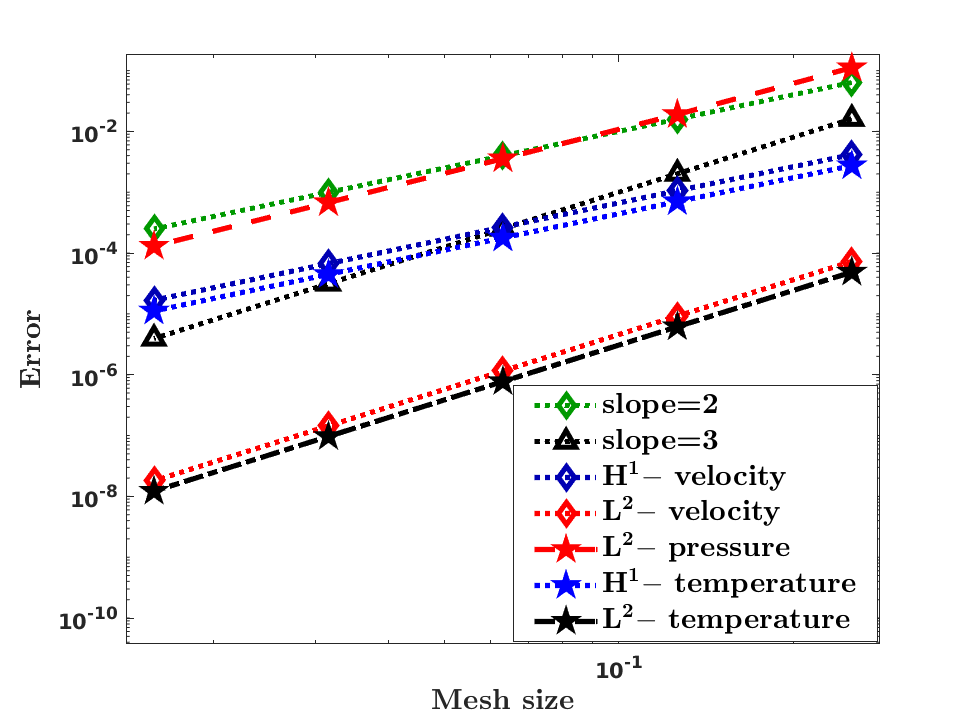}}\\
	\subfloat[$\Omega_4$]{\includegraphics[height=6cm, width=7.5cm]{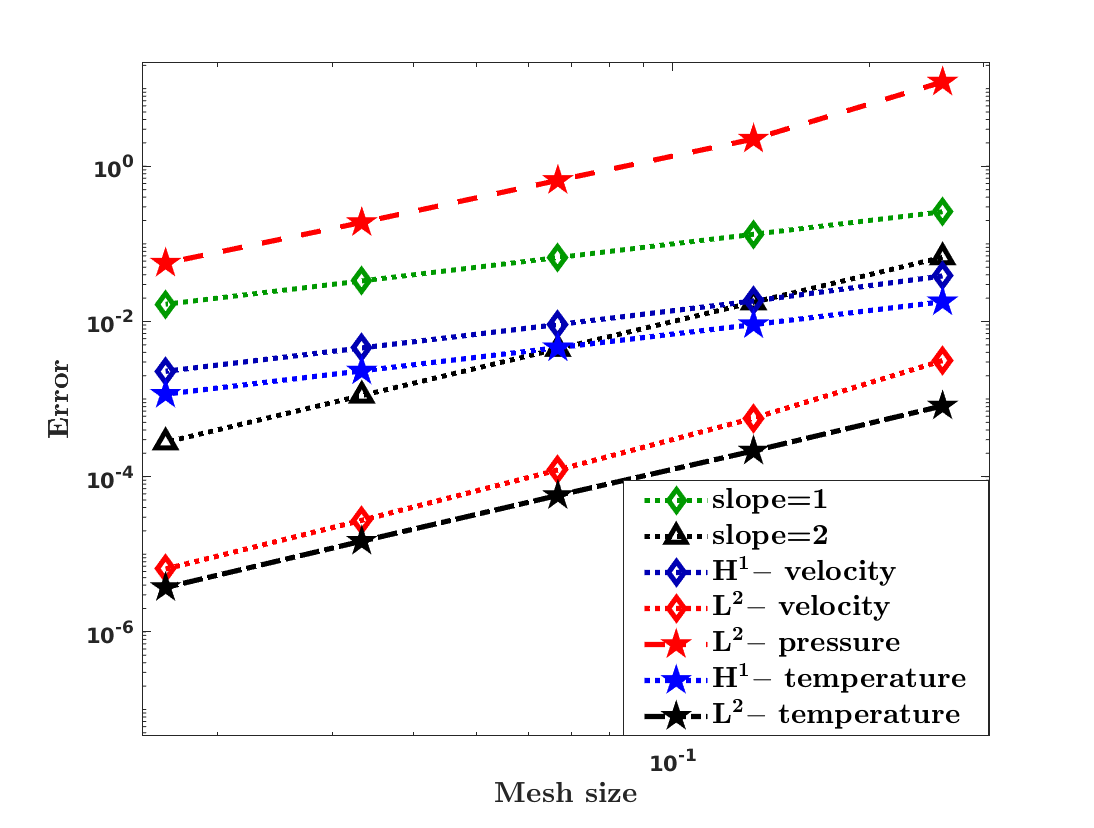}}
	\subfloat[$\Omega_4$]{\includegraphics[height=6cm, width=7.5cm]{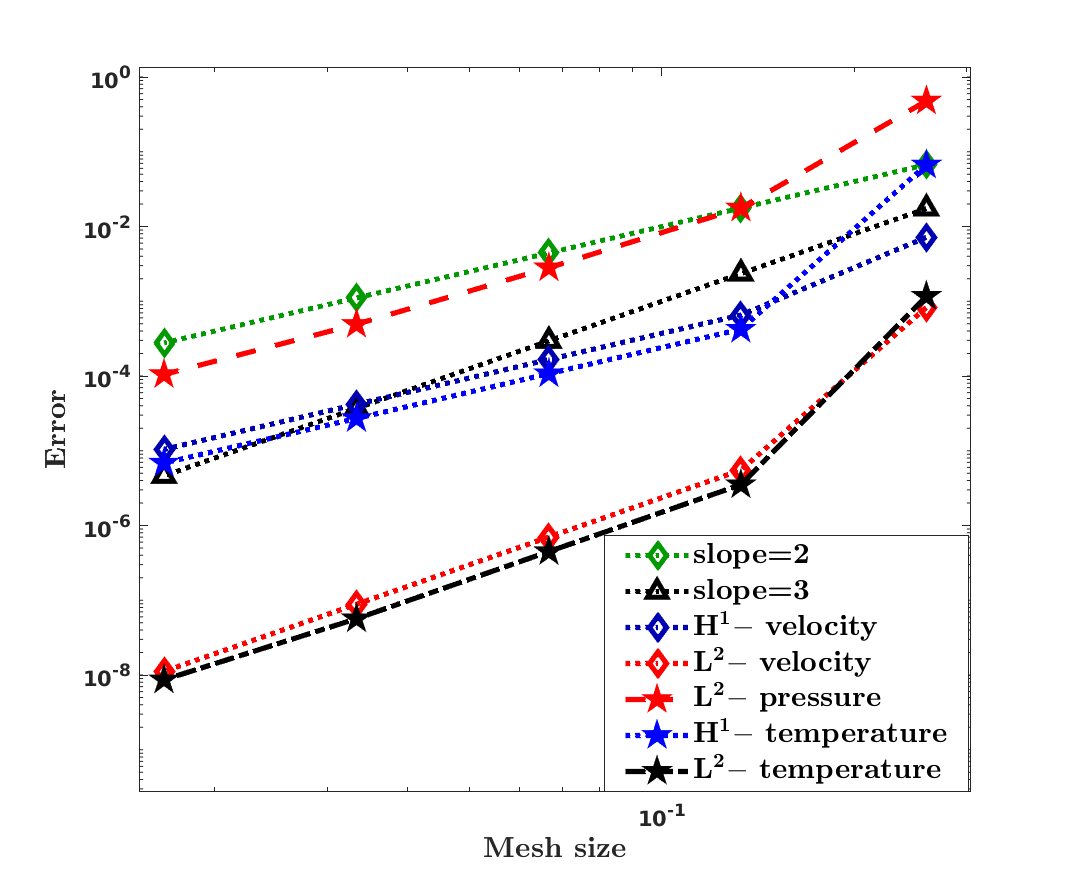}}
	\caption{ Example 2. (Case $\kappa=1$.) Convergence curves for VEM order $k=1$ (left) and $k=2$ (right).}
	\label{ex2_cng1} 
\end{figure}

\subsection{Example 2}
\label{case2}
In the second example, we aim to investigate the convergence behavior of the proposed method by introducing the diffusion-dominated and convection-dominated regimes (critical case) for the transport temperature equation. To do this, we consider the parameters of problem $(S)$ as follows:
\begin{align}
	\mu(\phi)= 1+ \phi + \sin^2(\phi),  \qquad \alpha=1, \qquad \kappa=1\,\, \text{and}\,\, 10^{-6} . \nonumber 
\end{align}
Furthermore, the external loads and boundary conditions are determined by the exact solution of problem $(S)$, given by 
\begin{empheq}[left= \empheqlbrace]{align}
	\mathbf{u}(x,y) &= \big[ x^2 y(1-x) (1-y), \, -(2x-3x^2) \big( \frac{y^2}{2} - \frac{y^3}{3} \big)\big]^T,\nonumber \\
	p(x,y) &= -100 x^2 + 100/3, \nonumber\\
	\phi(x,y) &= x^2 y(1-x) (1-y) + 600. \nonumber 
\end{empheq} 
\begin{figure}[h]
	\centering
	\subfloat[$\Omega_3$]{\includegraphics[height=6cm, width=7.5cm]{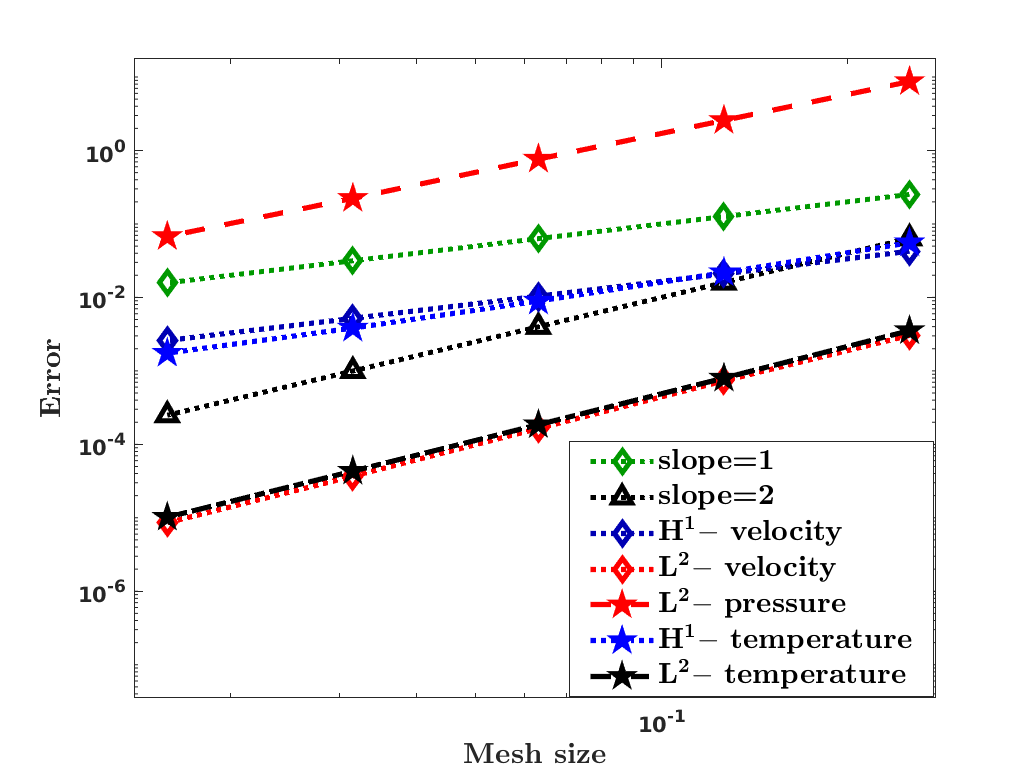}}
	\subfloat[$\Omega_3$]{\includegraphics[height=6cm, width=7.5cm]{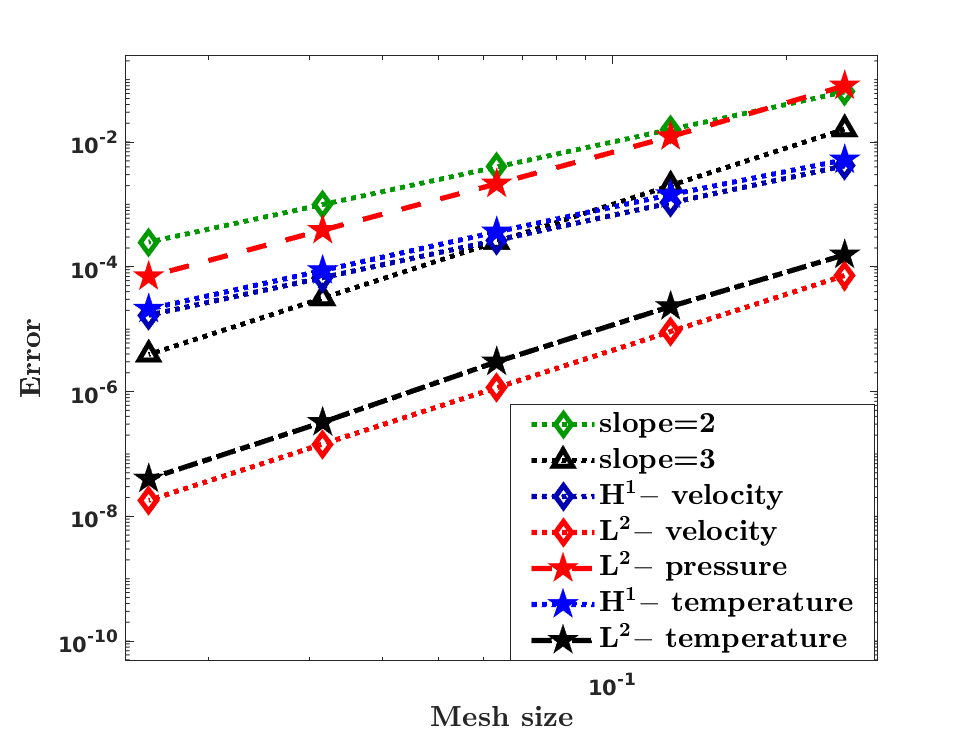}}\\
	\subfloat[$\Omega_4$]{\includegraphics[height=6cm, width=7.5cm]{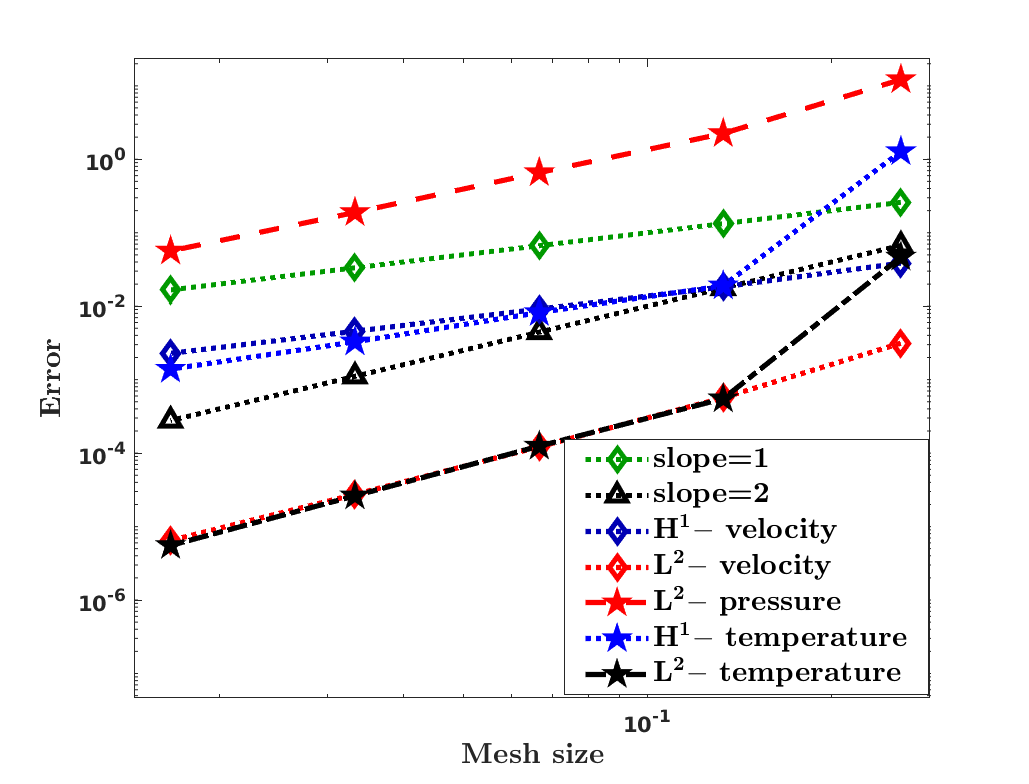}}
	\subfloat[$\Omega_4$]{\includegraphics[height=6cm, width=7.5cm]{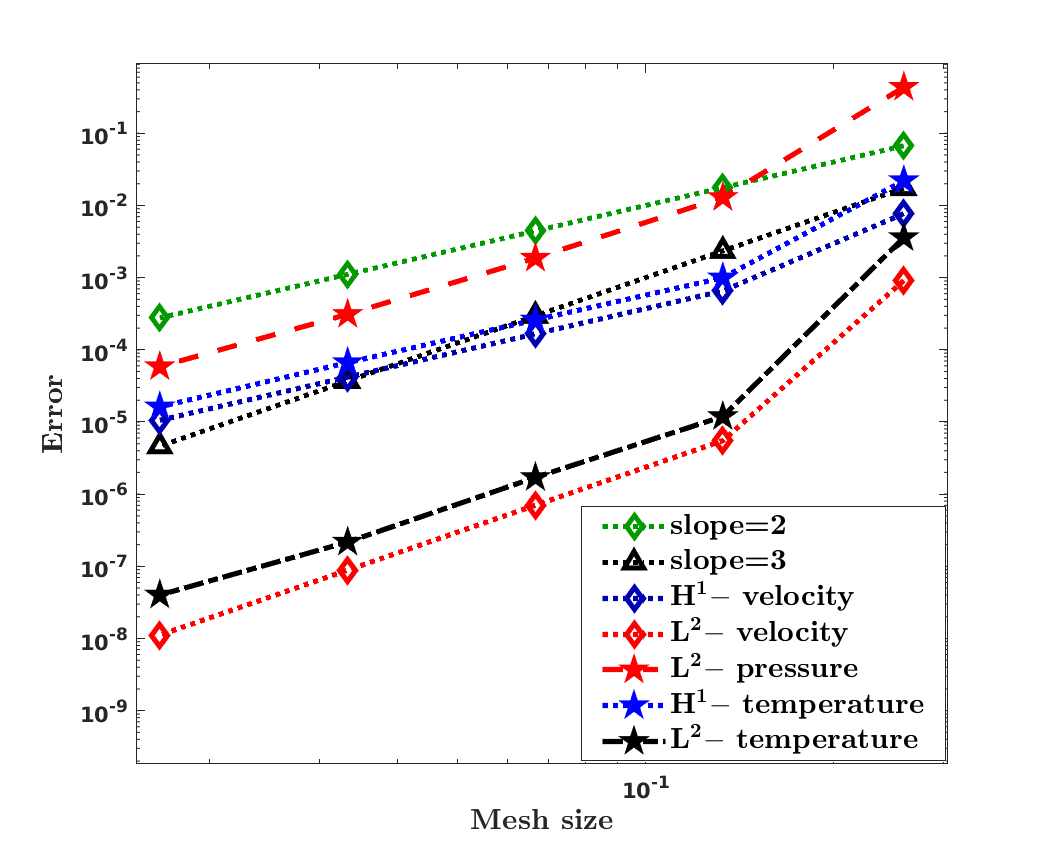}}
	\caption{ Example 2. (Case $\kappa=10^{-6}$.) Convergence curves for VEM order $k=1$ (left) and $k=2$ (right).}
	\label{ex2_cng2} 
\end{figure}

In the first case, where $\kappa=1$, we depict the convergence curves obtained using the proposed method in Figure \ref{ex2_cng1} for meshes $\Omega_3$ and $\Omega_4$ with VEM order $k=1$ and $2$. The numerical results demonstrate that the proposed method achieves optimal convergence rates for all virtual element pairs, validating the theoretical findings. Notably, the $L^2$-norm of the pressure error exhibits superconvergence {of order $\mathcal{O}(h^2)$ for both meshes with VEM order $k=1$, whereas the theoretical convergence rate for pressure is achieved of order $\mathcal{O}(h)$.}

Concerning the second case, when $\kappa=10^{-6}$, the error estimates become quasi-uniform in the energy norm as $\kappa \rightarrow 0$, and the condition \eqref{err_c} seems almost impossible to hold. Therefore, the second case represents the critical case associated with the convection-dominated regime. The convergence plots are shown in Figure \ref{ex2_cng2} for VEM order $k=1$ and $2$. Once again, the proposed method confirms the theoretical findings, particularly highlighting the superconvergence of the pressure in the $L^2$-norm for VEM order $k=1$. This underscores the robustness of the proposed stabilized method for convection-dominated regimes. 

\begin{figure}[h]
	\centering
	\subfloat[Exact velocity field.]{\includegraphics[width=6.5cm]{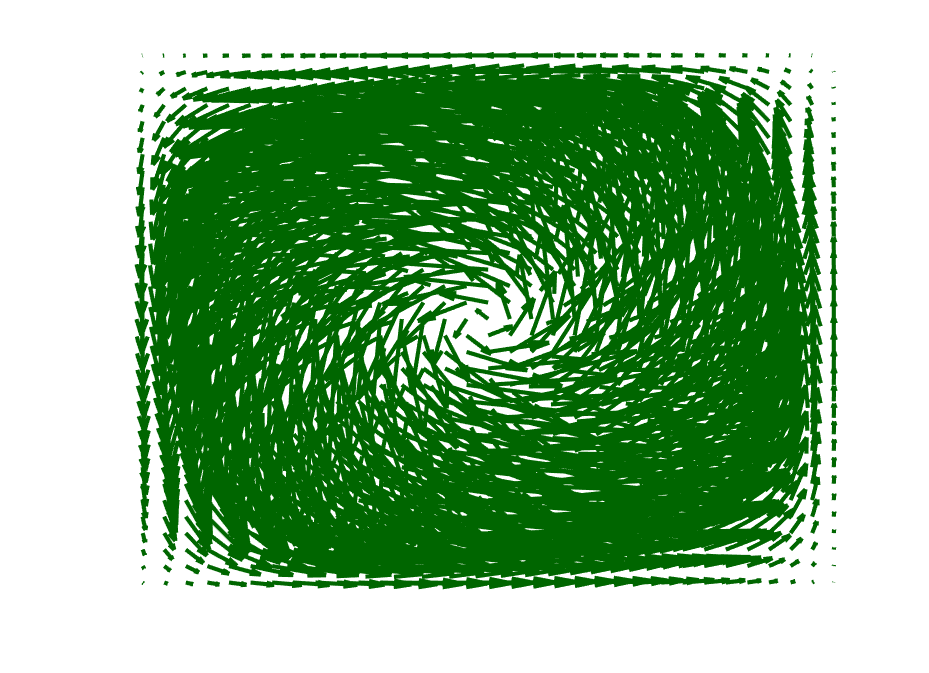}}
	\subfloat[Stabilized velocity field.]{\includegraphics[width=6.5cm]{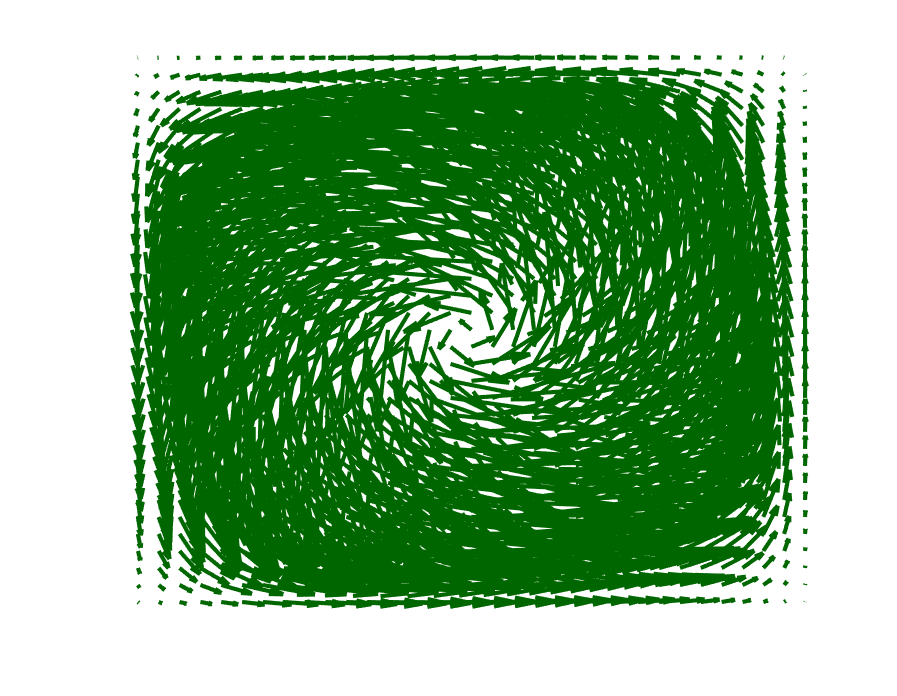}}\\
	\subfloat[Exact pressure.]{\includegraphics[width=6.5cm]{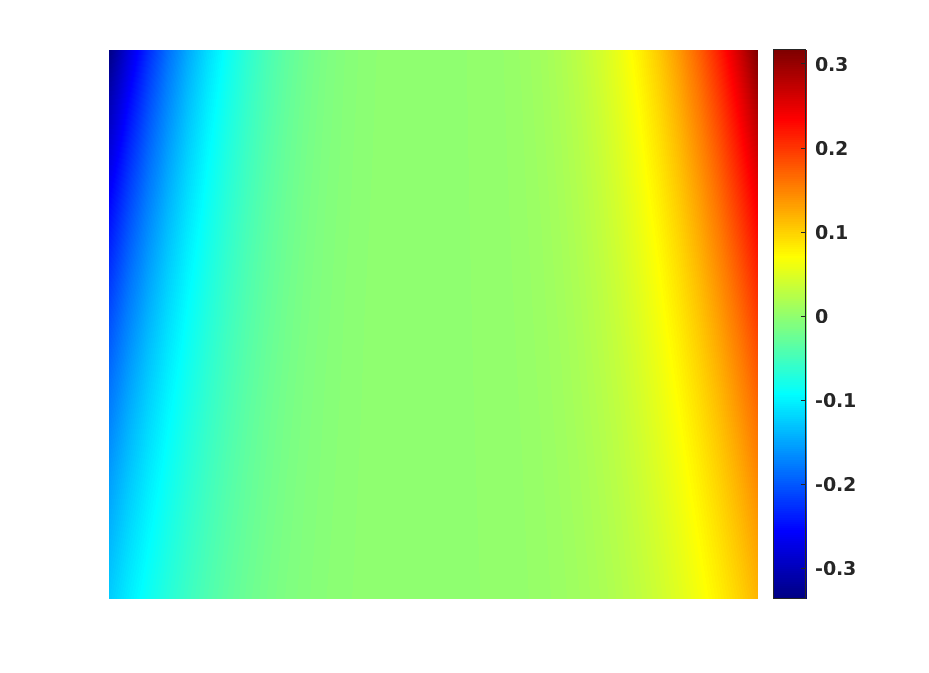}}
	\subfloat[Stabilized pressure.]{\includegraphics[width=6.5cm]{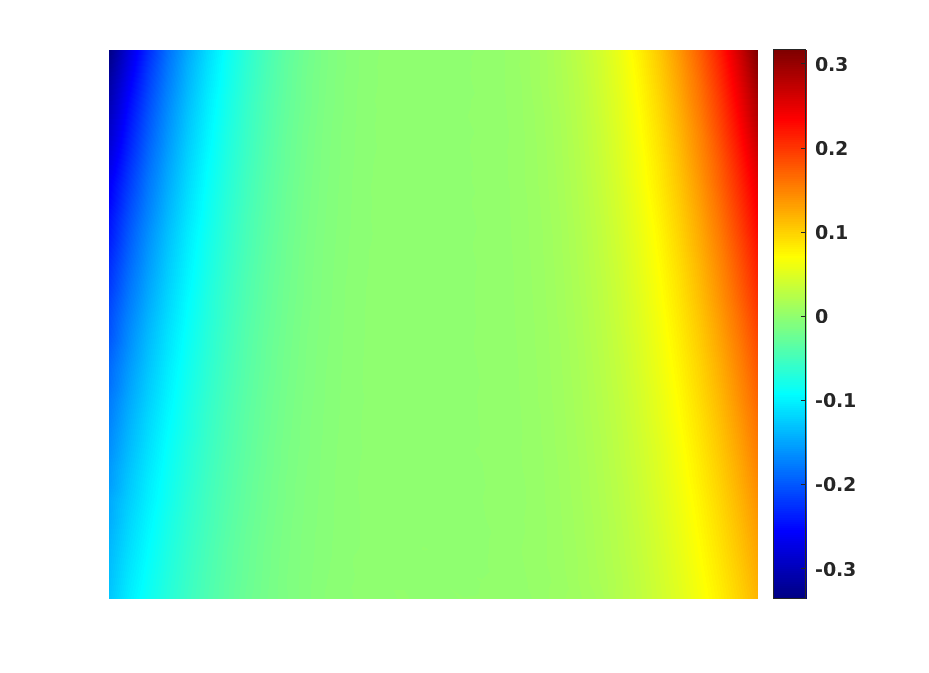}}\\
	\subfloat[Exact temperature.]{\includegraphics[width=6.5cm]{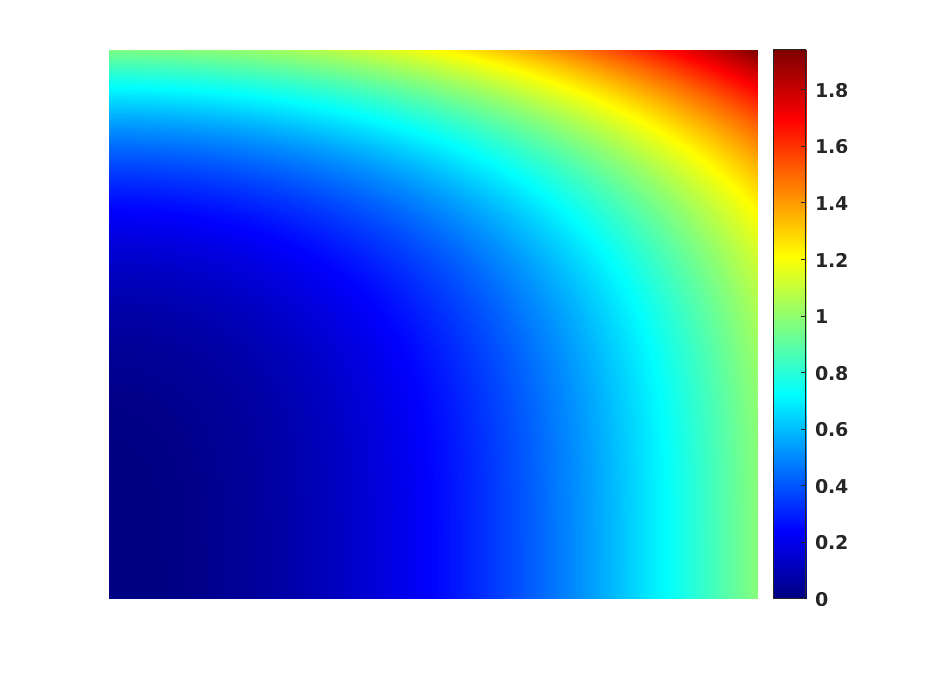}}
	\subfloat[Stabilized temperature.]{\includegraphics[width=6.5cm]{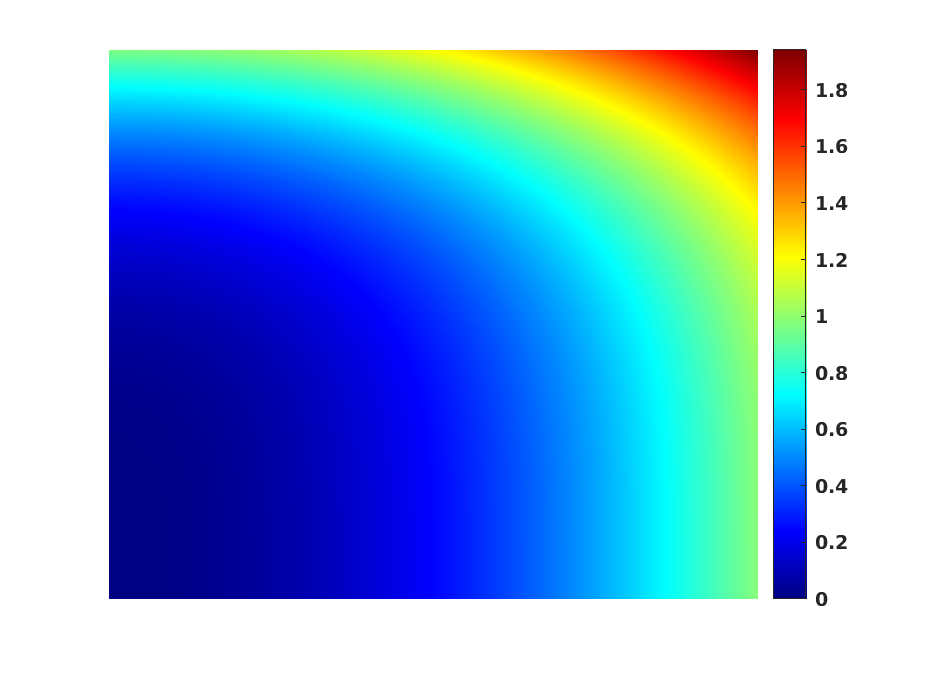}}
	\caption{ Example 3. Exact and stabilized solutions obtained using proposed VEM with mesh size $h=1/40$ for VEM order $k=1$.}
	\label{ex3_plots} 
\end{figure}

\subsection{Example 3}
\label{case3}
In this example, we test the robustness of the proposed method by considering nonlinear thermal conductivity that depends on the temperature, even though our theoretical analysis is restricted to constant thermal conductivity. To achieve this, we extend the classic benchmark problem discussed in \cite{mfem40}. For this example, we use the following parameters:
\begin{align}
	\mu(\phi)= e^{-\phi},  \qquad \alpha=1, \qquad \text{\Large $\kappa$}(\phi)= \kappa e^{\phi}, \nonumber 
\end{align}
where $\kappa$ is a positive constant. Additionally, the loads and boundary conditions are specified by the exact solution of problem $(S)$, which is given as follows:
\begin{empheq}[left= \empheqlbrace]{align}
	\mathbf{u}(x,y) &= \big[ 2 x^2 y(2y-1) (y-1)(x-1)^2, \, -2xy^2(y-1)^2(2x-1)(x-1)\big]^T,\nonumber \\
	p(x,y) &= e^{y}(x-0.5)^3, \nonumber\\
	\phi(x,y) &= x^2+y^4. \nonumber 
\end{empheq} 

In the first part, we discuss the convergence study with $\kappa = 10^{-3}$, using distorted squares $\Omega_2$ for VEM order $k=1$ and 2. In Figure \ref{ex3_plots}, we present the exact solution and stabilized VEM solution for VEM order $k=1$ with the mesh size $h=1/40$. Notably, we observe a remarkable agreement between the solutions. Furthermore, the VEM errors and convergence rates obtained using the proposed method are reported in Tables \ref{ex3_tb1} and \ref{ex3_tb2}. We emphasize that the proposed method achieves optimal convergence rates for the nonlinear thermal conductivity, as expected theoretically.

\begin{table}[t!]
	\setlength{\tabcolsep}{6pt}
	\centering 
	\caption{Example 3. The computational errors and rates using proposed method for VEM order $k=1$.}
	\label{ex3_tb1}
	\begin{tabular}{ccccccccccc}
		\toprule
		{$h$}&  {$E^\mathbf{u}_{H^1}$} & {rate} & {$E^{\mathbf{u}}_{L^2}$} & {rate} & {$E^p_{L^2}$} & {rate} & {$E^{\phi}_{H^1}$} & {rate} & {$E^{\phi}_{L^2}$} & {rate}  \\
		\midrule
		$1/5$  & 1.6332e-02 & --   & 4.1469e-04 & --   & 9.1920e-03 &--    & 1.8772e-01 & --   & 9.8843e-03 & -- \\ 
		$1/10$ & 8.2289e-03 & 0.99 & 1.1116e-04 & 1.90 & 2.8878e-03 & 1.67 & 9.4396e-02 & 0.99 & 2.5340e-03 & 1.96 \\ 
		$1/20$ & 4.1103e-03 & 1.00 & 2.9047e-05 & 1.94 & 1.1760e-03 & 1.30 & 4.7031e-02 & 1.01 & 6.4107e-04 & 1.99\\ 
		$1/40$ & 2.0603e-03 & 1.00 & 7.9007e-06 & 1.88 & 4.3814e-04 & 1.42 & 2.3623e-02 & 0.99 & 1.7421e-04 & 1.88\\ 
		$1/80$ & 1.0335e-03 & 1.00 & 2.3854e-06 & 1.73 & 1.9263e-04 & 1.19 & 1.1806e-02 & 1.00 & 4.9546e-05 & 1.81\\ 
		\bottomrule		
	\end{tabular}
	\smallskip
\end{table}

\begin{table}[t!]
	\setlength{\tabcolsep}{6pt}
	\centering 
	\caption{Example 3. The computational errors and rates using proposed scheme for VEM order $k=2$.}
	\label{ex3_tb2}
	\begin{tabular}{ccccccccccc}
		\toprule
		{$h$}&  {$E^\mathbf{u}_{H^1}$} & {rate} & {$E^{\mathbf{u}}_{L^2}$} & {rate} & {$E^p_{L^2}$} & {rate} & {$E^{\phi}_{H^1}$} & {rate} & {$E^{\phi}_{L^2}$} & {rate}  \\
		\midrule
		$1/5$  & 2.2254e-03 & -- & 3.2859e-05 & --   & 6.3929e-04 &--    & 6.9877e-03 & --   & 1.2961e-04 & -- \\ 
		$1/10$ & 5.6891e-04 & 1.97 & 4.2338e-06 & 2.97 & 1.7051e-04 & 1.91 & 1.7857e-03 & 1.97 & 1.6742e-05 & 2.95 \\ 
		$1/20$ & 1.4246e-04 & 2.00 & 5.2943e-07 & 3.00 & 4.5170e-05 & 1.92 & 4.4949e-04 & 1.99 & 2.1051e-06 & 2.99\\ 
		$1/40$ & 3.5844e-05 & 1.99 & 6.7493e-08 & 2.97 & 1.4136e-05 & 1.68 & 1.1366e-04 & 1.98 & 2.6923e-07 & 2.97\\ 
		$1/80$ & 8.9920e-06 & 2.00 & 8.4955e-09 & 2.99 & 3.7903e-06 & 1.90 & 2.8403e-05 & 2.00 & 3.3533e-08 & 3.01\\ 
		\bottomrule		
	\end{tabular}
	\smallskip
\end{table}
In the second part, we investigate the robustness of the proposed method with respect to nonlinear thermal conductivity in the presence of convection-dominated regimes for transport. To achieve this, we vary the nonlinear thermal conductivity parameter $\kappa$ from $1$ to $10^{-9}$, employing the distorted mesh $\Omega_2$ with the mesh size $h=1/20$. In Figure \ref{ex3_effect}, we depict the VEM errors for the virtual element pairs $(\mathbf{u}_h, p_h,\phi_h)$, defined in Section \ref{sec-error}. We observe that as $\kappa$ decreases from $1$ to $10^{-4}$, the VEM errors gradually increase. However, when 
$\kappa$ is further reduced to $10^{-9}$, the VEM errors remain approximately the same for all element pairs for each value of $\kappa\geq 10^{-5}$. Indeed, the curve lines become flat for all VEM errors. Thus, the practical behavior of the proposed method is robust with respect to nonlinear thermal conductivity.
\begin{figure}[h]
	\centering
	\subfloat[Order $1$.]{\includegraphics[height=5cm, width=6.5cm]{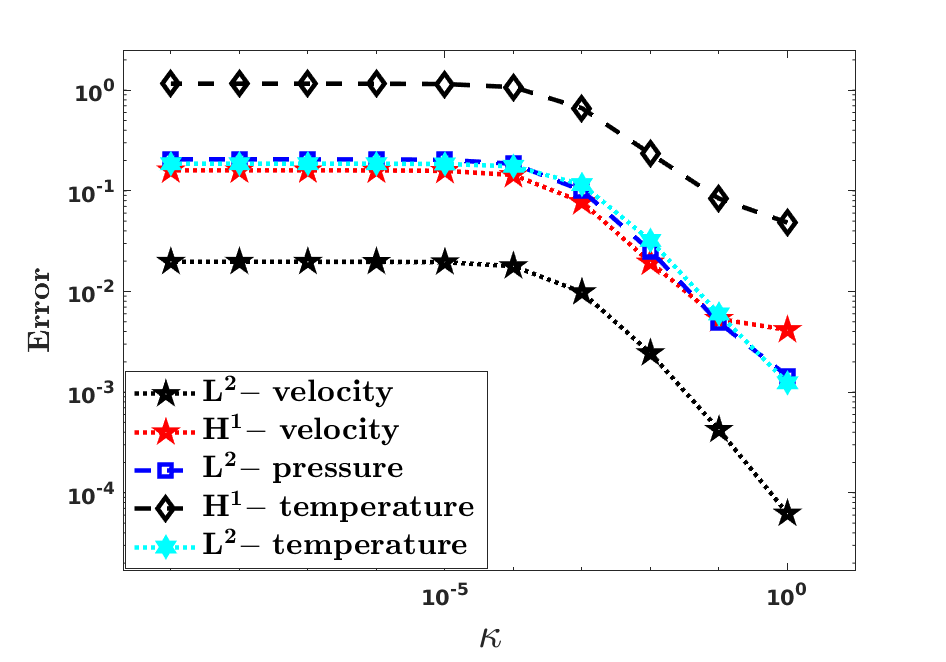}}
	\subfloat[Order $2$.]{\includegraphics[height=5cm,width=6.5cm]{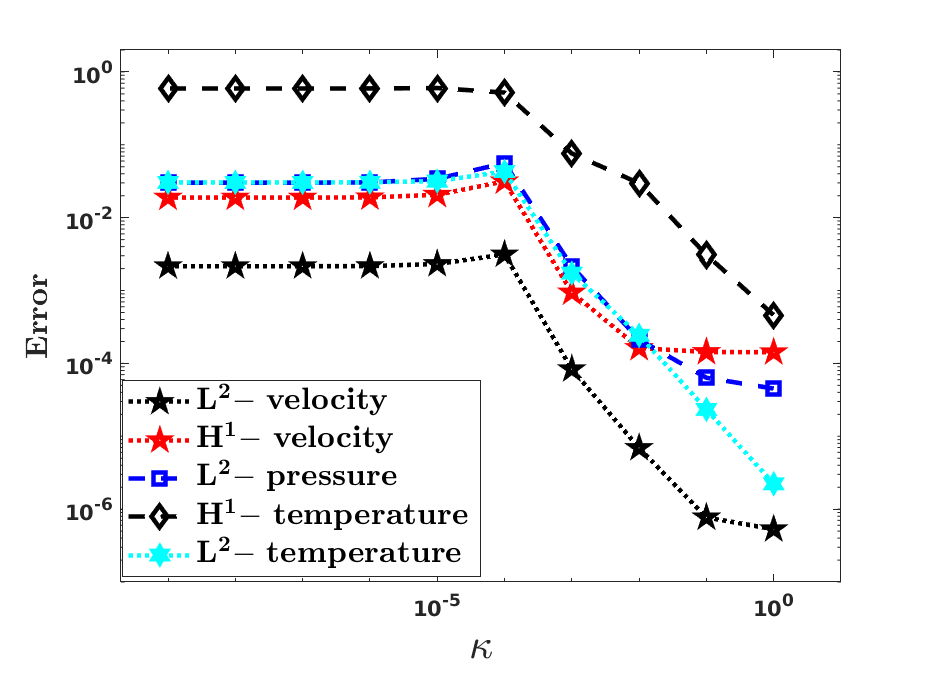}}
	\caption{Example 3. Variation of VEM error norms with varying $\kappa$ from $1$ to $10^{-9}$ with mesh size $h=1/20$ for $\Omega_2$.}
	\label{ex3_effect} 
\end{figure}

\subsection{Example 4}
\label{case4}
In this last example, we focus on investigating the efficiency and advantages of the proposed method for a physical benchmark problem discussed in \cite{mfem41} (see Example 6.6). In this case, we consider the following coupled problem on domain $\Omega=(0,4) \times (0,2) \backslash [2,4] \times [0,1]$, given as follows:
\begin{empheq}[left= \empheqlbrace]{align}
	- \mu \Delta \mathbf{u} + \nabla p &=\mathbf{0} \quad  \text{in} \,\, \Omega,\nonumber \\
	\nabla \cdot \mathbf{u} &= 0 \quad  \text{in} \,\, \Omega, \label{ex4}\\
	-\kappa \Delta \phi + \mathbf{u} \cdot \nabla \phi &= 0  \quad \text{in} \,\, \Omega. \nonumber 
\end{empheq}
The computational domain and the boundary conditions for the velocity and temperature are given in Figure \ref{ex4_a} and Eqn. \eqref{bnd},
\begin{empheq}[left= \empheqlbrace]{align}
	\mathbf{u}&= [0.5y(2-y), 0]^T \quad  \quad \,\,\, \,\,\text{on} \,\, \boldsymbol{\Gamma}_{\textbf{in}},\nonumber \\
	\mathbf{u}&= [ 4(y-1)(2-y), 0]^T \quad \,\, \text{on} \,\, \boldsymbol{\Gamma}_{\textbf{out}},\nonumber \\
	\mathbf{u} &= 0  \qquad  \qquad \qquad \qquad \quad \, \text{on} \,\, \partial \Omega \backslash (\boldsymbol{\Gamma}_{\textbf{in}}\cap \boldsymbol{\Gamma}_{\textbf{out}}), \label{bnd} \\
	\phi &= 1 \qquad \qquad \qquad \qquad \quad \, \text{on} \,\, \boldsymbol{\Gamma}_{\textbf{in}},\nonumber \\
	\kappa \nabla \phi \cdot \mathbf{n}&=0 \qquad \qquad \qquad \qquad \quad \, \text{on} \,\,\partial \Omega \backslash \boldsymbol{\Gamma}_{\textbf{in}},\nonumber 
\end{empheq}
where $\mathbf{n}$ is the outward unit normal. The exact temperature solution is $\phi_{exact}=1$.

\begin{figure}[h]
	\centering
	\begin{tikzpicture}
		\draw[thick] (0,0) -- (2,0) -- (2,1) -- (4,1) -- (4,2) -- (0,2) -- cycle;
		
		\draw[thick, dash pattern=on 1pt off 1pt] plot[domain=0:2, smooth, variable=\y] ({0.5*\y*(2-\y)}, \y); 
		\foreach \y in {0.25,0.5,...,1.8} {
			\draw[thick, ->, >=stealth] ({0}, \y) -- ({0.5*\y*(2-\y)}, \y); 
		}
		
		\draw[thick,dash pattern=on 1pt off 1pt] plot[ domain=1:2, smooth,  variable=\y] ({4+4*(\y-1)*(2-\y)}, \y); 
		\foreach \y in {1.15,1.30,...,2.0} {
			\draw[thick, ->, >=stealth] ({4}, \y) -- ({4+4*(\y-1)*(2-\y)}, \y); 
		}
		\node at (3.25, 1.5) [right] {$\boldsymbol{\Gamma}_{\textbf{out}}$};
		\node at (0, 1) [left] {$\boldsymbol{\Gamma}_{\textbf{in}}$};
		\node[anchor=west] at (6, 1) {\includegraphics[height=3cm, width=6cm]{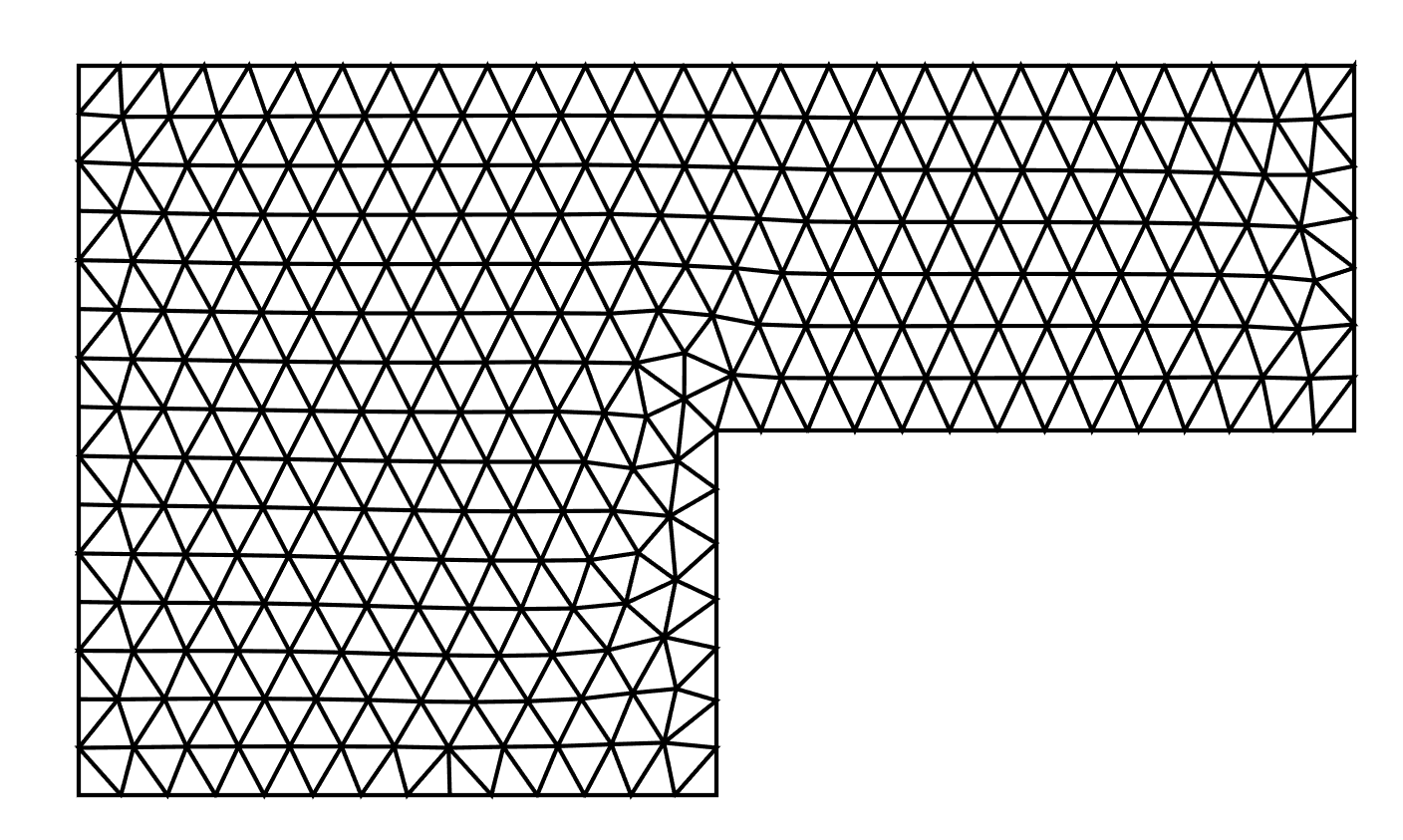}};
	\end{tikzpicture}
	\caption{Example 4. Left: domain with boundary conditions. Right: computational domain with $h=1/4$. }
	\label{ex4_a}
\end{figure}
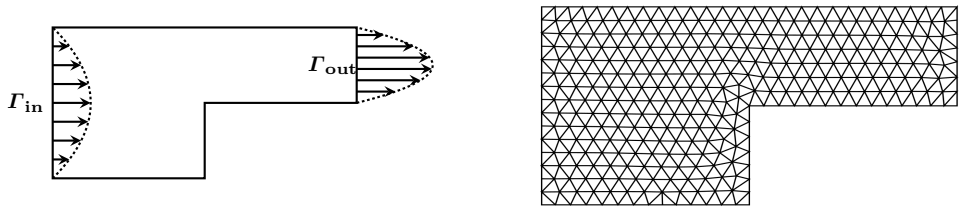

\begin{figure}[h]
	\centering
	{\includegraphics[height=3cm, width=6cm]{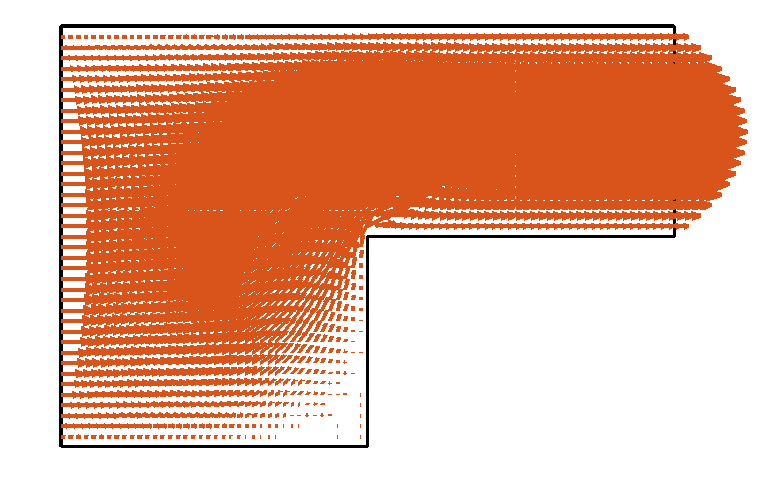}}
	\caption{ Example 4. Velocity vector field with $h=1/64$ for VEM order $k=1$.}
	\label{ex4_vecfield} 
\end{figure}

\begin{figure}[h]
	\centering
	\subfloat[]{\includegraphics[height=3.5cm, width=7cm]{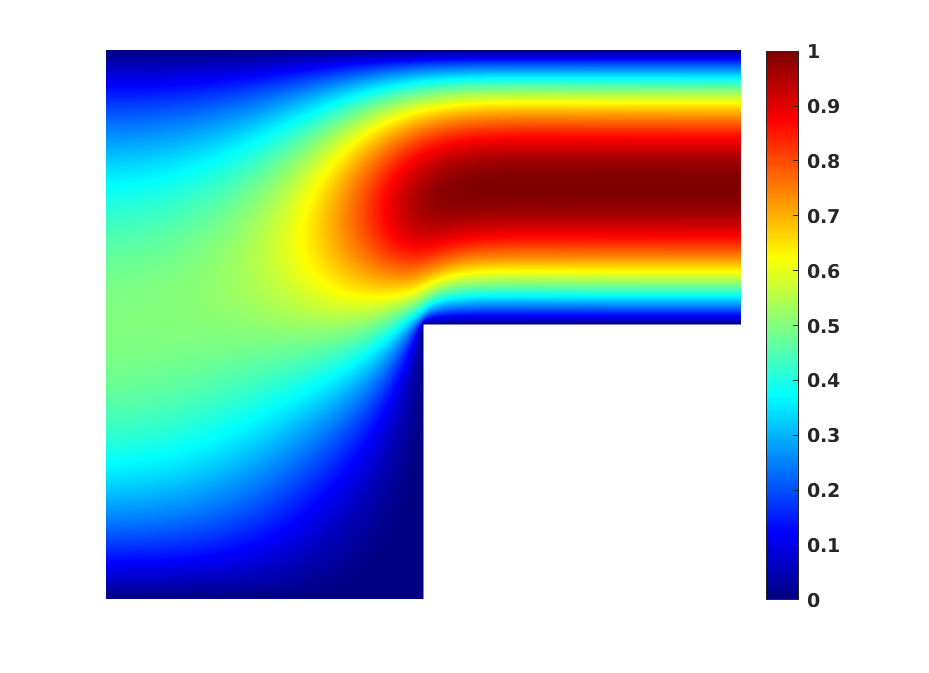}}
	\subfloat[]{\includegraphics[height=4.5cm, width=9cm]{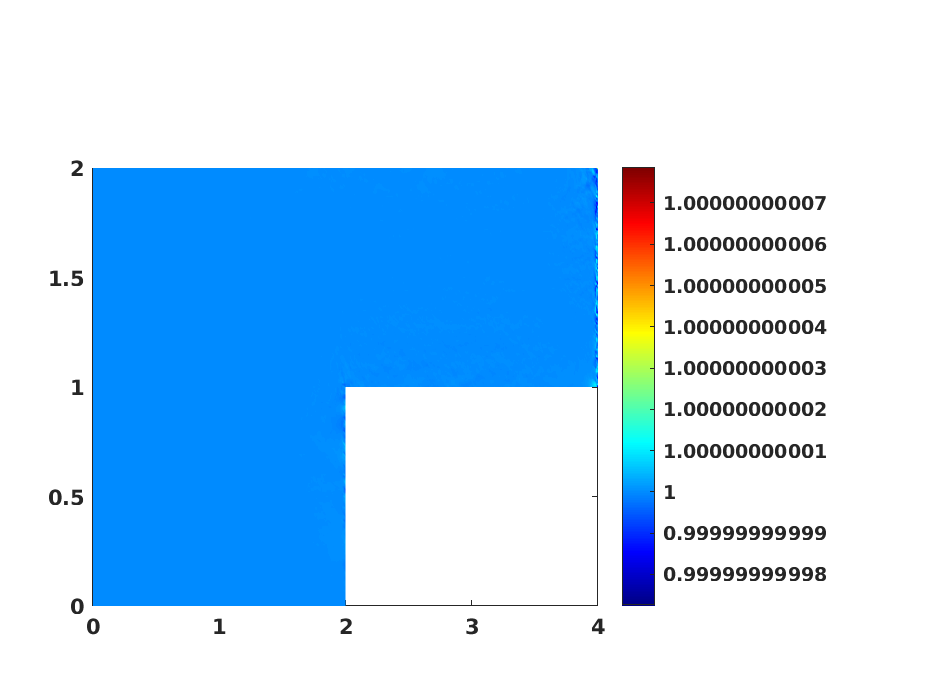}}
	\caption{ Example 4. (Case $\mu=10^{-2}, \kappa= 10^{-6}$.) (a) Stabilized velocity norm  and (b) stabilized temperature with mesh size $h=1/64$ for VEM order $k=1$.}
	\label{ex4_plotsk1} 
\end{figure}

\begin{table}[t!]
	\setlength{\tabcolsep}{15pt}
	\centering 
	\caption{Example 4. Minimal and maximal of $(\phi_h-1)$ for VEM order $k=1$ with $\mu=10^{-2}$ and $\kappa= 10^{-6}$.}
	\label{ex4_tb1}
	\smallskip
	\begin{tabular}{cccccc}
		\toprule
		{$1/h$}&  {4} & 8 & 16 & 32 & 64   \\
		\midrule
		$\min(\phi_h-1)$ & -1.324496e-13 & -4.597434e-13  & -8.915091e-13 & -1.095080e-11 & -3.134626e-11 \\ 
		$\max(\phi_h-1)$ & 2.453593e-13  &  2.999823e-13  & 8.179679e-12  & 1.049960e-11  & 7.848322e-11 \\ 
		\bottomrule		
	\end{tabular}
	\smallskip
\end{table}

\begin{figure}[h]
	\centering
	\subfloat[]{\includegraphics[ height=3.5cm, width=7cm]{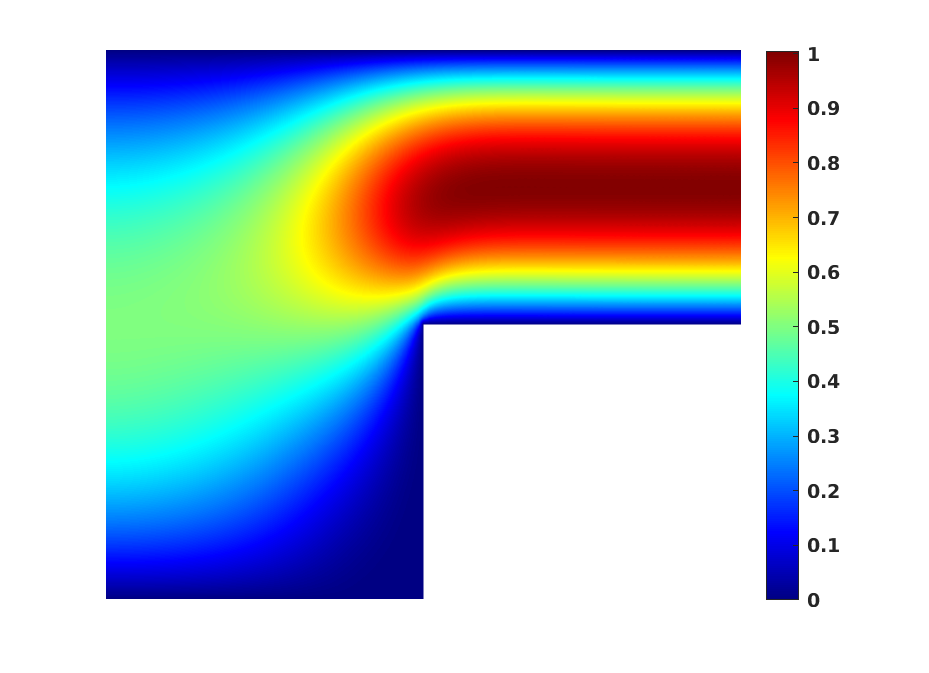}}
	\subfloat[]{\includegraphics[height=4.5cm, width=9cm]{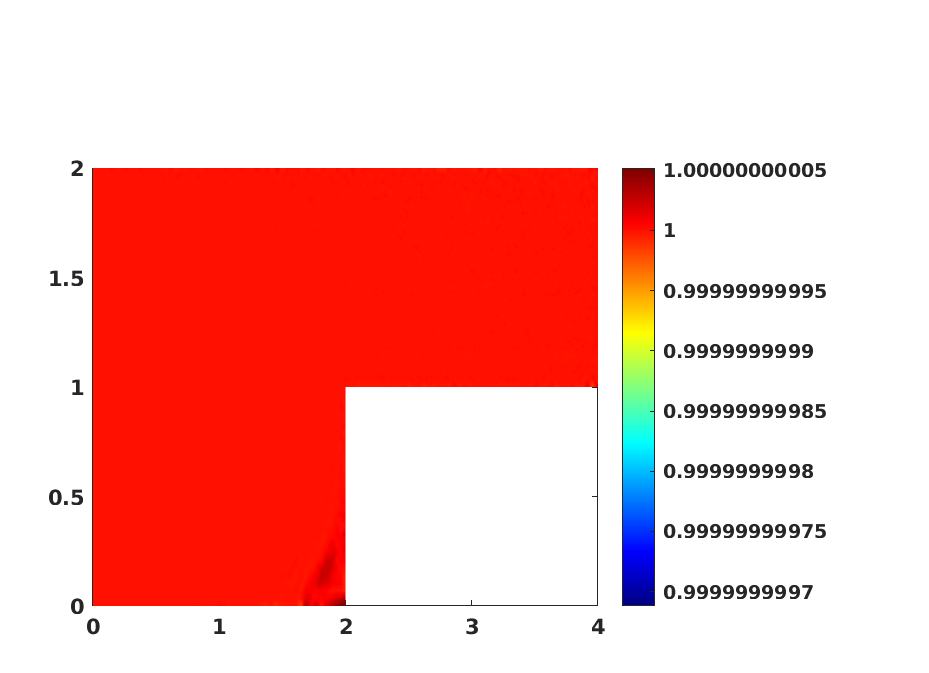}}
	\caption{ Example 4. (Case: $\mu=10^{-2}, \kappa= 10^{-6}$.) (a) Stabilized velocity norm  and (b) stabilized temperature with mesh size $h=1/16$ for VEM order $k=2$.}
	\label{ex4_plotsk2a} 
\end{figure}

\begin{table}[t!]
	\setlength{\tabcolsep}{22pt}
\centering
	\caption{Example 4. Minimal and maximal of $(\phi_h-1)$ for VEM order $k=2$ with $\mu=10^{-2}$ and $\kappa= 10^{-6}$. }
	\label{ex4_tb2}
	\begin{tabular}{ccccc}
		\toprule
		{$1/h$}&  {4} & 8 & 16 & 32    \\
		\midrule
		$\min(\phi_h-1)$ & -3.120919e-10 & -3.157863e-11 & -1.635458e-11 & -7.758871e-11  \\ 
		$\max(\phi_h-1)$ & 1.797118e-11  & 6.335976e-11  & 5.250955e-11  & 8.660028e-11   \\ 
		\bottomrule		
	\end{tabular}
	\smallskip
\end{table}

\begin{figure}[h]
	\centering
	\subfloat[]{\includegraphics[height=3.5cm, width=7cm]{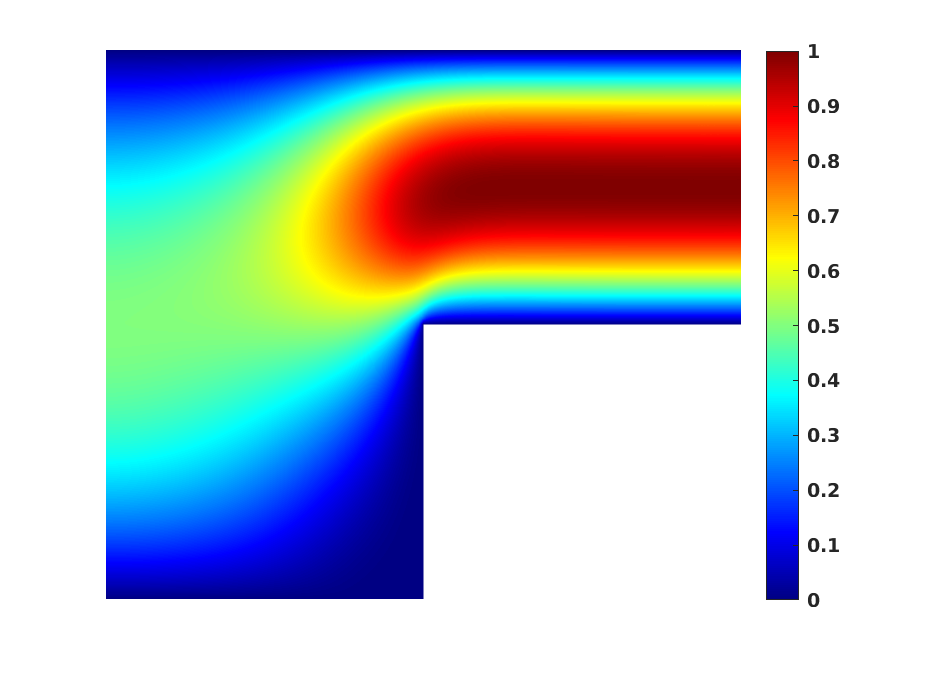}} 
	\subfloat[]{\includegraphics[height=4.35cm, width=7.cm]{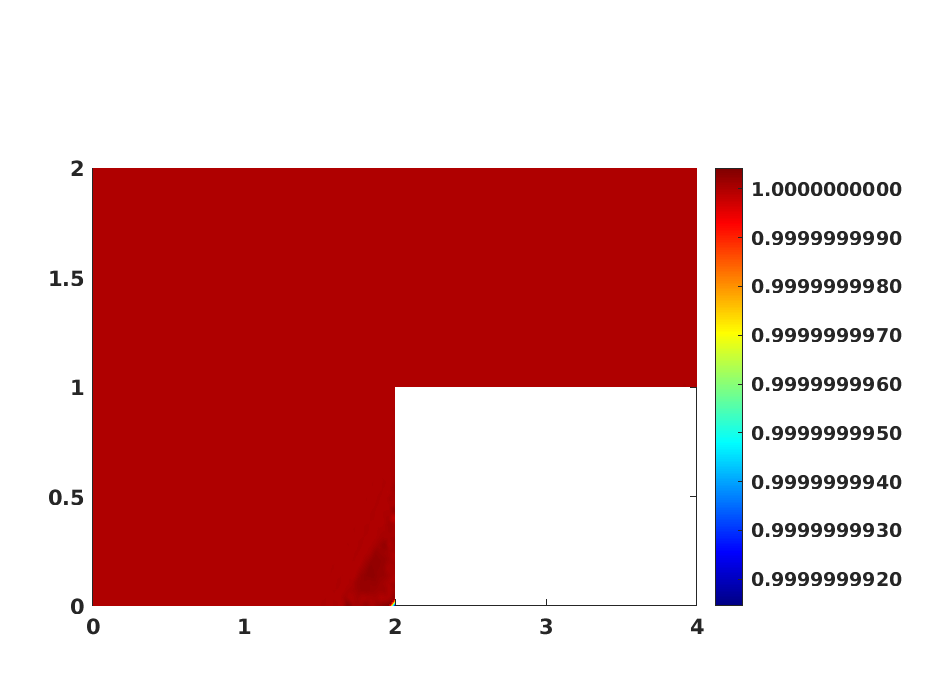}} 
	\caption{ Example 4. (Case: $\mu=10^{-4}, \kappa= 10^{-9}$.) (a) Stabilized velocity norm  and (b) stabilized temperature with mesh size $h=1/16$ for VEM order $k=2$.}
	\label{ex4_plotsk2c} 
\end{figure}

\begin{table}[t!]
	\setlength{\tabcolsep}{22pt}
	\centering
	\caption{Example 4. Minimal and maximal of $(\phi_h-1)$ for VEM order $k=2$ with $\mu=10^{-4}$ and $\kappa= 10^{-9}$. }
	\label{ex4_tb3}
	\begin{tabular}{ccccc}
		\toprule
		{$1/h$}&  {4} & 8 & 16 & 32    \\
		\midrule
		$\min(\phi_h-1)$ & -4.299228e-12  & -9.461535e-08 & -8.547160e-09 & -1.226361e-09  \\ 
		$\max(\phi_h-1)$ & 3.686536e-09  & 7.751533e-11  & 4.422558e-10 & 9.066590e-10  \\ 
		\bottomrule		
	\end{tabular}
	\smallskip
\end{table}

 Following \cite{mfem41,mvem3}, we decompose the domain $\Omega$ using a sequence of triangular meshes with diameter $h=1/4$, $1/8$, $1/16$, $1/32$, and $1/64$. We emphasize that the triangulation of $\Omega$ is generated using the 2D mesh generation methodology described in  \cite{mfem42}.

Our first case focuses on the same parameters used in \cite{mfem41}, where the viscosity $\mu=10^{-2}$ and thermal conductivity $\kappa=10^{-6}$. Notably, the transport equation for temperature in \eqref{ex4} is convection-dominated. The numerical results obtained using the proposed method are presented in Figures \ref{ex4_vecfield} and  \ref{ex4_plotsk1} and Table \ref{ex4_tb1} for VEM order $k=1$. We observe that the stabilized VEM temperature is almost free from non-physical oscillations, except for some very mild oscillations of maximum order $10^{-10}$ near the right walls at $x=2$ and $x=4$. Moreover, these oscillations can be easily removed by increasing the density of elements near the right walls, though this is beyond the scope of the present work.  However, it is worth noting that the discrete velocity space is not divergence-free. Nonetheless, the stabilization term $\mathcal{L}_{1,h}(\cdot,\cdot)$ effectively reduces the violation of the divergence constraints, whereas, in classical mixed finite elements, such violations lead to strong non-physical oscillations in the discrete temperature solution.  The numerical results for VEM order $k=2$ are shown in Figure \ref{ex4_plotsk2a} and Table \ref{ex4_tb2}, which are consistent with the results of \cite{mvem3}.

Additionally, we investigate the performance of the proposed method in the strongly convection-dominated regimes for the transport temperature equation. To do this, we fix $\mu=10^{-4}$ and $\kappa=10^{-9}$. The numerical results obtained using the proposed method are depicted in Figure \ref{ex4_plotsk2c} and Table \ref{ex4_tb3}. Obviously,  the proposed method effectively addresses strongly convection-dominated regimes.

\section{Conclusion} \label{sec:7}
We present a stabilized virtual element method to address the coupled Stokes-Temperature equation on general polygons, employing equal-order virtual element pairs. The stabilization techniques belong to a family of symmetric stabilization methods based on local projections. The proposed method introduces separate stabilization terms, avoiding the inclusion of higher-order derivative terms or bilinear forms involving velocity, pressure and temperature. We derive the stability of the continuous problem using the Banach contraction theorem. Furthermore, the existence and uniqueness of the discrete solution are shown using the Brouwer fixed-point theorem and the contraction theorem. Error estimates are derived in the energy norms, showing optimal convergence rates. Additionally, several numerical examples are presented, confirming the theoretical convergence rates, including both diffusion-dominated and convection-dominated regimes for the transport equation. Notably, the numerical behavior of the proposed method is robust with respect to both linear and nonlinear (temperature-dependent) thermal conductivity.

\subsection*{Data Availability}
Data sharing is not applicable to this article as no datasets were generated or analyzed.

\section*{Declarations}
\subsection*{Conflict of interest} 
The authors declare no conflict of interest during the current study.

		\clearpage
		
		\bibliographystyle{plain}
		\bibliography{references}
		
	\end{document}